\documentclass[11pt]{amsart}
\usepackage[mathscr]{euscript}
\usepackage[utf8]{inputenc}

\usepackage{enumitem} 
\usepackage{tikz} 
\usepackage{tikz-cd} 
\usetikzlibrary{arrows,calc,decorations.markings}
\usepackage{thmtools,thm-restate}
\usepackage{hyperref} 
\usepackage[capitalise,nosort]{cleveref} 
\usepackage{amsmath}
\usepackage{amssymb}
\usepackage{bbm}
\usepackage{mathtools}
\usepackage[margin=1.1in]{geometry}
\usepackage{stmaryrd} 
\usepackage{xparse}
\usepackage{blindtext}
\usepackage{cite}

\makeatletter
\newcommand{\@bbify}[1]{
  \ifcsname b#1\endcsname
  \message{WARNING: Overwriting b#1 with blackboard letter!}
  \fi
  \expandafter\edef\csname b#1\endcsname
  {\noexpand\ensuremath{\noexpand\mathbb #1}\noexpand\xspace}}
\newcommand{\@calify}[1]{
  \ifcsname c#1\endcsname
  \message{WARNING: Overwriting c#1 with calligraphic letter!}
  \fi 
  \expandafter\edef\csname c#1\endcsname
  {\noexpand\ensuremath{\noexpand\mathcal #1}\noexpand\xspace}}
\newcommand{\@bfify}[1]{
  \ifcsname bf#1\endcsname
  \message{WARNING: Overwriting c#1 with bold letter!}
  \fi
  \expandafter\edef\csname bf#1\endcsname
  {\noexpand\ensuremath{\noexpand\mathbf #1}\noexpand\xspace}}
\newcounter{@letter}\stepcounter{@letter}
\loop\@bbify{\Alph{@letter}}\@calify{\Alph{@letter}}\@bfify{\Alph{@letter}}
\ifnum\the@letter<26\stepcounter{@letter}\repeat
\makeatother

\newenvironment{tz}{\begin{center}\begin{tikzpicture}[scale=1]}{\end{tikzpicture}\end{center}}

\tikzstyle{d}=[double distance=.3ex]
\tikzstyle{w}=[preaction={draw=white, -,line width=4pt}]

\newcounter{diagram}
\newenvironment{diagram}{\setcounter{diagram}{\value{thm}}\refstepcounter{thm}\refstepcounter{diagram}\begin{center}\normalfont{(\thediagram)}\hfill\begin{tikzpicture}[baseline=(current bounding box.center)]}
    {\end{tikzpicture}\hfill\text{ }\end{center}}

\tikzset{over/.style={auto=false,fill=white,inner sep=1.5pt, minimum size=0, outer sep=0}, 
pro/.style={postaction={decorate,decoration={
        markings,
        mark=at position .5 with {\node at (0,0) {$\bullet$};}
      }},
      inner sep=1ex,
      },n/.style={double equal sign distance, -implies},t/.style={double distance=2.5pt, -implies, postaction={draw,-}},
  }
\tikzset{%
node distance=1.5cm, la/.style={scale=0.8}, rr/.style={xshift=1.5cm},
space/.style={xshift=.5cm},
    symbol/.style={%
        draw=none,
        every to/.append style={%
            edge node={node [sloped, allow upside down, auto=false]{$#1$}}},
            
    }
}
  
\def\cellslide{0.5}
\def\celllength{.2cm}

\NewDocumentCommand{\cell}{ O{} O{m} O{\cellslide} O{\celllength} m m m }{
  \coordinate (mid) at ($({#5})!{#3}!({#6})$);
  \coordinate (start) at ($(mid)!{#4}!({#5})$);
  \coordinate (end) at ($(mid)!{#4}!({#6})$);
  \draw[#2] (start) to node
  [inner sep=4pt,outer sep=0,minimum size=0,#1]{{#7}} (end);
}

\declaretheorem[name=Theorem,numbered=yes]{theoremA}

\newtheorem{thm}{Theorem}[subsection] 
\newtheorem*{thm*}{Theorem}

\newtheorem{lemma}[thm]{Lemma}
\newtheorem{prop}[thm]{Proposition}

\theoremstyle{definition}
\newtheorem{defn}[thm]{Definition}
\newtheorem{notation}[thm]{Notation}
\newtheorem{constr}[thm]{Construction}

\theoremstyle{remark}
\newtheorem{rmk}[thm]{Remark}

\newtheorem{ex}[thm]{Example}

\crefname{lem}{Lemma}{Lemmas}
\crefname{thm}{Theorem}{Theorems}
\crefname{defn}{Definition}{Definitions}
\crefname{prop}{Proposition}{Propositions}
\crefname{rmk}{Remark}{Remarks}
\crefname{cor}{Corollary}{Corollaries}
\crefname{ex}{Example}{Examples}
\crefname{notation}{Notation}{Notations}
\crefname{constr}{Construction}{Constructions}
\crefname{recall}{Recall}{Recalls}
\crefname{descr}{Description}{Descriptions}
\crefname{para}{\textsection}{\textsection\textsection}

\newlist{rome}{enumerate}{7}
\setlist[rome]{label=(\roman*)}

\newlist{alphabet}{enumerate}{7}
\setlist[alphabet]{label=(\alph*)}

\newcommand{\pushout}[1]{\node at ($({#1})-(10pt,-10pt)$) {$\ulcorner$};}
\newcommand{\pullback}[1]{\node at ($({#1})+(10pt,-10pt)$) {$\lrcorner$};}
\NewDocumentCommand{\punctuation}{ m m O{5pt} }{\node at ($(#1.east)-(0,#3)$) {#2};}

\newcommand{\defThn}{\ensuremath{\theta}}
\newcommand{\defS}{k}

\NewDocumentCommand\Thn{O{n-1}}{\ensuremath{\Theta}_{{#1}}}
\NewDocumentCommand\Thnsset{O{n-1}}{\mathit{s}\set^{\ensuremath{\Theta_{\scriptscriptstyle{#1}}^{\op}}}}
\NewDocumentCommand\Thnssset{O{n-1}}{\mathit{ss}\set^{\ensuremath{\Theta_{\scriptscriptstyle{#1}}^{\op}}}}
\NewDocumentCommand\CSThn{O{n-1}}{{\ensuremath{\scriptscriptstyle(\infty,{#1})}}}
\NewDocumentCommand\AraThn{O{n-1}}{{\ensuremath{\scriptscriptstyle(\infty,{#1})}}}

\newcommand{\Thnssetslice}[1]{\Thnsset_{/{#1}}}
\newcommand{\cat}{\cC\!\mathit{at}}
\newcommand{\set}{\cS\!\mathit{et}}
\newcommand{\sset}{\mathit{s}\set}
\newcommand{\Pcat}{\mathcal{P}\cat}
\newcommand{\DThn}{\Delta\times \Thn}
\newcommand{\DThnS}{\DThn\times \Delta}
\newcommand{\ThnS}{\Thn\times \Delta}
\newcommand{\Dop}{\Delta^{\op}}

\newcommand{\Nec}{\mathcal{N}\mathit{ec}}

\newcommand{\MSspace}{\sset_{\Kan}}
\newcommand{\projsThnsset}{(\Thnsset_{\CSThn})^{\Delta^{\op}}_\proj}

\newcommand{\injThnspace}{(\sset_{\Kan})^{\Thnop}_\inj}
\newcommand{\injsThnspace}{(\sset_{\Kan})^{\ThnSop}_\inj}
\newcommand{\injsThnspaceslice}[1]{(\sset_{\Kan})^{\ThnSop}_{\inj\;\;\;\;\;\; /{#1}}}

\newcommand{\MSThnsset}{\Thnsset_{\CSThn}}
\newcommand{\MSThncat}{\Thnsset_{\CSThn}\text{-}\cat}
\newcommand{\injsThnsset}{(\Thnsset_{\CSThn})^{\Delta^{\op}}_\inj}

\newcommand{\Dset}{\set^{\Dop}}

\newcommand{\diag}{\mathrm{diag}}

\newcommand{\SSset}{\mathit{s}\sset}

\newcommand{\Thnset}{\set^{\Thnop}}
\newcommand{\sThnsset}{\mathit{s}\set^{\Thnop\times \Dop}}
\newcommand{\sThnssetslice}[1]{\sThnsset_{\;\;\;\;\;\;\;\;\;\;/{#1}}}
\newcommand{\sThnssetsliceshort}[1]{\sThnsset_{/{#1}}}

\newcommand{\Thncat}{\Thnsset\text{-}\cat}
\newcommand{\sThncat}{\Thnssset\text{-}\cat}
\newcommand{\Thnop}{{\Theta_{\scriptscriptstyle n-1}^{\op}}}
\newcommand{\DThnop}{\Dop\times \Thnop}

\newcommand{\ThnSop}{\Thnop\times \Dop}

\newcommand{\MSrightfib}[1]{(\sThnssetslice{{#1}})_{\CSThn}^\mathrm{right}}

\newcommand{\pcatThn}{\Pcat(\Thnsset)}
\newcommand{\pcatproj}{\Pcat(\Thnsset_{\CSThn})_\proj}
\newcommand{\pcatinj}{\Pcat(\Thnsset_{\CSThn})_\inj}
\newcommand{\fib}{\mathrm{fib}}

\NewDocumentCommand\Sp{O{m}}{Sp[{#1}]}
\NewDocumentCommand\repD{O{m}}{F[{#1}]}
\NewDocumentCommand\repS{O{\defS}}{\ensuremath{\Delta}[{#1}]}
  
\NewDocumentCommand\repThn{O{\defThn}}%
  {\Thn{[{#1}]}}
\NewDocumentCommand\repDThn{O{m} O{\defThn}}%
  {F[{#1}]\times \Thn{[{#2}]}}
\NewDocumentCommand\repDThnS{O{m} O{\defThn} O{\defS}}%
  {F[{#1}]\times \Thn{[{#2}]}\times \ensuremath{\Delta}[{#3}]}
\NewDocumentCommand\repThnS{O{\defThn} O{\defS}}%
  {\Thn[{#1}]\times \ensuremath{\Delta}[{#2}]}

\newcommand{\Ch}{\mathfrak{C}}
\newcommand{\Nh}{\mathfrak{N}}
\newcommand{\St}{\mathrm{St}}
\newcommand{\Un}{\mathrm{Un}}
\newcommand{\mapDelta}{\mu}

\DeclareMathOperator{\Ob}{Ob}
\DeclareMathOperator{\colim}{colim}

\DeclareMathOperator{\Ho}{Ho}

\DeclareMathOperator{\Hom}{Hom}

\DeclareMathOperator{\op}{op}
\DeclareMathOperator{\ev}{ev}
\DeclareMathOperator{\id}{id}
\DeclareMathOperator{\im}{Im}

\newcommand{\Kan}{\CSThn[0]}
\newcommand{\inj}{\mathrm{inj}}
\newcommand{\proj}{\mathrm{proj}}
\newcommand{\sCat}{{\sset\text{-}\cat}}

\newcommand{\tndnec}[3]{\Nec({#1})_{{#2},{#3}}^{\mathrm{tnd}}}
\newcommand{\catnec}[3]{\Nec({#1})_{{#2},{#3}}}

\newcommand{\NL}{N^h}
\newcommand{\CL}{c^h}

\NewDocumentCommand\Cone{O{f} O{\ensuremath{\mapDelta}}}{{\mathrm{Cone}_{#2}({#1})}}
\NewDocumentCommand\newGbar{O{f} O{i} O{\ensuremath{\mapDelta}}}{{\overline{\cF}^{\,{#2}}_{#3}({#1})}}
\NewDocumentCommand\newG{O{f} O{i} O{\ensuremath{\mapDelta}}}{{\cF^{\,{#2}}_{#3}({#1})}}
\NewDocumentCommand\Hbar{O{i} O{m+1}}{{\overline{H}^{\,{#1}}_{#2}}}

\author{Lyne Moser}
\address{Fakultät für Mathematik, Universität Regensburg, 93040 Regensburg, Germany}
\email{lyne.moser@ur.de}

\author{Nima Rasekh}
\address{Max Planck Institute for Mathematics, Bonn, Germany}
\email{rasekh@mpim-bonn.mpg.de}

\author{Martina Rovelli}
\address{Department of Mathematics and Statistics, University of Massachusetts Amherst, Amherst, USA}
\email{mrovelli@umass.edu}


\title{An $(\infty,n)$-categorical straightening-Unstraightening construction}

\begin{document}

\maketitle

\begin{abstract}
We provide an $(\infty,n)$-categorical version of the straightening-unstraightening construction, asserting an equivalence between the $(\infty,n)$-category of double $(\infty,n-1)$-right fibrations over an $(\infty,n)$-category $\cC$ and that of the $(\infty,n)$-functors from $\cC$ valued in $(\infty,n-1)$-categories. We realize this in the form of a Quillen equivalence between appropriate model structures; on the one hand, a model structure for double $(\infty,n-1)$-right fibrations over a generic precategory object $W$ in $(\infty,n-1)$-categories and, on the other hand, a model structure for $(\infty,n)$-functors from its homotopy coherent categorification $\mathfrak{C} W$ valued in $(\infty,n-1)$-categories.
\end{abstract}

\setcounter{tocdepth}{1}
\tableofcontents

\section*{Introduction}

\subsection*{The rise of fibrations}
Categories have established themselves in a variety of mathematical settings, ranging from geometry and topology to algebra and representation theory, which has served as a motivation for developing a proper theory of categories and in particular the study of the category of sets. Indeed, the fact that every category has hom sets implies that every category embeds in its presheaf category and is hence equivalent to its category of representable functors. This result is known as the Yoneda lemma and the embedding as the Yoneda embedding.

Combining our understanding of the category of sets with the Yoneda embedding has far reaching implications for category theory and many important applications. For example, it enables us to reduce limits of the most complicated diagrams to equalizers and products. It also permits us to formally study geometrically motivated objects, such as schemes, via their category of sheaves, and in fact this can be seen as a key innovation of the Grothendieck school in algebraic geometry. 

In recent decades categories have been generalized to a variety of higher categories with the goal of capturing the relevant structure. On the one side, categories form a $2$-category, keeping track of categories, functors and natural transformations, which, for instance, permits us to effectively discuss adjunctions inside the $2$-category of categories and is hence an effective way to capture purely categorical data. On the other side, topological spaces form an $(\infty,1)$-category with structure given by spaces, continuous maps and weak homotopy equivalences, which permits us to make sense of homotopy invariant constructions internal to the $(\infty,1)$-category and presents us with effective methods to capture purely homotopical data. However, not every mathematical structure exhibits only categorical or homotopical data, but rather combines these two. Examples include $(\infty,1)$-categories themselves, derived stacks on schemes, but also $n$-manifolds studied from the perspective of the theory of bordisms. In all those cases, the objects naturally form an $(\infty,n)$-category which combines the structure of an $n$-dimensional category with appropriately chosen homotopical data.

In a similar way to categories, these $(\infty,n)$-categories have found a variety of applications and particularly in the cases $n=1$ and $n=2$. This provides a motivation for a higher categorical analogue of the Yoneda lemma and the study of presheaves, which, however, has been proven to be far more challenging than one initially anticipated. Indeed, as part of the Yoneda lemma we would like to associate to each object in our $(\infty,n)$-category a representable functor. However, here functoriality is given by composition which is not strict in all relevant models and examples, necessitating an alternative approach.

Fortunately, this issue already arises when working with presheaves valued in groupoids! As they played an important role in modern algebraic geometry, the one major philosophy to deal with such situations was already developed in the sixties: {\it categorical fibrations}. At that time Grothendieck defined what we now call Grothendieck fibrations and the Grothendieck construction, to associate to each groupoid valued functor a fibration and showed that categories of stacks can be described as categories of such fibrations.

In the setting of $(\infty,1)$-categories the Grothendieck construction served as a motivation for Lurie to develop the ``straightening-unstraightening" adjunction, which establishes an equivalence between presheaves valued in $(\infty,1)$-categories (resp.~spaces) and Cartesian (resp.~right) fibrations. This result has then been the basis for a wide range of $(\infty,1)$-categorical results, such as the Yoneda embedding, symmetric monoidal $(\infty,1)$-categories, ... . On the other hand, the situation regarding $(\infty,n)$-categories for $n>1$ has remained far more unclear.

\subsection*{Towards a general straightening for \texorpdfstring{$(\infty,n)$}{(oo,n)}-categories}
As explained above, the straightening construction has been a key input needed to advance the theory of $(\infty,1)$-categories. As such we anticipate a similar important role for a straightening construction for $(\infty,n)$-categories, which establishes an equivalence between presheaves valued in the $(\infty,n)$-category of $(\infty,n-1)$-categories with an appropriately defined notion of fibration over a certain object $W$. As a result, several flavors of such a construction have already been studied in a variety of settings. In each one of those cases the authors have chosen one (or several similar) models of $(\infty,n)$-categories as a basis for a straightening construction.

\begin{enumerate}[leftmargin=0.6cm]
    \item Lurie has generalized his own construction for $(\infty,1)$-categories to $(\infty,2)$-categories using scaled simplicial sets \cite{lurieGoodwillie}. This approach has been further studied by   Gagna--Harpaz--Lanari \cite{gagnaharpazlanari2020inftytwolimits} and by Abell\'{a}n-Garc\'{\i}a--Stern \cite{abellangarciastern20222cartfibi,abellangarciastern20222cartfibii}.
    \item In \cite{RasekhD} the second author studied a straightening construction focused on models of $(\infty,n)$-categories developed by Bergner--Rezk \cite{br2}, that works for all $n$, however, restricts to fibrations where the basis $W$ is the strict nerve of a strictly enriched category.
    \item In \cite{nuiten20222straightening}  Nuiten presents a straightening construction that holds for all $n$ however only applies to fibrations where the basis $W$ is an $n$-fold complete Segal space. 
\end{enumerate}

While these results have greatly contributed to a better understanding of fibrations, they also include some shortcomings, which restrict their applicability. In order to illustrate possible challenges, let us focus here on the case where the basis $W$ is a precategory object in the category $\Thnsset$ of $\Thn$-spaces, which can be endowed with a model structure $\MSThnsset$ that makes it a model of $(\infty,n-1)$-categories \cite{rezkTheta}. The category $\pcatThn$ of such objects is defines as a full subcategory of simplicial objects in $\Thn$-spaces with discrete level $0$ and has been shown to carry a model structure that makes it a model of $(\infty,n)$-categories \cite{br2}. Now, the results in $(1)$ would only apply to the case $n=2$ and even in that case one would have to translate via an intricate web of equivalences of various models of $(\infty,2)$-categories \cite{gagnaharpazlanari2022equivalence}. The approaches $(2)$ and $(3)$ do apply to all~$n$, however, the basis object $W$ does not range over all objects in $\pcatThn$, but instead only over the ones coming from strict nerves of strictly enriched categories or $n$-Segal categories (i.e., $(\infty,n)$-categories). While up to equivalence these include all objects of $\pcatThn$, we now mention a few concrete reasons why we would want a straightening construction for all objects, rather than just some objects which cover all equivalence classes.
\begin{enumerate}[leftmargin=0.6cm]
    \item From a computational perspective, we like to have the ability to define and study colimits of graphs without necessarily forming their corresponding category, as it might significantly complicate constructions. As a simple example, the graph given by a single loop is a finite graph with two non-degenerate simplices. This means the data of a fibration over such a graph has a very simple structure. On the other side, its associated category, known as the free endomorphism category, is in fact infinite.
    \item As $\pcatThn$ is in particular a presheaf category, every $(\infty,n)$-category permits an explicit filtration via a chain of subsimplices, which provides us with one effective tool to prove a variety of results via induction. However, while the starting point and the end result of the filtration are $(\infty,n)$-categories the various steps in between will generally not be and so any inductive argument that involves straightening will need the ability to straighten general objects, rather than just $(\infty,n)$-categories. An example for the effectiveness of this method is the proof of \cite[Theorem 3.2.0.1]{htt}, which relies on the inductive argument given in \cite[Lemma 2.2.3.5]{htt}
    \item As (co)limits in an $(\infty,n)$-category involve an infinite level of coherent interlocking data, we often would like to use strictification methods to compute the (co)limits more effectively. However, this requires using the free contractible homotopy coherent diagram, which internalizes all the required homotopy coherence. For a general diagram it is nearly impossible to compute, however, a straightening construction would give us a direct way of computing the desired free homotopy coherent diagram by applying straightening to the identity fibration, which can be used to explicitly calculate (co)limits. Indeed, this is precisely the method used by Riehl and Verity to give an explicit description of weighted (co)limits which they use to identify (co)limits in quasi-categories  \cite{RiehlVerityNCoh}. Note here that, while they do not use the terminology of straightening explicitly, their description in  \cite[Definition 5.2.8]{RiehlVerityNCoh} precisely coincides with the straightening construction found in \cite[\S 2.2.1]{htt}.
\end{enumerate}

These are just some of the reasons motivating us to construct a straightening construction for all $n$ and for all simplicial objects rather than just simplicial objects that come from strict nerves or $(\infty,n)$-categories. 

\subsection*{A straightening for \texorpdfstring{$(\infty,n)$}{(oo,n)}-categories}
Our work exactly remedies this situation by constructing a straightening functor for all objects in the category $\pcatThn$, which establishes a general equivalence between functors and fibrations. Given an object $W\in \pcatThn$, on one side, we consider the category of $\Thnsset$-enriched functors $\Ch W^{\op}\to \Thnsset$ equipped with the projective model structure $[\Ch W^{\op},\MSThnsset]_\proj$. Here, the $\Thnsset$-enriched category $\Ch W$ denotes the homotopy coherent categorification of $W$, where $\Ch$ is left adjoint of the homotopy coherent nerve $\Nh\colon \Thncat\to \pcatThn$ \cite{MRR1}. On the other side, we consider the category of maps in the larger ambient $\sThnsset$ over $W$ equipped with the model structure $\MSrightfib{W}$ for double $(\infty,n-1)$-right fibrations \cite{RasekhD}.

Our main result now consists of constructing an explicit enriched Quillen equivalence between these two model structures, which appears as \cref{stunmainQE}.

\begin{theoremA}
Let $W$ be an object in $\pcatThn$. There is a Quillen equivalence enriched over $\MSThnsset$ and natural in $W$
\begin{tz}
\node[](1) {$\MSrightfib{W}$}; 
\node[right of=1,xshift=3.9cm](2) {$[\Ch W^{\op},\MSThnsset]_\proj$}; 
\punctuation{2}{.};

\draw[->] ($(2.west)-(0,5pt)$) to node[below,la]{$\Un_{W}$} ($(1.east)-(0,5pt)$);
\draw[->] ($(1.east)+(0,5pt)$) to node[above,la]{$\St_{W}$} ($(2.west)+(0,5pt)$);

\node[la] at ($(1.east)!0.5!(2.west)$) {$\bot$};
\end{tz}
\end{theoremA}

As part of proving this result, we also construct a partial inverse, meaning a left adjoint from functors to fibrations, that also induces an enriched Quillen equivalence. This result appears as \cref{prop:intisQE}. 

\begin{theoremA} 
Let $\cC$ be a fibrant $\MSThnsset$-enriched category. There is a Quillen equivalence enriched over $\MSThnsset$
\begin{tz}
\node[](1) {$[\cC^{\op},\MSThnsset]_\proj$}; 
\node[right of=1,xshift=3.8cm](2) {$\MSrightfib{\Nh\cC}$}; 
\punctuation{2}{.};

\draw[->] ($(2.west)-(0,5pt)$) to node[below,la]{$\cH_\cC^\Nh$} ($(1.east)-(0,5pt)$); 
\draw[->] ($(1.east)+(0,5pt)$) to node[above,la]{$\int_\cC^\Nh$} ($(2.west)+(0,5pt)$);

\node[la] at ($(1.east)!0.5!(2.west)$) {$\bot$};
\end{tz}
\end{theoremA}

Let us record various key aspects and implications of the main result.
\begin{enumerate}[leftmargin=0.6cm]
    \item For $n=0$, the adjunction $\St_W\dashv \Un_W$ is a Quillen equivalence between right fibrations over a general precategory $W\in \Pcat(\sset)$ and $\sset$-enriched functors out of its homotopy coherent categorification. We hence get an alternative version of the original straightening construction due to Lurie \cite{htt}, which uses Segal categories as its model of $(\infty,1)$-categories instead of quasi-categories. We should note that straightening in the context of Segal categories has only been studied over strict nerves of categories \cite{debrito2018leftfibration,rasekh2023left}. Moreover, the Quillen equivalence $\int_\cC^\Nh\dashv \cH_\cC^\Nh$ provides an reversed equivalence for Kan-enriched categories, which has been studied via quasi-categories \cite{heutsmoerdijk2015leftfibrationi,heutsmoerdijk2016leftfibrationii}, whereas the case of Segal categories was again restricted to fibrations over strict nerves \cite{rasekh2023left}.
    \item In \cite{gepnerhaugseng2015enriched} Gepner--Haugseng define a theory of enriched $\infty$-categories and in particular show that $(\infty,n)$-categories, in the sense of Segal $n$-categories, coincide with $\infty$-categories enriched in $(\infty,n-1)$-categories \cite[ Theorem 4.4.7, Remark 5.3.10]{gepnerhaugseng2015enriched}. Hence, as $\MSThnsset$ is a model of $(\infty,n-1)$-categories and both model structures and the adjunction are enriched, it follows from \cite[Theorem 1.1, Corollary 6.18]{haugseng2015rectenrichedinftycat} that the Quillen equivalence $\St_W\dashv \Un_W$ establishes an equivalence of $(\infty,n)$-categories. 
    \item In \cite{riehlverity2022elements} Riehl--Verity introduce $\infty$-cosmoi as a formal $2$-categorical method to study various $(\infty,1)$-categorical (and some $(\infty,n)$-categorical) aspects. In particular they show in \cite[Proposition E.1.1]{riehlverity2022elements} that for an appropriate choice of model category (which by \cite[E.2.6. Proposition]{riehlverity2022elements} applies to $\MSrightfib{W}$), their subcategory of bifibrant objects gives us an $\infty$-cosmos, meaning we get an $\infty$-cosmos of double $(\infty,n-1)$-right fibrations. Hence, our enriched Quillen equivalence induces a biequivalence of $\infty$-cosmoi from the $\infty$-cosmos of functors valued in $(\infty,n-1)$-categories to the $\infty$-cosmos of double $(\infty,n-1)$-right fibrations \cite[Corollary E.1.2]{riehlverity2022elements}.
    
     This in particular means that all $2$-categorical phenomena (such as adjunctions) will uniquely correspond to each other \cite[Proposition 10.3.6]{riehlverity2022elements}. This result is new even for $n=0$, as the straightening construction by Lurie is not an enriched adjunction and hence cannot be cosmological \cite[Lemma 2.1.3.3, Remark 2.2.2.6]{htt}.
    \item While the Quillen equivalence $\int_\cC^\Nh\dashv\cH_\cC^\Nh$ only holds over homotopy coherent nerves of strictly enriched categories, it adds significant computational power to our results. In fact we already used it as an inverse to prove that the adjunction $\St_W\dashv \Un_W$ is a Quillen equivalence. Beyond this example, note that all representable objects in $\pcatThn$ are themselves given as nerves of strictly enriched categories, meaning the two Quillen equivalences give us a detailed characterization of functors and fibrations in that situation.
    \item Note that having a Quillen equivalence $\St_W\dashv\Un_W$ for a general object $W\in \pcatThn$ is the best result we expect to hold and there should not be such a Quillen equivalence for a general object $W\in \sThnsset$. Indeed, any non-trivial structure in $W_0$ would be lost when applying the homotopy coherent categorification $\Ch$, which in particular means that the fibrations would carry additional information beyond enriched functors out of $\Ch W$ valued in $\Thnsset$. See \cite[\S 5.3]{RasekhD} for several counter-examples for various results regarding fibrations over simplicial objects with non-trivial degree $0$.
    \item As we explained above, the projective cofibrant replacement of the terminal functor is a key ingredient for the computation of limits of homotopy coherent diagrams. The adjunction $\St_W\dashv \Un_W$ now gives us an effective method to compute it.
    Indeed, as all objects in $\MSrightfib{W}$ are cofibrant, the derived counit map of the adjunction $\St_W\dashv\Un_W$, which is a projective cofibrant replacement, is given by the strict counit map. Applying this to the terminal object in $[\Ch W^{\op},\Thnsset]$, we can conclude that a cofibrant replacement is given via $\St_W(\id_W)$ (as the identity is the terminal object in $\sThnssetslice{W}$). 
\end{enumerate}

\subsection*{Applying necklace calculus}
As part of our work in \cite{MRR1} we introduced a {\it necklace calculus} that was an integral part in our study of the homs of the homotopy coherent categorifications. Given the computational strength of the method we also anticipated further applications. In this work we precisely present one such application, applying necklace calculus to give concrete computations of the straightening functor. 

Concretely, we want to show that the straightening construction preserves various classes of morphisms in order to establish it is left Quillen. While we can use formal ideas to reduce it to certain generating diagrams (\cref{reductiontrick}), any progress beyond that point requires explicit computations, which are only possible via necklace calculus. Examples of the kind of results we are able to obtain using this method can be found in  \cref{Stofalphafmono} and \cref{Stofpushprod}.

\subsection*{Notations and conventions}

We write:
\begin{itemize}[leftmargin=0.6cm]
\item $\repD\in \set^{\Dop}$ for the representable at $m\geq 0$, and $\partial\repD$ for the boundary of $\repD$,
   \item $F[\defThn,\defS]\coloneqq \repThn\times \repS\in \Thnsset$ for the representable at $(\defThn,[\defS])\in \ThnS$,
    \item $\repD[m,\defThn,\defS]\in \sThnsset$ for the representable at $([m],\defThn,[\defS])\in \DThnS$, 
    \item $\repD[m,X]\in \sThnsset$ for the product $\repD[m]\times X$, and $\partial\repD[m,X]$ for the product $\partial\repD[m]\times X$ for $m\geq 0$ and $X\in \Thnsset$, 
\end{itemize}
The categories $\Dset$ and $\Thnsset$, are naturally included into $\sThnsset$, and we regard the above as objects of it without further specification. We refer to the objects of $\Thnsset$ as \emph{$\Thn$-spaces}.

\subsection*{Acknowledgements}
The third author is grateful for support from the National Science Foundation under Grant No. DMS-2203915.

\section{Strict and homotopy coherent nerves}

We first collect here the necessary background for the content of the paper. In \cref{subsec:MS(inftyn-1)}, we recall the model structure for complete Segal $\Thn$-spaces modeling $(\infty,n-1)$-categories. In \cref{subsec:MSenriched,section:precat}, we recall two models of $(\infty,n)$-categories given by the model structures for strictly enriched categories and weakly enriched categories in complete Segal $\Thn$-spaces. These are related by Quillen equivalences given by the strict and homotopy coherent categorification-nerve adjunctions, which we recollect in \cref{subsec:nerves}. Finally, in \cref{subsec:necklaces}, we recall the notion of necklaces and how they are used to describe the hom $\Thn$-spaces of the homotopy coherent categorification.

\subsection{Model structure for \texorpdfstring{$(\infty,n-1)$}{(infinity,n-1)}-categories} \label{subsec:MS(inftyn-1)}

We recall the model structure $\MSThnsset$ on $\Thnsset$ for $(\infty,n-1)$-categories given by Rezk's complete Segal $\Thn$-spaces \cite{rezkTheta}.

For $n\geq1$, recall from \cite{JoyalDisks} Joyal's cell category $\Thn[n]$. For $n=1$, then $\Thn=\Thn[0]$ is the terminal category, and for $n>1$, the category $\Thn$ is the \emph{wreath product} $\Delta\wr \Thn[n-2]$ (see e.g.~\cite[Definition 3.1]{BergerIterated}).

We denote by $\MSspace$ the Kan-Quillen model structure. The model structure $\MSThnsset$ is defined recursively as a localization of the injective model structure $\injThnspace$ on the category of $\Thn$-presheaves valued in $\MSspace$ with respect to a set $S_{\CSThn}$ of maps in $\Thnset$. 

The set $S_{\CSThn[0]}$ is the empty set, and for $n>1$ the set $S_{\CSThn}$ consists of the following monomorphisms: 
\begin{itemize}[leftmargin=0.6cm]
    \item the \emph{Segal maps} $\Thn{[1]}(\theta_1)\amalg_{[0]} \ldots\amalg_{[0]} \Thn{[1]}(\theta_\ell)\hookrightarrow \Thn{[\ell]}(\theta_1,\ldots,\theta_\ell)$, for all $\ell\geq 1$ and $\theta_1,\ldots,\theta_\ell\in \Thn[n-2]$,
    \item the \emph{completeness map} $\repD[0]\hookrightarrow N \bI$ seen as a map in $\Thnset$ through the canonical inclusion $\Dset\hookrightarrow \Thnset$ induced by pre-composition along the projection $\Thn\to \Delta$ given by $[\ell](\theta_1,\ldots,\theta_l)\mapsto [\ell]$, where $\bI$ denotes the free-living isomorphism, 
    \item the \emph{recursive maps} $\Thn{[1]}(A)\hookrightarrow\Thn{[1]}(B)$, where $A\hookrightarrow B\in \Thnsset[n-2]$ ranges over all maps in $S_{\CSThn[n-2]}$.
\end{itemize}
Note that by \cite[Theorem 8.1]{rezkTheta} the model structure $\MSThnsset$ obtained by localizing the injective model structure $\injThnspace$ with respect to the set $S_{\CSThn}$ is cartesian closed. This is enough to guarantee that the model structure $\MSThnsset$ is excellent in the sense of \cite[Definition A.3.2.16]{htt}.

\subsection{Enriched model structures for \texorpdfstring{$(\infty,n)$}{(infinity,n)}-categories} \label{subsec:MSenriched}

As the model structure $\MSThnsset$ is excellent, the category $\Thncat$ supports the left proper model structure $\MSThncat$ from \cite[\textsection3.10]{br1}, obtained as a special instance of \cite[Proposition A.3.2.4, Theorem A.3.2.24]{htt}. We recall the features of this model structure needed in this paper. They rely on the notion of \emph{homotopy category}, and we refer the reader to \cite[Definition~1.2.1,~Proposition~1.2.3]{MRR1} for a definition in the case of $\MSThnsset$.

In the model structure $\MSThncat$,
\begin{itemize}[leftmargin=0.6cm]
\item a $\Thnsset$-enriched category $\cC$ is \emph{fibrant} if, for all objects $a,b\in \cC$, the hom $\Thn$-space $\Hom_\cC(a,b)$ is fibrant in $\MSThnsset$,
    \item a $\Thnsset$-enriched functor $F\colon \cC\to \cD$ is a \emph{weak equivalence}
    if the induced functor between homotopy categories $\Ho F\colon \Ho\cC\to \Ho\cD$ is essentially surjective on objects, and for all objects $a,b\in \cC$ the induced map
    \[F_{a,b}\colon\Hom_{\cC}(a,b)\to\Hom_\cD(Fa,Fb)\]
    is a weak equivalence in $\MSThnsset$.
\end{itemize} 

Many of the $\Thnsset$-enriched categories that feature in this paper have the following property, so we introduce a terminology that streamlines the exposition.

\begin{defn} 
A $\Thnsset$-enriched category $\cC$ is \emph{directed} if
\begin{itemize}[leftmargin=0.6cm]
    \item its set of objects $\Ob\cC$ is $\{0,1,\ldots,m\}$, for some $m\geq 0$, 
    \item for $0\leq j\leq i\leq m$, the hom $\Thn$-space $\Hom_\cC(i,j)$ is given by
    \[\Hom_\cC(i,j)=\begin{cases}
\emptyset & \text{if} \; j<i \\
\repS[0] & \text{if} \; j=i.
\end{cases}\] 
    \end{itemize}
 In particular, the composition maps in a directed $\Thnsset$-enriched category $\cC$ involving the above hom $\Thn$-spaces are uniquely determined. Moreover, the value of a $\Thnsset$-enriched functor from a directed $\Thnsset$-enriched category is also uniquely determined on these hom $\Thn$-spaces. Similarly, it is enough to verify the naturality conditions for a $\Thnsset$-enriched natural transformations between such $\Thnsset$-enriched functors with respect to these hom $\Thn$-spaces.
\end{defn}

\subsection{Weakly enriched model structures for \texorpdfstring{$(\infty,n)$}{(infinity,n)}-categories}
\label{section:precat}

Let $\pcatThn$ denote the full subcategory of $\sThnsset$ spanned by those $(\DThn)$-spaces $W$ such that $W_0$ is discrete, i.e., such that $W_0$ is in the image of $\set\hookrightarrow \Thnsset$. The canonical inclusion $I\colon \pcatThn\to \sThnsset$ admits a left adjoint $L$, so there is an adjunction
\begin{tz}
\node[](2) {$\sThnsset$}; 
\node[right of=2,xshift=2.7cm](3) {$\pcatThn$}; 
\punctuation{3}{.};

\draw[->] ($(3.west)-(0,5pt)$) to node[below,la]{$I$} ($(2.east)-(0,5pt)$);
\draw[->] ($(2.east)+(0,5pt)$) to node[above,la]{$L$} ($(3.west)+(0,5pt)$);

\node[la] at ($(2.east)!0.5!(3.west)$) {$\bot$};
\end{tz}

In \cite{br1}, Bergner--Rezk construct two model structures on the category $\pcatThn$: the ``projective-like'' and the ``injective-like'' model structures. Here, we denote these two model structures by $\pcatproj$ and $\pcatinj$, respectively. As shown in \cite[Proposition 7.1]{br1}, these model structures are Quillen equivalent via the identity functor.

Let $\injsThnsset$ and $\projsThnsset$ denote the injective and projective model structures on the category $\sThnsset$ of simplicial objects in $\MSThnsset$.

An object $W$ is \emph{fibrant} in $\pcatinj$ (resp.~$\pcatproj$) if $W$ is fibrant in $\injsThnsset$ (resp.~$\projsThnsset$) and the Segal map 
\[ W_m\to W_1\times^{(h)}_{W_0}\ldots \times^{(h)}_{W_0} W_1 \]
is a weak equivalence in $\MSThnsset$, for all $m\geq 1$. Here, the ordinary pullbacks are homotopy pullbacks because they are taken over the discrete object $W_0$ (see \cite[\textsection 4.1]{br1}).

\subsection{Strict and homotopy coherent nerves} \label{subsec:nerves}

There is a canonical inclusion 
\[ N\colon \Thncat\to \pcatThn, \]
which admits a left adjoint $c\colon \pcatThn\to \Thncat$. We refer to $N$ as the \emph{strict nerve} and to $c$ as the \emph{strict categorification}. The following appears as \cite[Theorem 7.6]{br1}.

\begin{prop} \label{projQuillen}
The adjunction $c\dashv N$ is a Quillen equivalence
\begin{tz}
\node[](1) {$\pcatproj$}; 
\node[right of=1,xshift=3.2cm](2) {$\MSThncat$};
\punctuation{2}{.};

\draw[->] ($(2.west)-(0,5pt)$) to node[below,la]{$N$} ($(1.east)-(0,5pt)$);
\draw[->] ($(1.east)+(0,5pt)$) to node[above,la]{$c$} ($(2.west)+(0,5pt)$);

\node[la] at ($(1.east)!0.5!(2.west)$) {$\bot$};
\end{tz} 
\end{prop}

In \cite[\textsection2]{MRR1}, we construct a homotopy coherent version of this categorification--nerve adjunction, and we briefly recall it here.

First, the homotopy coherent categorification--nerve adjunction by Cordier--Porter \cite{CordierPorterHomotopyCoherent}, as depicted below left, induces by post-composition an adjunction as below right.
\begin{tz}
\node[](1) {$\Dset$}; 
\node[right of=1,xshift=1.5cm](2) {$\sCat$}; 
\punctuation{2}{};

\draw[->] ($(2.west)-(0,5pt)$) to node[below,la]{$\NL$} ($(1.east)-(0,5pt)$);
\draw[->] ($(1.east)+(0,5pt)$) to node[above,la]{$\CL$} ($(2.west)+(0,5pt)$);

\node[la] at ($(1.east)!0.5!(2.west)$) {$\bot$};

\node[right of=2,xshift=3cm](1) {$(\Dset)^{\DThnop}$}; 
\node[right of=1,xshift=3.3cm](2) {$(\sCat)^{\DThnop}$}; 

\draw[->] ($(2.west)-(0,5pt)$) to node[below,la]{$\NL_*$} ($(1.east)-(0,5pt)$);
\draw[->] ($(1.east)+(0,5pt)$) to node[above,la]{$\CL_*$} ($(2.west)+(0,5pt)$);

\node[la] at ($(1.east)!0.5!(2.west)$) {$\bot$};
\end{tz}
The adjunction $\CL_*\dashv \NL_*$ in turn restricts to an adjunction between full subcategories
\begin{tz}
\node[](1) {$\pcatThn$}; 
\node[right of=1,xshift=2.7cm](2) {$\sThncat$}; 
    \punctuation{2}{.};

\draw[->] ($(2.west)-(0,5pt)$) to node[below,la]{$\NL_*$} ($(1.east)-(0,5pt)$);
\draw[->] ($(1.east)+(0,5pt)$) to node[above,la]{$\CL_*$} ($(2.west)+(0,5pt)$);

\node[la] at ($(1.east)!0.5!(2.west)$) {$\bot$};
\end{tz}
Next, the diagonal functor $\delta\colon \Delta\to \Delta\times \Delta$ induces an adjunction as below left, which induces by base-change an adjunction between categories of enriched categories as below right.
\begin{tz}
\node[](1) {$\SSset^{\Thnop}$}; 
\node[right of=1,xshift=1.8cm](2) {$\sset^{\Thnop}$}; 

\draw[->] ($(2.west)-(0,5pt)$) to node[below,la]{$(\delta_*)_*$} ($(1.east)-(0,5pt)$);
\draw[->] ($(1.east)+(0,5pt)$) to node[above,la]{$\diag$} ($(2.west)+(0,5pt)$);

\node[la] at ($(1.east)!0.5!(2.west)$) {$\bot$};

\node[right of=2,xshift=3cm](1) {$\sThncat$}; 
\node[right of=1,xshift=2.5cm](2) {$\Thncat$}; 

\draw[->] ($(2.west)-(0,5pt)$) to node[below,la]{$((\delta_*)_*)_*$} ($(1.east)-(0,5pt)$);
\draw[->] ($(1.east)+(0,5pt)$) to node[above,la]{$\diag_*$} ($(2.west)+(0,5pt)$);

\node[la] at ($(1.east)!0.5!(2.west)$) {$\bot$};
\end{tz}

The homotopy coherent categorification-nerve adjunction is then defined to be the following composite of adjunctions.
\begin{tz}
\node[](1) {$\pcatThn$}; 
\node[right of=1,xshift=2.7cm](2) {$\sThncat$}; 
\node[right of=2,xshift=2.5cm](3) {$\Thncat$}; 

\node at ($(1.west)-(6pt,.3pt)$) {$\Ch\colon$};
\node at ($(3.east)+(7pt,-1.5pt)$) {$\colon\! \Nh$};

\draw[->] ($(3.west)-(0,5pt)$) to node[below,la]{$((\delta_*)_*)_*$} ($(2.east)-(0,5pt)$);
\draw[->] ($(2.east)+(0,5pt)$) to node[above,la]{$\diag_*$} ($(3.west)+(0,5pt)$);

\node[la] at ($(1.east)!0.5!(2.west)$) {$\bot$};

\draw[->] ($(2.west)-(0,5pt)$) to node[below,la]{$\NL_*$} ($(1.east)-(0,5pt)$);
\draw[->] ($(1.east)+(0,5pt)$) to node[above,la]{$\CL_*$} ($(2.west)+(0,5pt)$);

\node[la] at ($(2.east)!0.5!(3.west)$) {$\bot$};
\end{tz} 

The following appears as \cite[Theorem~4.3.3]{MRR1}.

\begin{prop} \label{injQuillen}
The adjunction $\Ch\dashv \Nh$ is a Quillen equivalence
\begin{tz}
\node[](1) {$\pcatinj$}; 
\node[right of=1,xshift=3.1cm](2) {$\MSThncat$};
\punctuation{2}{.};

\draw[->] ($(2.west)-(0,5pt)$) to node[below,la]{$\Nh$} ($(1.east)-(0,5pt)$);
\draw[->] ($(1.east)+(0,5pt)$) to node[above,la]{$\Ch$} ($(2.west)+(0,5pt)$);

\node[la] at ($(1.east)!0.5!(2.west)$) {$\bot$};
\end{tz}
\end{prop}

We also show in \cite[Proposition~4.3.2]{MRR1} the following result, which compares the strict and the homotopy coherent nerves. In particular, this shows that the homotopy coherent nerve is an injective fibrant replacement of the strict nerve.

\begin{prop} \label{NsimeqNh}
Let $\cC$ be a fibrant $\MSThnsset$-enriched category. The canonical map
\[\varphi\colon N\cC\to \Nh\cC\]
is a weak equivalence in $\injsThnspace$. 
\end{prop}

\subsection{Necklaces and homotopy coherent categorification} \label{subsec:necklaces}

We recollect here some useful results from \cite[\textsection2]{MRR1} studying the hom $\Thn$-spaces of the categorification $\Ch$. For this, we first recall the main terminology about necklaces, introduced in \cite[\textsection 3]{Dugger}. 

A \emph{necklace} is a simplicial set, i.e., an object in $\Dset$, given by a wedge of representables 
\[ T=\repD[m_1]\vee \ldots \vee \repD[m_t] \] 
obtained by gluing $m_i\in \repD[m_i]$ to $0\in\repD[m_{i+1}]$ for all $1\leq i\leq t-1$. By convention, if $t>1$, then $m_i>0$ for all $1\leq i\leq t$. We say that $\repD[m_i]$ is a \emph{bead} of $T$, and an initial or final vertex in some bead is a \emph{joint} of $T$. We write $B(T)$ for the set of beads of $T$. 

We consider the necklace $T$ to be a bi-pointed simplicial set $(T,\alpha,\omega)$ where $\alpha$ is the initial vertex $\alpha=0\in \repD[m_0]\hookrightarrow T$ and $\omega$ is the final vertex $\omega=m_t\in \repD[m_t]\hookrightarrow  T$.
We write $\Nec$ for the full subcategory of the category $\Dset_{*,*}$ of bi-pointed simplicial sets spanned by the necklaces.

 Given a simplicial set $K$ and $a,b\in K_0$, we denote by $K_{a,b}$ the simplicial set bi-pointed at $(a,b)\colon \repD[0]\amalg\repD[0]\to K$. A \emph{necklace} in $K_{a,b}$ is a bi-pointed map $T\to K_{a,b}$ with $T$ a necklace. We denote by $\catnec{K}{a}{b}\coloneqq \Nec_{/K_{a,b}}$ the category of necklaces $T\to K_{a,b}$ in $K$ from $a$ to $b$. 

A necklace $T=\repD[m_1]\vee\ldots\vee \repD[m_t]\to K_{a,b}$ is \emph{totally non-degenerate} if, for all $0\leq i \leq t$, its restriction to the $i$-th bead
\[ \repD[m_i]\hookrightarrow \repD[m_1]\vee\ldots\vee \repD[m_k]=T \to K\]
is a non-degenerate $m_i$-simplex of $K$. We write $\tndnec{K}{a}{b}$ for the full subcategory of $\catnec{K}{a}{b}$ spanned by the totally non-degenerate necklaces. 

Using the language of necklaces, we obtain as \cite[Proposition~2.4.1]{MRR1} the following description of the hom $\Thn$-spaces of the categorification $\Ch$.

\begin{prop} \label{cor:computationhomsC}
 Let $W$ be an object in $\pcatThn$ and $a,b\in W_0$. Then there is a natural isomorphism in $\Thnsset$
 \[ \Hom_{\Ch W}(a,b)\cong \diag( \colim_{T\in \catnec{W_{-,\star,\star}}{a}{b}} \Hom_{\CL T}(\alpha,\omega)), \]
 where $\colim_{T\in \catnec{W_{-,\star,\star}}{a}{b}} \Hom_{\CL T}(\alpha,\omega)\in \Thnssset$ is given at $\defThn\in\Thn$ and $\defS\geq 0$ by the colimit $\colim_{T\in \catnec{W_{-,\defThn,\defS}}{a}{b}} \Hom_{\CL T}(\alpha,\omega)$ in $\sset$.
 \end{prop}

We get a description of the hom $\Thn$-spaces of the categorification~$\Ch$ in terms of totally non-degenerate necklaces, under the assumption of each level being a \emph{$1$-ordered} simplicial set; see \cite[Definition~2.1.3]{MRR1}.

 \begin{ex} \label{examplesof1ordered}
For $m\geq 0$, by \cite[Remark~2.1.6]{MRR1}, the simplicial set $\repD$ is $1$-ordered.
 \end{ex}

The following appears as \cite[Corollary 2.4.2]{MRR1}.

 \begin{prop} \label{cor:computationshomC1ordered}
Let $W$ be an object in $\pcatThn$ and $a,b\in W_0$. Suppose that, for all $\defThn\in \Thn$ and $\defS\geq 0$, the simplicial set $W_{-,\defThn,\defS}$ is $1$-ordered. Then there is a natural isomorphism in $\Thnsset$
\[ \Hom_{\Ch W}(a,b)\cong \diag( \colim_{T\in \tndnec{W_{-,\star,\star}}{a}{b}} \Hom_{\CL T}(\alpha,\omega)) \]
where $\colim_{T\in \tndnec{W_{-,\star,\star}}{a}{b}} \Hom_{\CL T}(\alpha,\omega)\in \Thnssset$ is given at $\defThn\in\Thn$ and $\defS\geq 0$ by the colimit $\colim_{T\in \tndnec{W_{-,\defThn,\defS}}{a}{b}} \Hom_{\CL T}(\alpha,\omega)$ in $\sset$.
 \end{prop}

We now describe the category of totally non-degenerate necklaces of a $1$-ordered simplicial set, and recall the bead functor construction from \cite[Remark~3.2.7]{MRR1}. 

\begin{rmk} \label{tndnec1ordered}
    Let $K$ be a $1$-ordered simplicial set and $a,b\in K_0$. By \cite[Lemma~2.1.9]{MRR1}, a necklace $f\colon T\to K_{a,b}$ is totally non-degenerate if and only if the map $f$ is a monomorphism in $\Dset$. Hence the objects of the category $\tndnec{K}{a}{b}$ are monomorphisms $T\hookrightarrow K_{a,b}$ with $T$ a necklace, and, by the cancellation property of monomorphisms, the morphisms are precisely monomorphisms $U\hookrightarrow T$ of necklaces over $K_{a,b}$. In particular, this category is a poset.
\end{rmk}

\begin{rmk} \label{beadfunctor}
    Given a monomorphism $g\colon U\hookrightarrow T$ between necklaces, there is an induced map $B(g)\colon B(U)\to B(T)$ between the sets of necklaces, which sends a bead $\repD[m_i]$ of $U$ to the unique bead $B(g)(\repD[m_i])$ of $T$ that contains $\repD[m_i]$. 

    Since all morphisms in $\tndnec{K}{a}{b}$ are monomorphisms, the above assignment induces a functor $B\colon \tndnec{K}{a}{b}\to \set$.
\end{rmk}

\section{Grothendieck construction over the homotopy coherent nerve}

In this section, we build a first Quillen equivalence between strictly enriched functors valued in $(\infty,n-1)$-categories and right double $(\infty,n-1)$-fibrations. In \cref{subsec:MSrightfib}, we first recall the model structure for right double $(\infty,n-1)$-fibrations, and in \cref{subsec:MSproj} the projective model structure for enriched functors. In \cref{subsec:Grothendieck}, we build a Grothendieck construction over the homotopy coherent nerve and prove that it gives a Quillen equivalence
\begin{tz}
\node[](1) {$[\cC^{\op},\MSThnsset]_\proj$}; 
\node[right of=1,xshift=3.8cm](2) {$\MSrightfib{\Nh\cC}$}; 

\punctuation{2}{,};
\draw[->] ($(2.west)-(0,5pt)$) to node[below,la]{$\cH_\cC^\Nh$} ($(1.east)-(0,5pt)$);
\draw[->] ($(1.east)+(0,5pt)$) to node[above,la]{$\int_\cC^\Nh$} ($(2.west)+(0,5pt)$);

\node[la] at ($(1.east)!0.5!(2.west)$) {$\bot$};
\end{tz}
for every fibrant $\MSThnsset$-enriched category $\cC$. To do so, we compare with the Grothendieck construction over the strict nerve constructed by the second author in \cite{RasekhD}.

\subsection{Model structure for right double \texorpdfstring{$(\infty,n-1)$}{(infinity,n-1)}-fibrations} \label{subsec:MSrightfib}

Let us fix an object $W$ of $\pcatThn$, seen as an object of $\sThnsset$. In this section, we first recall the model structure $\MSrightfib{W}$ for right double $(\infty,n-1)$-fibration introduced by the second named author in \cite{RasekhD}.

We denote by $\injsThnspace$ the injective model structure on the category $\sThnsset$ of $(\DThn)$-presheaves valued in $\MSspace$. Then recall that by \cite[Theorem 7.6.5]{Hirschhorn} the slice category $\sThnssetslice{W}$ admits a model structure created by the forgetful functor $\sThnssetslice{W}\to \injsThnspace$. We denote this model structure by $\injsThnspaceslice{W}$. 

\begin{rmk} \label{rem:gencofinjslice}
By \cite[Theorem~1.5]{HirschhornOvercategories} and \cite[Theorem 15.6.27]{Hirschhorn}, a set of generating (trivial) cofibrations for the model structure $\injsThnspaceslice{W}$ is given by the set containing the maps
\[ \textstyle \partial\repD[m,Y]\amalg_{\partial\repD[m,X]} \repD[m,X]\xhookrightarrow{\iota_m\widehat{\times} f} \repD[m,Y]\to W \]
for all $m\geq 0$, where $\iota_m\colon \partial\repD\hookrightarrow \repD$ denotes the boundary inclusion, $f\colon X\hookrightarrow Y$ ranges over all maps in a set of generating (trivial) cofibrations in $\injThnspace$, and $\repD[m,Y]\to W$ over all such maps in $\sThnsset$.
\end{rmk}

The model structure $\MSrightfib{W}$ is defined to be the localization of the model structure $\injsThnspaceslice{W}$ with respect to the following monomorphisms: 
\begin{itemize}[leftmargin=0.6cm]
\item the monomorphism
\[ \repD[0,\defThn,0]\xhookrightarrow{[\langle m\rangle,\defThn,0]} \repD[m,\defThn,0]\to W, \]
for all $m\geq 0$ and $\defThn\in \Thn$, where $\langle m\rangle$ is the map induced by the morphism $[0]\to [m]$ of $\Delta$ that picks $m\in [m]$ and $\repD[m,\defThn,0]\to W$ ranges over all such maps in $\sThnsset$,
\item the monomorphism
\[ \repD[0,X]\xhookrightarrow{[\id_{[0]},f]}\repD[0,Y]\to W, \]
where $f\colon X\hookrightarrow Y$ ranges over all maps in $S_{\CSThn}$ (see \cref{subsec:MS(inftyn-1)}), and $\repD[0,Y]\to W$ over all such maps in $\sThnsset$. 
\end{itemize}

To describe the fibrant objects of this model structure, we use the following result to claim that certain homotopy pullback squares can be computed as strict pullbacks.

\begin{lemma} \label{lem:homotopypullback}
 Let $D$ be a discrete $\Thn$-space, i.e., in the image of $\set\hookrightarrow \Thnsset$. Then every pullback square of the following form in $\Thnsset$ is a homotopy pullback in $\injThnspace$. 
\begin{diagram} \label{pullback}
        \node[](1) {$P$}; 
        \node[right of=1,xshift=.5cm](2) {$Y$}; 
        \node[below of=1](3) {$X$}; 
        \node[below of=2](4) {$D$}; 

        \draw[->] (1) to node[above,la]{$f'$} (2); 
        \draw[->] (1) to node[left,la]{$g'$} (3); 
        \draw[->] (2) to node[right,la]{$g$} (4); 
        \draw[->] (3) to node[below,la]{$f$} (4);
        \pullback{1};
    \end{diagram}
\end{lemma}

\begin{proof} 
    Given an element $d\in D$, we denote by $\fib_d X$, $\fib_d Y$, and $\fib_dP$ the fibers in $\Thnsset$ of the maps $X\to D$, $Y\to D$, and $P\to D$ at $d$. Consider the following pullback square in $\Thnsset$
  \begin{tz}
        \node[](1) {$\coprod_{d\in D} \widehat{\fib_d X}\times \fib_d Y$}; 
        \node[right of=1,xshift=2.8cm](2) {$\coprod_{d\in D}\fib_d Y$}; 
        \node[below of=1](3) {$\coprod_{d\in D} \widehat{\fib_d X}$}; 
        \node[below of=2](4) {$\coprod_{d\in D}\{d\}$}; 

        \draw[->] (1) to (2); 
        \draw[->] (1) to (3); 
        \draw[->] (2) to node[right,la]{$g$} (4); 
        \draw[->] (3) to (4);
        \pullback{1};
    \end{tz}
    where $\fib_d X\to \widehat{\fib_d X}$ denotes a fibrant replacement in $\injThnspace$.
    Given that by construction the bottom map is a fibration and the model structure $\injThnspace$ is right proper by \cite[Theorem 15.3.4]{Hirschhorn}, this square is a homotopy pullback square by \cite[Corollary 13.3.8]{Hirschhorn}.
    Using the fact that $D$ is discrete and that $\injThnspace$ is cartesian closed, one can check that the square is weakly equivalent to the original square \eqref{pullback}. Hence, the latter is also a homotopy pullback, as desired.
\end{proof}

By \cite[Theorem 5.15, Notation 5.16]{RasekhD}, the fibrant objects in $\MSrightfib{W}$ admit the following description.

A map $p \colon P\to W$ in $\sThnsset$ is a \emph{right double $(\infty,n-1)$-fibration} if it satisfies the following conditions: 
\begin{itemize}[leftmargin=0.6cm]
\item it is a fibration in $\injsThnspaceslice{W}$, 
\item for all $m\geq 0$, the following square is a (homotopy) pullback square in $\Thnsset$,
\begin{tz}
        \node[](1) {$P_m$}; 
        \node[right of=1,xshift=.6cm](2) {$P_0$}; 
        \node[below of=1](3) {$W_m$}; 
        \node[below of=2](4) {$W_0$}; 

        \draw[->] (1) to node[above,la]{$p_m$} (2); 
        \draw[->] (1) to node[left,la]{$\langle m\rangle^*$} (3); 
        \draw[->] (2) to node[right,la]{$\langle m\rangle^*$} (4); 
        \draw[->] (3) to node[below,la]{$p_0$} (4);
        \pullback{1};
    \end{tz}
\item for all $a\in W_0$, the (homotopy) fiber $\fib_a P$ of $p\colon P_0\to W_0$ at $a$ is fibrant in $\MSThnsset$.
\end{itemize}

The category $\sThnssetslice{W}$ can be made into a tensored and cotensored $\Thnsset$-enriched category, with tensor described as follows. Given an object $p\colon P\to W$ in $\sThnssetslice{W}$ and an object $X\in \Thnsset$, the tensor $p\otimes X$ is the object of $\sThnssetslice{W}$ obtained as the composite
\[ p\otimes X\colon P\times X\xrightarrow{\pi} P\xrightarrow{f} W, \]
where $\pi$ denotes the canonical projection. The following result can be deduced from \cite[Theorem 5.15]{RasekhD}.

\begin{prop} \label{MSrightfibenriched}
The model structure $\MSrightfib{W}$ is enriched over $\MSThnsset$.
\end{prop}

Every map $f\colon W\to Z$ in $\pcatThn$, seen as a map of $\sThnsset$, induces by post-composition a functor $f_!\colon \sThnssetslice{W}\to \sThnssetslice{Z}$ and the latter admits as a right adjoint the functor $f^*\colon \sThnssetslice{Z}\to \sThnssetslice{W}$ obtained by taking pullbacks along $f$. This adjunction has good homotopical properties. 

\begin{prop} \label{postcompisQP}
Let $f\colon W\to Z$ be a map in $\pcatThn$. The adjunction $f_!\dashv f^*$ is a Quillen pair
\begin{tz}
\node[](1) {$\MSrightfib{W}$}; 
\node[right of=1,xshift=3.9cm](2) {$\MSrightfib{Z}$}; 
\punctuation{2}{.};

\draw[->] ($(2.west)-(0,5pt)$) to node[below,la]{$f^*$} ($(1.east)-(0,5pt)$);
\draw[->] ($(1.east)+(0,5pt)$) to node[above,la]{$f_!$} ($(2.west)+(0,5pt)$);

\node[la] at ($(1.east)!0.5!(2.west)$) {$\bot$};
\end{tz}
It is further a Quillen equivalence when $f\colon W\to Z$ is a weak equivalence in $\injsThnspace$ or in $\pcatinj$.
\end{prop}

\begin{proof}
    The fact that $f_!\dashv f^*$ is a Quillen pair is \cite[Theorem 5.38]{RasekhD}, and the fact that it is a Quillen equivalence when $f$ is a weak equivalence in $\injsThnspace$ is the first bullet point of \cite[Theorem 5.38]{RasekhD}. It remains to show that it is a Quillen equivalence when $f\colon W\to Z$ is a weak equivalence in $\pcatinj$.
    
    For this, we first show that every object $W\in \pcatThn$ has \emph{weakly constant objects} in the sense of \cite[Definition 1.66]{RasekhD}.  Using \cite[Lemma 1.77]{RasekhD}, this happens if and only if any fibrant replacement $W\to \widehat{W}$ in the model structure from \cite[Proposition 5.9]{br2} in the case where $\cC=\Thn$, denoted here by $(\MSThnsset)^{\Dop}_{\mathrm{CSeg}}$, is \emph{homotopically constant}, meaning that each map $\widehat{W}_{0,[0]}\to \widehat{W}_{0,\defThn}$ is a weak equivalence in $\MSspace$ for all $\defThn\in \Thn$. 

    We take a fibrant replacement $W\to \widehat{W}$ in the model structure from \cite[Proposition 5.8]{br2} in the case where $\cC=\Thn$, denoted here by $(\MSThnsset)^{\Dop}_{\mathrm{C}_0\mathrm{Seg}}$. We show that $\widehat{W}$ is homotopically constant and that it is a fibrant replacement of $W$ in $(\MSThnsset)^{\Dop}_{\mathrm{CSeg}}$. By running the small object argument to $W$ in order to obtain~$\widehat{W}$, we see that each step preserves the property of being homotopically constant. Hence, as $W_{0,[0]}=W_{0,\defThn}$ by assumption, it follows that $\widehat{W}_{0,[0]}\to \widehat{W}_{0,\defThn}$ is a weak equivalence in $\MSspace$ for all $\defThn\in \Thn$. The fact that $\widehat{W}$ is fibrant in $(\MSThnsset)^{\Dop}_{\mathrm{CSeg}}$ follows from applying \cite[Proposition 5.10]{br2} using that, as $\widehat{W}$ is homotopically constant, it is fibrant in the model structure from \cite[Proposition 5.6]{br2} in the case where $\cC=\Thn$, denoted here by $(\MSThnsset)^{\Dop}_{\mathrm{hcC}_0\mathrm{Seg}}$. 

    This shows that every object of $\pcatThn$ has weakly constant objects. Hence we can now apply the third bullet point of \cite[Theorem 5.38]{RasekhD} to deduce that $f_!\dashv f^*$ is a Quillen equivalence when $f\colon W\to Z$ is a weak equivalence in $\pcatinj$. To see that this applies in our case, note that the model structure $\MSThnsset$ is cartesian closed so that the cartesian mapping space condition is automatic, and that the functor $I\colon \pcatinj\to (\MSThnsset)^{\Dop}_{\mathrm{hcC}_0\mathrm{Seg}}$ preserves and reflects weak equivalences by \cite[Theorem 9.6]{br2} and a generalization of the argument in \cite[Lemma 3.18]{MOR}.
\end{proof}

\subsection{Projective model structure for enriched functors} \label{subsec:MSproj}

The category $\Thnsset$ is cartesian closed and so it can be seen as a $\Thnsset$-enriched category. Let us fix a $\Thnsset$-enriched category $\cC$. We denote by $[\cC^{\op},\Thnsset]$ the category of $\Thnsset$-enriched functors from $\cC^{\op}$ to $\Thnsset$ and $\Thnsset$-enriched natural transformations between them.

By \cite[Theorem 4.4 (ii)]{Moserinj}, the category $[\cC^{\op},\Thnsset]$ of $\Thnsset$-enriched functors supports the projective model structure $[\cC^{\op},\MSThnsset]_\proj$. 

The category $[\cC^{\op},\Thnsset]$ can be made into a tensored and cotensored $\Thnsset$-enriched category, with tensor described as follows. Given a $\Thnsset$-enriched functor $F\colon \cC^{\op}\to \Thnsset$ and an object $X\in \Thnsset$, the tensor $F\otimes X$ is the $\Thnsset$-enriched functor obtained as the composite 
\[ F\otimes X\colon \cC^{\op}\xrightarrow{F} \Thnsset\xrightarrow{(-)\times X} \Thnsset. \]
The following result can be deduced from \cite[Theorem 5.4]{Moserinj}.

\begin{prop} \label{MSprojenriched}
    The model structure $[\cC^{\op},\MSThnsset]_\proj$ is enriched over $\MSThnsset$.
\end{prop}

\begin{rmk} \label{gencofprojective}
By \cite[Remark 3.3.5]{htt}, a set of generating (trivial) cofibrations for the model structure $[\cC^{\op},\MSThnsset]_\proj$ is given by the set containing the $\Thnsset$-enriched natural transformations
\[ \Hom_\cC(-,a)\otimes X\to\Hom_\cC(-,a)\otimes Y \]
for all objects $a\in \cC$, where $X\hookrightarrow Y$ ranges over a set of generating (trivial) cofibrations of $\MSThnsset$. 
\end{rmk}

Every $\Thnsset$-enriched functor $F\colon \cC\to \cD$ in $\Thncat$ induces by pre-composition a functor $F^*\colon [\cD^{\op},\Thnsset]\to [\cC^{\op},\Thnsset]$. The latter admits as a left adjoint the functor $F_!\colon [\cC^{\op},\Thnsset]\to [\cD^{\op},\Thnsset]$ obtained by taking the $\Thnsset$-enriched left Kan extension along $F$. Note that the adjunction $F_!\dashv F^*$ is enriched over $\Thnsset$ by \cite[Theorem 4.50]{Kelly}, i.e., the functor $F_!$ commutes with tensors. This adjunction has good homotopical properties, as we now recall from \cite[Proposition~A.3.3.7(1) and Proposition~A.3.3.8(1)]{htt}.

\begin{prop} \label{LKEisQP}
Let $F\colon \cC\to \cD$ be a $\Thnsset$-enriched functor. The adjunction $F_!\dashv F^*$ is a Quillen pair
\begin{tz}
\node[](1) {$[\cC^{\op},\MSThnsset]_\proj$}; 
\node[right of=1,xshift=3.6cm](2) {$[\cD^{\op},\MSThnsset]_\proj$}; 
\punctuation{2}{.};

\draw[->] ($(2.west)-(0,5pt)$) to node[below,la]{$F^*$} ($(1.east)-(0,5pt)$);
\draw[->] ($(1.east)+(0,5pt)$) to node[above,la]{$F_!$} ($(2.west)+(0,5pt)$);

\node[la] at ($(1.east)!0.5!(2.west)$) {$\bot$};
\end{tz}
It is further a Quillen equivalence when $F\colon \cC\to \cD$ is a weak equivalence in $\MSThncat$.
\end{prop}

\subsection{Grothendieck constructions} \label{subsec:Grothendieck}

In \cite[Definition 3.20]{RasekhD}, the second author constructs a Gro\-then\-dieck construction over the strict nerve. We here give an alternative -- though equivalent by \cite[Remark 3.21]{RasekhD} -- presentation of this construction, alongside a new variant over the homotopy coherent nerve.

Let us fix a $\Thnsset$-enriched category $\cC$. There are Gro\-then\-dieck constructions 
\[ \textstyle \int_\cC^N\colon [\cC^{\op},\Thnsset]\to \sThnssetslice{N\cC} \quad \text{and} \quad \int_\cC^\Nh\colon [\cC^{\op},\Thnsset]\to \sThnssetslice{\Nh\cC}.\]
On objects, they send a $\Thnsset$-enriched functor $F\colon \cC^{\op}\to \Thnsset$ to the maps in $\Thnsset$
\[ \textstyle\pi_F\colon \int_\cC^N F\to N\cC \quad \text{and} \quad \pi_F\colon \int_\cC^\Nh F\to \Nh\cC \] 
given at level $0$ by the canonical projection 
\[ \textstyle (\pi_F)_0\colon (\int_\cC^N F)_0=(\int_\cC^\Nh F)_0\coloneqq \coprod_{a\in \Ob\cC} Fa\to\Ob\cC=(N\cC)_0=(\Nh\cC)_0 \]
and at level $m\geq 1$ by the pullbacks in $\Thnsset$ 
\begin{tz}
        \node[](1) {$(\int_\cC^N F)_m$}; 
        \node[right of=1,xshift=1.5cm](2) {$(N\cC)_m$}; 
        \node[below of=1](3) {$(\int_\cC^N F)_0$}; 
        \node[below of=2](4) {$(N\cC)_0$}; 

        \draw[->] (1) to node[above,la]{$(\pi_F)_m$} (2); 
        \draw[->] (1) to node[left,la]{$\langle m\rangle^*$} (3); 
        \draw[->] (2) to node[right,la]{$\langle m\rangle^*$} (4); 
        \draw[->] (3) to node[below,la]{$(\pi_F)_0$} (4);
        \pullback{1};

        \node[right of=2,xshift=2cm](1) {$(\int_\cC^\Nh F)_m$}; 
        \node[right of=1,xshift=1.5cm](2) {$(\Nh\cC)_m$}; 
        \node[below of=1](3) {$(\int_\cC^\Nh F)_0$}; 
        \node[below of=2](4) {$(\Nh\cC)_0$}; 
        \punctuation{4}{.};

        \draw[->] (1) to node[above,la]{$(\pi_F)_m$} (2); 
        \draw[->] (1) to node[left,la]{$\langle m\rangle^*$} (3); 
        \draw[->] (2) to node[right,la]{$\langle m\rangle^*$} (4); 
        \draw[->] (3) to node[below,la]{$(\pi_F)_0$} (4);
        \pullback{1};
    \end{tz}

    We further provide the simplicial structure of $\int_\cC^\Nh F$. The one for $\int_\cC^N F$ can be constructed in a similar manner, or deduced from the isomorphism from \cite[Remark 3.21]{RasekhD}. 
    
    We first state the following lemma, which is simply an application of the universal property of pullbacks.

    \begin{lemma} \label{lem:simpstruct}
        Let $\mapDelta\colon [\ell]\to [m]$ be a morphism in $\Delta$ such that $\mapDelta(\ell)=m$. Then there is a unique map $\mapDelta^*\colon (\int_\cC^\Nh F)_m\to (\int_\cC^\Nh F)_{\ell}$ in $\Thnsset$ making the following diagram commute. 
        \begin{tz}
        \node[](1) {$(\int_\cC^\Nh F)_{\ell}$}; 
        \node[above left of=1,xshift=-1cm,yshift=.2cm](1') {$(\int_\cC^\Nh F)_m$};
        \node[right of=1,xshift=1.3cm](2) {$(\Nh\cC)_{\ell}$};
        \node[above left of=2,xshift=-1cm,yshift=.2cm](2') {$(\Nh\cC)_m$};
        \node[below of=1](3) {$(\int_\cC^\Nh F)_0$}; 
        \node[below of=2](4) {$(\Nh\cC)_0$}; 

        \draw[dashed,->] (1') to node[right,la,yshift=5pt,pos=0.4]{$\mapDelta^*$} (1);
        \draw[->] (2') to node[right,la,yshift=5pt,pos=0.4]{$\mapDelta^*$} (2);
        \draw[->,bend right] (1') to node[left,la]{$\langle m\rangle^*$} (3);
        \draw[->] (1) to node[above,la]{$(\pi_F)_{\ell}$} (2);
        \draw[->] (1') to node[above,la]{$(\pi_F)_m$} (2'); 
        \draw[->] (1) to node[left,la]{$\langle \ell\rangle^*$} (3); 
        \draw[->] (2) to node[right,la]{$\langle \ell\rangle^*$} (4); 
        \draw[->] (3) to node[below,la]{$(\pi_F)_0$} (4);
        \pullback{1};
        \end{tz}
    \end{lemma}

    \begin{constr}
    For $m\geq 1$, given that the coface maps $d^i\colon [m-1]\to [m]$ for $0\leq i<m$ and the codegeneracy morphisms $s^j\colon [m]\to [m-1]$ for $0\leq j\leq m-1$ in $\Delta$ satisfy the condition in \cref{lem:simpstruct}, we get induced face maps $d_i\colon (\int_\cC^\Nh F)_m\to (\int_\cC^\Nh F)_{m-1}$ for $0\leq i<m$ and degeneracy maps $s_j\colon (\int_\cC^\Nh F)_{m-1}\to (\int_\cC^\Nh F)_m$ for $0\leq j\leq m-1$ compatible with those of $\Nh\cC$ through the projection $\pi_F$.
    
    It remains to define the face maps $d_m\colon (\int_\cC^\Nh F)_m\to (\int_\cC^\Nh F)_{m-1}$ for $m\geq 1$. If $m=1$, define 
    \[ \textstyle d_1\colon (\int_\cC^\Nh F)_1\cong \coprod_{a,b\in \Ob\cC} \Hom_\cC(a,b)\times Fb\xrightarrow{\coprod_{a,b\in \Ob\cC}\ev_{a,b}^F} \coprod_{a\in\Ob\cC} Fa=(\int_\cC^\Nh F)_0 \]
    where $\ev_{a,b}^F\colon \Hom_\cC(a,b)\times Fb\to Fa$ denotes the unique map corresponding under the adjunction $(-)\times Fb\dashv \Hom_{\Thnsset}(Fb,-)$ to the induced map $F_{a,b}\colon \Hom_\cC(a,b)\to \Hom_{\Thnsset}(Fb,Fa)$ on hom $\Thn$-spaces. If $m>1$, define $d_m$ to be the following composite
    \begin{tz}
        \node[](1) {$(\int_\cC^\Nh F)_m\cong (\Nh\cC)_m\times_{(\Nh\cC)_0} (\int_\cC^\Nh F)_0$}; 
        \node[right of=1,xshift=7cm](2) {$(\Nh\cC)_{m-1}\times_{(\Nh\cC)_0} (\Nh\cC)_1\times_{(\Nh\cC)_0} (\int_\cC^\Nh F)_0$}; 
        \node[below of=2,yshift=.6cm](4) {$(\Nh\cC)_{m-1}\times_{(\Nh\cC)_0} (\int_\cC^\Nh F)_1$}; 
        \node[below of=4](5) {$(\Nh\cC)_{m-1}\times_{(\Nh\cC)_0} (\int_\cC^\Nh F)_0$};
        \node[below of=5,yshift=.5cm](6) {$(\int_\cC^\Nh F)_{m-1}$};

        \draw[->] ($(1.south)-(1.5cm,0)$) to node[below,la]{$d_m$} ($(6.west)$);
        \draw[->] (1) to node[above,la]{$\rho^*\times_{(\Nh\cC)_0} (\int_\cC^\Nh F)_0$} (2);
        \node at ($(2)!0.5!(4)$) {\rotatebox{270}{$\cong$}};
        \node at ($(5)!0.5!(6)$) {\rotatebox{270}{$\cong$}};
        \draw[->] (4) to node[right,la]{$(\Nh\cC)_{m-1}\times_{(\Nh\cC)_0} d_1$} (5);
    \end{tz}
    where $\rho\colon \repD[m-1]\amalg_{\repD[0]}\repD[1]\to \repD[m]$ is the map of $\Dset$ which is induced by the morphisms $d^m\colon [m-1]\to [m]$ and $\langle m-1,m\rangle\colon [1]\to [m]$ of $\Delta$. Note that the face map $d_m$ as defined above is compatible with the corresponding face of $\Nh\cC$ through the projection $\pi_F$.

    The simplicial identities then follow from those of $\Nh\cC$ and the unitality of $F$. Hence this defines a simplicial object $\int_\cC^\Nh F\colon \Dop\to \Thnsset$, i.e., an object of $\sThnsset$. 
    \end{constr}

   Now, on morphisms, the functors $\int_\cC^N$ and $\int_\cC^\Nh$ send a $\Thnsset$-enriched natural transformation $\eta\colon F\to G$ to the maps in $\sThnssetslice{\Nh\cC}$
   \[ \textstyle\int_\cC^N\eta\colon \int_\cC^N F\to \int_\cC^N G \quad \text{and} \quad \int_\cC^\Nh\eta\colon \int_\cC^\Nh F\to \int_\cC^\Nh G \] given at level $0$ by the map 
   \[ \textstyle(\int_\cC^N\eta)_0=(\int_\cC^\Nh\eta)_0\coloneqq \coprod_{a\in \Ob\cC} \eta_a\colon \coprod_{a\in \Ob\cC} Fa\to \coprod_{a\in \Ob\cC} Ga \]
   and at level $m\geq 1$ by the unique map determined by the universal property of pullbacks. The compatibility of $\eta$ with the simplicial structure of $\int_\cC^\Nh F$ and $\int_\cC^\Nh G$ follows from the $\Thnsset$-enriched naturality condition of $\eta$.

   By \cite[Lemma 3.30]{RasekhD}, the functor $\int_\cC^N$ admits a right adjoint $\cH_\cC^N$ and the following appears as \cite[Theorem 5.50]{RasekhD}.

\begin{prop} \label{intQENima}
Let $\cC$ be a $\Thnsset$-enriched category. The adjunction $\int_\cC^N\dashv \cH_\cC^N$ is a Quillen equivalence
\begin{tz}
\node[](1) {$[\cC^{\op},\MSThnsset]_\proj$}; 
\node[right of=1,xshift=3.8cm](2) {$\MSrightfib{N\cC}$};
\punctuation{2}{.};

\draw[->] ($(2.west)-(0,5pt)$) to node[below,la]{$\cH_\cC^N$} ($(1.east)-(0,5pt)$);
\draw[->] ($(1.east)+(0,5pt)$) to node[above,la]{$\int_\cC^N$} ($(2.west)+(0,5pt)$);

\node[la] at ($(1.east)!0.5!(2.west)$) {$\bot$};
\end{tz}
\end{prop}

We now want to prove that the Grothendieck construction $\int_\cC^{\Nh}$ is also a Quillen equivalence. First we have the following.
   
\begin{lemma}
There is an adjunction  
\begin{tz}
\node[](1) {$[\cC^{\op},\Thnsset]$}; 
\node[right of=1,xshift=2.6cm](2) {$\sThnssetslice{\Nh\cC}$}; 
\punctuation{2}{.};

\draw[->] ($(2.west)-(0,5pt)$) to node[below,la]{$\cH_\cC^\Nh$} ($(1.east)-(0,5pt)$);
\draw[->] ($(1.east)+(0,5pt)$) to node[above,la]{$\int_\cC^\Nh$} ($(2.west)+(0,5pt)$);

\node[la] at ($(1.east)!0.5!(2.west)$) {$\bot$};
\end{tz}
\end{lemma} 

\begin{proof}
Note that the categories involved are locally presentable and so, by the adjoint functor theorem, it is enough to show that the functor $\int_\cC^\Nh\colon [\cC^{\op},\Thnsset]\to \sThnssetslice{\Nh\cC}$ preserves colimits. This follows from the definition of $\int_\cC^\Nh$ and the fact that coproducts and pulling back along $(\Nh\cC)_m \to (\Nh\cC)_0$ commute with colimits, as $\Thnsset$ is locally cartesian closed.
\end{proof}

We first prove that the above adjunction also forms a Quillen pair. 

\begin{prop} \label{lem:intprestensors} 
Let $F\colon \cC^{\op}\to \Thnsset$ be a $\Thnsset$-enriched functor and consider an object $X\in \Thnsset$. There is a natural isomorphism in $\sThnssetslice{\Nh\cC}$ 
\[ \textstyle \int_\cC^\Nh (F\otimes X)\cong (\int_\cC^\Nh F)\otimes X. \]
In particular, the adjunction $\int_\cC^\Nh\dashv \cH_\cC^\Nh$ is enriched over $\Thnsset$.
\end{prop}

\begin{proof}
This follows directly from the definition of $\int_\cC^\Nh$ and the fact that coproducts and pullbacks in $\Thnsset$ commute with products. 
\end{proof}

\begin{prop} \label{intQP}
Let $\cC$ be an $\Thnsset$-enriched category. The adjunction $\int_\cC^\Nh\dashv \cH_\cC^\Nh$ is a Quillen pair enriched over $\MSThnsset$
\begin{tz}
\node[](1) {$[\cC^{\op},\MSThnsset]_\proj$}; 
\node[right of=1,xshift=3.8cm](2) {$\MSrightfib{\Nh\cC}$}; 
\punctuation{2}{.}; 

\draw[->] ($(2.west)-(0,5pt)$) to node[below,la]{$\cH_\cC^\Nh$} ($(1.east)-(0,5pt)$);
\draw[->] ($(1.east)+(0,5pt)$) to node[above,la]{$\int_\cC^\Nh$} ($(2.west)+(0,5pt)$);

\node[la] at ($(1.east)!0.5!(2.west)$) {$\bot$};
\end{tz}
\end{prop}

\begin{proof}
By \cref{gencofprojective}, a set of generating (trivial) cofibrations for $[\cC^{\op},\MSThnsset]_\proj$ is given by the set containing the $\Thnsset$-enriched natural transformations
\[ \Hom_\cC(-,a)\otimes X\to \Hom_\cC(-,a)\otimes Y \]
for all objects $a\in \cC$, where $X\hookrightarrow Y$ ranges over a set of generating (trivial) cofibrations of $\MSThnsset$. Using \cref{lem:intprestensors}, such a $\Thnsset$-enriched natural transformation is sent by the functor $\int_\cC^\Nh\colon [\cC^{\op},\MSThnsset]_\proj\to \MSrightfib{\Nh\cC}$ to the map in $\sThnssetslice{\Nh\cC}$
\[ \textstyle (\int_\cC^\Nh \Hom_\cC(-,a))\otimes X\hookrightarrow (\int_\cC^\Nh \Hom_\cC(-,a))\otimes Y, \]
which is a (trivial) cofibration in $\MSrightfib{\Nh\cC}$ by \cref{MSrightfibenriched}. This shows that the functor $\int_\cC^\Nh$ is left Quillen. 

The fact that the Quillen pair is enriched over $\MSThnsset$ follows from the $\MSThnsset$-enrichment of both model structures given by \cref{MSrightfibenriched,MSprojenriched} together with \cref{lem:intprestensors} showing that the functor $\int_\cC^\Nh$ is compatible with these enrichments.
\end{proof}

We now show that, in the case where $\cC$ is fibrant in $\MSThncat$, the above Quillen pair is further a Quillen equivalence.

\begin{prop} \label{triangleint}
Let $\cC$ be a fibrant $\MSThnsset$-enriched category. The following diagram of functors commutes up to weak equivalence,
\begin{tz}
\node[](1) {$[\cC^{\op},\MSThnsset]_\proj$};
\node[below of=1,xshift=-2.8cm](2) {$\MSrightfib{N\cC}$}; 
\node[below of=1,xshift=2.8cm](3) {$\MSrightfib{\Nh\cC}$}; 
\draw[->] (1) to node[left,la,yshift=4pt,xshift=-5pt]{$\int_\cC^N$} (2);
\draw[->] (1) to node(a)[right,la,yshift=4pt,xshift=5pt]{$\int_\cC^\Nh$} (3);
\draw[->] (2) to node[below,la]{$\varphi_!$} (3);

\cell[la,above,xshift=-5pt][n][.6]{2}{a}{$\simeq$};
\end{tz}
where $\varphi\colon N\cC\to \Nh\cC$ is the canonical map from \cref{NsimeqNh}.
\end{prop}

\begin{proof}
Let $F\colon \cC^{\op}\to \Thnsset$ be a $\Thnsset$-enriched functor. We first construct a natural map $\eta_F\colon \varphi_! \int_\cC^N F\to \int_\cC^\Nh F$ in $\sThnssetslice{\Nh\cC}$. By definition of the strict Grothendieck constructions, we have the following equality in $\Thnsset$
\[ \textstyle (\int_\cC^N F)_0=(\int_\cC^\Nh F)_0, \]
and we set $(\eta_F)_0$ to be the identity. Given the following diagram of cospans in $\Thnsset$,
\begin{tz}
\node[](1) {$(\int_\cC^N F)_0$}; 
\node[right of=1,xshift=1cm](2) {$(N\cC)_0$};
\node[right of=2,xshift=1cm](3) {$(N\cC)_m$};
\draw[->] (1) to node[above,la]{$(\pi_F)_0$} (2); 
\draw[->] (3) to node[above,la]{$\langle m\rangle^*$} (2); 

\node[below of=1](1') {$(\int_\cC^\Nh F)_0$}; 
\node[below of=2](2') {$(\Nh\cC)_0$};
\node[below of=3](3') {$(\Nh\cC)_m$};
\draw[->] (1') to node[below,la]{$(\pi_F)_0$} (2'); 
\draw[->] (3') to node[below,la]{$\langle m\rangle^*$} (2'); 

\draw[d] (1) to (1');
\draw[d] (2) to (2'); 
\draw[->] (3) to node[left,la]{$\varphi_m$} node[right,la]{$\simeq$} (3');
\end{tz}
there is a unique induced map  $(\int_\cC^N F)_m\to (\int_\cC^\Nh F)_m$ between their pullbacks in $\Thnsset$, and we set $(\eta_F)_m$ to be that map. Note that these maps $(\eta_F)_m$ for $m\geq 0$ assemble into a map $\eta_F\colon \varphi_!\int_\cC^N F\to \int_\cC^\Nh F$ in $\sThnssetslice{\Nh\cC}$. Moreover, this assignment is natural in $F$.  

Now, since all vertical maps in the diagram of cospans above are weak equivalences in $\injThnspace$ by \cref{NsimeqNh} and the pullbacks of the cospans are in particular homotopy pullbacks by \cref{lem:homotopypullback}, then the induced map $(\eta_F)_m\colon (\int_\cC^N F)_m\to (\int_\cC^\Nh F)_m$ is also a weak equivalence in $\injThnspace$. This shows that the map $\eta_F\colon \varphi_!\int_\cC^N F\to \int_\cC^\Nh F$ in $\sThnssetslice{\Nh\cC}$ is a weak equivalence in $\injsThnspaceslice{\Nh\cC}$. Hence, it is in particular a weak equivalence in its localization $\MSrightfib{\Nh\cC}$.
\end{proof}

\begin{thm} \label{prop:intisQE}
Let $\cC$ be a fibrant $\MSThnsset$-enriched category. The Quillen pair  $\int_\cC^\Nh\dashv \cH_\cC^\Nh$ is a Quillen equivalence enriched over $\MSThnsset$
\begin{tz}
\node[](1) {$[\cC^{\op},\MSThnsset]_\proj$}; 
\node[right of=1,xshift=3.8cm](2) {$\MSrightfib{\Nh\cC}$}; 
\punctuation{2}{.};

\draw[->] ($(2.west)-(0,5pt)$) to node[below,la]{$\cH_\cC^\Nh$} ($(1.east)-(0,5pt)$);
\draw[->] ($(1.east)+(0,5pt)$) to node[above,la]{$\int_\cC^\Nh$} ($(2.west)+(0,5pt)$);

\node[la] at ($(1.east)!0.5!(2.west)$) {$\bot$};
\end{tz}
\end{thm}

\begin{proof}
    We have a triangle of left Quillen functors from \cref{LKEisQP,intQENima,intQP}
    \begin{tz}
\node[](1) {$[\cC^{\op},\MSThnsset]_\proj$};
\node[below of=1,xshift=-2.8cm](2) {$\MSrightfib{N\cC}$}; 
\node[below of=1,xshift=2.8cm](3) {$\MSrightfib{\Nh\cC}$}; 
\draw[->] (1) to node[left,la,yshift=4pt,xshift=-5pt]{$\int_\cC^N$} (2);
\draw[->] (1) to node[right,la,yshift=4pt,xshift=5pt]{$\int_\cC^\Nh$} (3);
\draw[->] (2) to node[below,la]{$\varphi_!$} (3);

\cell[la,above,xshift=-5pt][n][.6]{2}{a}{$\simeq$};
\end{tz}
which commutes up to isomorphism at the level of homotopy categories by \cref{triangleint}. Since the map $\varphi\colon N\cC\to \Nh\cC$ is a weak equivalence in $\injsThnspace$ by \cref{NsimeqNh}, the functor $\varphi_!$ is a Quillen equivalence by \cref{postcompisQP}. Moreover, by \cref{intQENima} the functor $\int_\cC^N$ is a Quillen equivalence. Hence, by $2$-out-of-$3$, we conclude that the functor $\int_\cC^\Nh$ is also a Quillen equivalence, as desired.
\end{proof}

\section{Straightening-unstraightening: definition and computations}

In this section, we introduce the straightening-unstraightening adjunction 
\begin{tz}
 \node[](1) {$\sThnssetslice{W}$}; 
\node[right of=1,xshift=2.8cm](2) {$[\Ch W^{\op},\Thnsset]$}; 
\punctuation{2}{,};

\draw[->] ($(2.west)-(0,5pt)$) to node[below,la]{$\Un_W$} ($(1.east)-(0,5pt)$);
\draw[->] ($(1.east)+(0,5pt)$) to node[above,la]{$\St_W$} ($(2.west)+(0,5pt)$);

\node[la] at ($(1.east)!0.5!(2.west)$) {$\bot$};    
\end{tz} 
for every object $W\in \pcatThn$, and study its point-set properties. In \cref{subsec:defStUn}, we construct the above adjunction, and in \cref{subsec:propofSt}, we state some properties and computations of the straightening functor, whose proofs are deferred to the appendix. Then, in \cref{subsec:Stofboundary,subsec:Stoflastvertex}, we compute the straightening of maps of the form \[ \partial\repD[m,Y]\amalg_{\partial\repD[m,X]} \repD[m,X]\hookrightarrow L\repD[m,Y] \quad \text{and} \quad \repD[0,X]\xhookrightarrow{[\langle m\rangle,\id_X]} L\repD[m,X] \]
for $m\geq 0$ and $X\hookrightarrow Y$ a monomorphism in $\Thnsset$.

\subsection{Definition of straightening} \label{subsec:defStUn}

In this section, we construct the straightening-un\-straight\-ening adjunction. 

\begin{rmk} \label{sliceaspresheaf}
Given an object $W$ in $\sThnsset$, we denote by $\int_{\DThnS} W$ the category of elements of $W$. Then there is a canonical equivalence of categories
\[ \sThnssetslice{W}\simeq \set^{(\int_{\DThnS} W)^{\op}}. \]
\end{rmk}

Let us fix an object $W\in \pcatThn$, seen as an object of $\sThnsset$. Let $m,\defS\geq 0$, $\defThn\in \Thn$, and $\sigma\colon \repD[m,\defThn,\defS]\to W$ be a map in~$\sThnsset$.

\begin{notation} \label{pushoutlowersigma}
    Since $W$ is in $\pcatThn$, the map $\sigma$ corresponds under the adjunction $L\dashv I$ from \cref{section:precat} to a map $\sigma\colon L\repD[m,\defThn,\defS]\to W$ in $\pcatThn$. We write~$W_\sigma$ for the following pushout in $\pcatThn$ (and hence in $\sThnsset$).
    \begin{tz}
        \node[](1) {$L\repD[m,\defThn,\defS]$}; 
        \node[right of=1,xshift=1.8cm](2) {$W$}; 
        \node[below of=1](3) {$L\repD[m+1,\defThn,\defS]$}; 
        \node[below of=2](4) {$W_\sigma$}; 

        \draw[->] (1) to node[above,la]{$\sigma$} (2); 
        \draw[->] (1) to node[left,la]{$L [d^{m+1},\defThn,\defS]$} (3); 
        \draw[->] (2) to node[right,la]{$\iota_\sigma$} (4); 
        \draw[->] (3) to node[below,la]{$\sigma'$} (4);
        \pushout{4};
    \end{tz}
\end{notation}

\begin{rmk}
    By applying the colimit-preserving functor $(-)_0\colon \sThnsset\to \Thnsset$ to the above pushout, we get that 
    \[ (W_\sigma)_0\cong W_0\amalg \{\top\}, \]
    where $\top$ is the image under $\sigma'$ of the object $m+1\in L\repD[m+1,\defThn,\defS]_0\cong \{0,1,\ldots,m+1\}$.
\end{rmk}

    We define the $\Thnsset$-enriched functor  $\St_W(\sigma)\colon \Ch W^{\op}\to \Thnsset$ to be the following composite 
    \[ \St_W(\sigma)\colon \Ch W^{\op}\xrightarrow{\Ch (\iota_\sigma)} (\Ch W_\sigma)^{\op}\xrightarrow{\Hom_{\Ch W_\sigma}(-,\top)} \Thnsset. \]
    This construction extends to a functor
    \[ \textstyle\St_W\colon \int_{\DThnS} W\to [\Ch W^{\op}, \Thnsset],\] and by left Kan extending along the Yoneda embedding and using \cref{sliceaspresheaf}, we obtain a \emph{straightening-unstraightening} adjunction 
\begin{tz}
 \node[](1) {$\sThnssetslice{W}$}; 
\node[right of=1,xshift=2.8cm](2) {$[\Ch W^{\op},\Thnsset]$}; 
\punctuation{2}{.};

\draw[->] ($(2.west)-(0,5pt)$) to node[below,la]{$\Un_W$} ($(1.east)-(0,5pt)$);
\draw[->] ($(1.east)+(0,5pt)$) to node[above,la]{$\St_W$} ($(2.west)+(0,5pt)$);

\node[la] at ($(1.east)!0.5!(2.west)$) {$\bot$};    
\end{tz}

\subsection{Naturality, enrichment, and computations of straightening} \label{subsec:propofSt}

We now state some properties and computations of the straightening functor, whose proofs are deferred to the Appendix, as they involve technical results related to necklace calculus. 

We show in \cref{sec:naturalityofSt} that the straightening functor is natural in the following sense. 

\begin{prop} \label{basechangeSt}
Let $f\colon W\to Z$ be a map in $\pcatThn$. Then the following square of left adjoint functors commutes up to a natural isomorphism.
\begin{tz}
        \node[](1) {$\sThnssetslice{W}$}; 
        \node[right of=1,xshift=2.8cm](2) {$[\Ch W^{\op},\Thnsset]$}; 
        \node[below of=1](3) {$\sThnssetslice{Z}$}; 
        \node[below of=2](4) {$[\Ch Z^{\op},\Thnsset]$}; 

        \draw[->] (1) to node[above,la]{$\St_W$} (2); 
        \draw[->] (1) to node[left,la]{$f_!$} (3); 
        \draw[->] (2) to node[right,la]{$(\Ch f)_!$} (4); 
        \draw[->] (3) to node[below,la]{$\St_Z$} (4);
        \cell[la,above][n][.5]{2}{3}{$\cong$};
    \end{tz}
\end{prop}

Moreover, the straightening functor is enriched, as we show in \cref{sec:enrichmentofSt}.

\begin{prop} \label{lem:Stprestensors} 
Let $W$ be an object in $\pcatThn$ and $p\colon P\to W$ be an object in $\sThnssetslice{W}$. Then there is a natural isomorphism in $[\Ch W^{\op},\Thnsset]$ 
\[ \textstyle \St_W (p\otimes X)\cong \St_W(p)\otimes X. \]
In particular, the adjunction $\St_W\dashv \Un_W$ is enriched over $\Thnsset$.
\end{prop}

In \cref{sec:Stofgeneral}, we compute the straightening of maps $[\mapDelta,f]\colon \repD[\ell,X]\to L \repD[m,Y]$ in $\sThnsset$ with $\mapDelta\colon [\ell]\to [m]$ an injective map in $\Delta$ and $f\colon X\to Y$ a map in $\Thnsset$ between \emph{connected} $\Thn$-spaces.

For this, we first recall the following result, computing $L\repD[m,Y]$ as a certain pushout. Here we write $\pi_0\colon \Thnsset\to \set$ for the left adjoint to the inclusion $\set\hookrightarrow \Thnsset$.

\begin{rmk} \label{RmkPushout}
    By \cite[Lemma~3.1.1]{MRR1}, the object $L\repD[m,Y]$ of $\pcatThn$ can be computed as the following pushout in $\sThnsset$.
    \begin{tz}
        \node[](1) {$\coprod_{m+1} Y$}; 
        \node[right of=1,xshift=1.6cm](2) {$\repD[m,Y]$}; 
        \node[below of=1](3) {$\coprod_{m+1} \pi_0 Y$}; 
        \node[below of=2](4) {$L\repD[m,Y]$}; 

        \draw[->] (1) to (2); 
        \draw[->] (1) to (3); 
        \draw[->] (2) to (4); 
        \draw[->] (3) to (4);
        \pushout{4};
    \end{tz}
\end{rmk}

\begin{rmk}
    A map $f\colon A\to \repD[m,Y]$ in $\sThnsset$ induces a map $A\to L\repD[m,Y]$ by post-composing with the canonical map $\repD[m,Y]\to L\repD[m,Y]$ from the above pushout. For simplicity, we also write $f$ for this map, namely $f\colon A\to L\repD[m,Y]$. 
\end{rmk}

Recall from \cite[Definition 3.1.5]{MRR1} that a $\Thn$-space $Y$ is said to be \emph{connected} if the set $\pi_0 Y$ is isomorphic to a point.

\begin{rmk} \label{rem:repareconnected}
    For $\defThn\in \Thn$ and $\defS\geq 0$, the representable $\repD[\defThn,\defS]=\repThn\times \repS$ in $\Thnsset$ is a connected $\Thn$-space.
\end{rmk}

The computation of the straightening of $[\mapDelta,f]$ is given by a generalization of the construction of the straightening at representables. We introduce the following notation; compare to \cref{pushoutlowersigma}.

\begin{notation} \label{notationCone}
We write $\Cone$ for the following pushout in $\pcatThn$ (and hence in $\sThnsset$). 
\begin{tz}
        \node[](1) {$L \repD[\ell,X]$}; 
        \node[right of=1,xshift=2cm](2) {$L \repD[m,Y]$}; 
        \node[below of=1](3) {$L \repD[\ell+1,X]$}; 
        \node[below of=2](4) {$\Cone$}; 

        \draw[->] (1) to node[above,la]{$L [\mapDelta,f]$} (2); 
        \draw[->] (1) to node[left,la]{$L [d^{\ell+1},X]$} (3); 
        \draw[->] (2) to node[right,la]{$\iota_f$} (4); 
        \draw[->] (3) to (4);
        \pushout{4};
    \end{tz}
\end{notation}

\begin{rmk} 
     By \cref{Cone0}, there is an isomorphism in $\Thnsset$
     \[ \Cone_0\cong \{0,1,\ldots,m+1\}. \]
\end{rmk}

\begin{prop} \label{Stof[lX]}
    Let $\mapDelta\colon [\ell]\to [m]$ be an injective map in $\Delta$, and $f\colon X\to Y$ be map in $\Thnsset$ between connected $\Thn$-spaces. The straightening functor $\St_{L\repD[m,Y]}$ sends the object $[\mapDelta,f]\colon \repD[\ell,X]\to L\repD[m,Y]$ in $\sThnssetsliceshort{L\repD[m,Y]}$ to the $\Thnsset$-enriched functor 
    \[ \St_{L\repD[m,Y]} ([\mapDelta,f])\colon \Ch L\repD[m,Y]^{\op}\xrightarrow{\Ch(\iota_{f})} \Ch\Cone^{\op}\xrightarrow{\Hom_{\Ch \Cone}(-,m+1)} \Thnsset. \]
\end{prop}

We then compute in \cref{sec:Stofmono} the straightening of a map $[\mapDelta,f]\colon \repD[\ell,X]\to L \repD[m,Y]$ in $\sThnsset$ in the case where $\mapDelta\colon [\ell]\to [m]$ is an injective map in $\Delta$ and $f\colon X\hookrightarrow Y$ is a \emph{monomorphism} in $\Thnsset$ between connected $\Thn$-spaces.  

\begin{notation} \label{notn:+1}
    We define a functor $(-)+1\colon \Delta\to \Delta$ which sends an object $[m]$ to the object $[m+1]$, and a map $\mapDelta\colon [\ell]\to [m]$ to the map 
    \[ \mapDelta+1\colon [\ell+1]\to [m+1] \]
    with $(\mapDelta+1)(i)=\mapDelta(i)$ for all $0\leq i\leq \ell$ and $(\mapDelta+1)(\ell+1)=m+1$. 
\end{notation}

Recall the bead functor $B\colon \tndnec{\repD[m+1]}{i}{m+1}\to \set$ from \cref{beadfunctor}.

\begin{rmk} \label{rmk:lastbead}
    Given a necklace $T$, let us denote by $B_\omega^T$ its last bead. Then, given a monomorphism $g\colon U\hookrightarrow T$ between necklaces, since the last vertex of $U$ is mapped to the last vertex of $V$, the last bead $B_\omega^U$ of $U$ must be mapped into the last bead $B_\omega^T$ of $T$. In particular, the induced map $B(g)\colon B(U)\to B(T)$ between sets of beads from \cref{beadfunctor} is such that $B(g)(B_\omega^U)=B_\omega^T$.
\end{rmk}

\begin{notation} \label{notn:Falphatnd}
    For $0\leq i\leq m$, we define a functor 
    \[ \newG\colon (\tndnec{\repD[m+1]}{i}{m+1})^{\op}\to \Thnsset. \]
    It sends an object $T\hookrightarrow \repD[m+1]_{i,m+1}$ in $\tndnec{\repD[m+1]}{i}{m+1}$ to the $\Thn$-space
    \[ \begin{cases}
    (\prod_{B(T)\setminus \{B_\omega^{T}\}} Y)\times X & \text{if } B_\omega^{T}\subseteq \im(\mapDelta+1) \\
    \emptyset & \text{else }
    \end{cases} \]
    and a map $g\colon U\hookrightarrow T$ in $\tndnec{\repD[m+1]}{i}{m+1}$ to the map in $\Thnsset$
    \[  \begin{cases}
    (\prod_{B(T)\setminus \{B^{T}_\omega\}} Y)\times X\xrightarrow{\hat{f}B(g)^*} (\prod_{B(U)\setminus \{B^{U}_\omega\}} Y)\times X & \text{if } B^{U}_\omega\subseteq B^{T}_\omega\subseteq \im(\mapDelta+1) \\
    \emptyset\to (\prod_{B(U)\setminus \{B^{U}_\omega\}} Y)\times X & \text{if } B^{U}_\omega\subseteq \im(\mapDelta+1), \, B^{T}_\omega \not\subseteq \im(\mapDelta+1) \\
    \emptyset \to \emptyset & \text{else ,}
    \end{cases}\]
    where $\hat{f} B(g)^*$ is the composite in $\Thnsset$
    \begin{tz}
        \node[](1) {$(\prod_{B(T)\setminus \{B_\omega^{T} \}} Y)\times X$}; 
        \node[right of=1,xshift=5.4cm](2) {$(\prod_{B(U)\setminus B(g)^{-1}\{B_\omega^{T}\}} Y)\times (\prod_{B(g)^{-1} \{B_\omega^{T}\}} X)$}; 
        \node[below of=2](3) {$(\prod_{B(U)\setminus \{B^{U}_\omega\}} Y)\times X$}; 
        \draw[->] (1) to node[above,la]{$B(g)^*$} (2);
        \draw[->] (2) to node[right,la]{$\id_{\prod_{B(U)\setminus B(g)^{-1}\{B_\omega^{T}\}} Y}\times (\prod_{B(g)^{-1} \{B_\omega^{T}\}\setminus \{B_\omega^{U}\}} f)\times \id_X$} (3);
        \draw[->] (1) to node[below,la]{$\hat{f} B(g)^*$} (3);
    \end{tz}

    For $\defThn\in \Thn$ and $\defS\geq 0$, we write $\newG_{\defThn,\defS}$ for the composite 
    \[ \newG_{\defThn,\defS}\colon (\tndnec{\repD[m+1]}{i}{m+1})^{\op}\xrightarrow{\newG}\Thnsset\xrightarrow{(-)_{\defThn,\defS}} \set. \]
\end{notation}

Recall the homotopy coherent nerve functor $c^h\colon \Dset\to \sCat$ from \cref{subsec:nerves}.

\begin{notation} \label{notationH}
    For $0\leq i\leq m$, we define a functor 
    \[ H^i_{m+1}\colon \tndnec{\repD[m+1]}{i}{m+1}\to \sset. \]
    It sends an object $T\hookrightarrow\repD[m+1]_{i,m+1}$ in $\tndnec{\repD[m+1]}{i}{m+1}$ to the space $\Hom_{\CL T}(\alpha,\omega)$ and a map $g\colon U\hookrightarrow T$ in $\tndnec{\repD[m+1]}{i}{m+1}$ to the map in $\sset$
    \[ (\CL g)_{\alpha,\omega}\colon \Hom_{\CL T}(\alpha,\omega)\to \Hom_{\CL U}(\alpha,\omega). \]
\end{notation}

\begin{rmk} \label{rmk:enrichment}
Let $\varphi\colon \sset^{\DThnop}\cong \Thnssset$ be one of the two canonical isomorphism, and consider the inclusion
\[ \iota\colon\Thnsset\cong \set^{\DThnop}\hookrightarrow \sset^{\DThnop} \stackrel{\varphi}{\cong} \Thnssset. \]
Then $\sset^{\DThnop}$ is canonically enriched over $\sset$ and we consider the $\sset$-enrichment on $\Thnssset$ induced via $\varphi\colon \sset^{\DThnop}\cong \Thnssset$.
\end{rmk}

\begin{prop} \label{Stofalphafmono}
    Let $\mapDelta\colon [\ell]\to [m]$ be an injective map in $\Delta$, and $f\colon X\hookrightarrow Y$ be a monomorphism in $\Thnsset$ between connected $\Thn$-spaces. For $0\leq i\leq m$, there is a natural isomorphism in $\Thnsset$
    \[ \St_{L\repD[m,Y]}([\mapDelta,f])(i)\cong \diag(\colim^{H^i_{m+1}}_{(\tndnec{\repD[m+1]}{i}{m+1})^{\op}} \iota\newG). \]
\end{prop}

\subsection{Straightening of \texorpdfstring{$\partial\repD[m,Y]\amalg_{\partial\repD[m,X]} \repD[m,X]\to L\repD[m,Y]$}{bla}} \label{subsec:Stofboundary}

We first compute the straightening of a map $\partial \repD[m,X]\to L\repD[m,Y]$ in $\Thnsset$ where $m\geq 0$ and $f\colon X\hookrightarrow Y$ is a monomorphism in $\Thnsset$ between connected $\Thn$-spaces.

\begin{rmk} \label{rem:boundaryascoeq}
    Recall that the boundary $\partial\repD$ can be computed as the following coequalizer in $\Dset$
    \[ \textstyle \coprod_{0\leq s<t\leq m} \repD[m-2]\rightrightarrows \coprod_{0\leq s\leq m} \repD[m-1]\to \partial\repD. \]
    We denote by $\iota_m\colon \partial\repD\hookrightarrow \repD$ the boundary inclusion. 
    
    Since products in $\sThnsset$ commute with colimits, we get that $\partial\repD[m,X]$ can be computed as the following coequalizer in $\sThnsset$
    \[ \textstyle \coprod_{0\leq s<t\leq m} \repD[m-2,X]\rightrightarrows \coprod_{0\leq s\leq m} \repD[m-1,X]\to \partial\repD[m,X]. \]
\end{rmk}

\begin{notation}
    We define a functor 
    \[ \newG[f][0][\partial,m]\colon (\tndnec{\repD[m+1]}{0}{m+1})^{\op}\to \Thnsset. \]
    It sends an object $T\hookrightarrow \repD[m+1]_{0,m+1}$ in $\tndnec{\repD[m+1]}{0}{m+1}$ to the $\Thn$-space
    \[ \begin{cases}
    (\prod_{B(T)\setminus \{B_\omega^{T}\}} Y)\times X & \text{if } T\neq\repD[m+1] \\
    \emptyset & \text{if } T=\repD[m+1]
    \end{cases} \]
    and a map $g\colon U\hookrightarrow T$ in $\tndnec{\repD[m+1]}{0}{m+1}$ to the map in $\Thnsset$
    \[  \begin{cases}
    (\prod_{B(T)\setminus \{B^{T}_\omega\}} Y)\times X\xrightarrow{\hat{f}B(g)^*} (\prod_{B(U)\setminus \{B^{U}_\omega\}} Y)\times X & \text{if } U,T\neq\repD[m+1]\\
    \emptyset\to (\prod_{B(U)\setminus \{B^{U}_\omega\}} Y)\times X & \text{if } U\neq\repD[m+1],\, T=\repD[m+1] 
    \end{cases}\]
\end{notation}

\begin{lemma} \label{lem:Fcoeq}
    For $0< i\leq m$, there is a coequalizer in $(\Thnsset)^{(\tndnec{\repD[m+1]}{i}{m+1})^{\op}}$
    \[ \textstyle \coprod_{0\leq s<t\leq m} \newG[f][i][d^sd^t]\rightrightarrows \coprod_{0\leq s\leq m} \newG[f][i][d^s]\to \newG[f][i][\id_{[m]}], \]
    and for $i=0$, there is a coequalizer in $(\Thnsset)^{(\tndnec{\repD[m+1]}{i}{m+1})^{\op}}$
    \[ \textstyle \coprod_{0\leq s<t\leq m} \newG[f][0][d^sd^t]\rightrightarrows \coprod_{0\leq s\leq m} \newG[f][0][d^s]\to \newG[f][0][\partial,m]. \]
\end{lemma}

\begin{proof}
    To show that the above diagrams are coequalizers in $(\Thnsset)^{(\tndnec{\repD[m+1]}{i}{m+1})^{\op}}$, it is enough to show that they are coequalizers in $\Thnsset$ when evaluating the functors at each object $T\hookrightarrow \repD[m+1]_{i,m+1}$ in $\tndnec{\repD[m+1]}{i}{m+1}$. This is a straightforward computation.
\end{proof}

We can now compute the straightening of the map $[\iota_m,f]\colon \partial\repD[m,X]\to L\repD[m,Y]$.

\begin{prop}
    Let $m\geq 0$, and $f\colon X\hookrightarrow Y$ be a monomorphism in $\Thnsset$ between connected $\Thn$-spaces. For $0< i\leq m$, there is a natural isomorphism in $\Thnsset$
    \[ \St_{L\repD[m,Y]}([\iota_m,f])(i)\cong \diag(\colim^{H^i_{m+1}}_{(\tndnec{\repD[m+1]}{i}{m+1})^{\op}} \iota\newG[f][i][\id_{[m]}]), \]
    and for $i=0$, there is a natural isomorphism in $\Thnsset$
    \[ \St_{L\repD[m,Y]}([\iota_m,f])(0)\cong \diag(\colim^{H^0_{m+1}}_{(\tndnec{\repD[m+1]}{0}{m+1})^{\op}} \iota\newG[f][0][\partial,m]). \]
\end{prop}

\begin{proof}
For $0<i\leq m$, we write $\newG[f][i][\partial,m]\coloneqq \newG[f][i][\id_{[m]}]$. Since the straightening functor \[ \St_{L\repD[m,Y]}\colon \sThnssetsliceshort{L\repD[m,Y]}\to [\Ch L\repD[m,Y]^{\op},\Thnsset] \]
commutes with colimits, for $0\leq i\leq m$, we obtain using \cref{rem:boundaryascoeq} a coequalizer in $\Thnsset$
\[ \textstyle \coprod_{0\leq s<t\leq m} \St_{L\repD[m,Y]}([d^sd^t,f])(i)\rightrightarrows \coprod_{0\leq s\leq m} \St_{L\repD[m,Y]}([d^s,f])(i)\to \St_{L\repD[m,Y]}([\iota_m,f])(i)\]
Now, by applying the colimit-preserving functor 
\[ \diag(\colim^{H^i_{m+1}}_{(\tndnec{\repD[m+1]}{i}{m+1})^{\op}} \iota(-))\colon (\Thnsset)^{(\tndnec{\repD[m+1]}{i}{m+1})^{\op}} \to \Thnsset \]
to the coequalizer in $(\Thnsset)^{(\tndnec{\repD[m+1]}{i}{m+1})^{\op}}$ from \cref{lem:Fcoeq}
\[ \textstyle \coprod_{0\leq s<t\leq m} \newG[f][i][d^sd^t]\rightrightarrows \coprod_{0\leq s\leq m} \newG[f][i][d^s]\to \newG[f][i][\partial,m], \]
we obtain a coequalizer in $\Thnsset$, which is isomorphic to the one above by
\cref{Stofalphafmono}. Hence, we get an isomorphism in $\Thnsset$
\[ \St_{L\repD[m,Y]}([\iota_m,f])(i)\cong \diag(\colim^{H^i_{m+1}}_{(\tndnec{\repD[m+1]}{0}{m+1})^{\op}} \iota\newG[f][i][\partial,m]). \qedhere \]
\end{proof}

From now on, we assume that $f\colon X\hookrightarrow Y$ is a monomorphism in $\Thnsset$ with only its target~$Y$ connected and we compute the straightening of the pushout-product map 
\[ \iota_m\widehat{\times} f\colon \partial\repD[m,Y]\amalg_{\partial\repD[m,X]} \repD[m,X]\to L\repD[m,Y]. \]

\begin{notation} \label{notation:Gmf}
    We define a functor 
    \[ \cG_m^0(f)\colon (\tndnec{\repD[m+1]}{0}{m+1})^{\op}\to \Thnsset. \]
    It sends an object $T\hookrightarrow \repD[m+1]_{0,m+1}$ in $\tndnec{\repD[m+1]}{0}{m+1}$ to the $\Thn$-space
    \[ \begin{cases}
    \prod_{B(T)} Y & \text{if } T\neq\repD[m+1] \\
    X & \text{if } T=\repD[m+1]
    \end{cases} \]
    and a map $g\colon U\hookrightarrow T$ in $\tndnec{\repD[m+1]}{0}{m+1}$ to the map in $\Thnsset$
    \[  \begin{cases}
    \prod_{B(T)} Y\xrightarrow{B(g)^*} \prod_{B(U)} Y & \text{if } U,T\neq\repD[m+1] \\
    X\xrightarrow{f} Y\xrightarrow{B(g)^*} \prod_{B(U)} Y & \text{if } U\neq\repD[m+1],\, T=\repD[m+1].
    \end{cases}\]
\end{notation}

\begin{rmk}
    Recall from \cite[Remark~4.2.5]{MRR1} that every $\Thn$-space $X$ can be decomposed as a coproduct $X= \coprod_j X_j$, where each $X_j$ is a connected $\Thn$-space.
\end{rmk}

\begin{lemma} \label{lem:pushout}
    Let $X=\coprod_j X_j$ be a decomposition of $X$ into its connected components and write $f=\sum_j f_j\colon \coprod_j X_j\hookrightarrow Y$. There is a pushout in $(\Thnsset)^{(\tndnec{\repD[m+1]}{0}{m+1})^{\op}}$
    \begin{tz}
        \node[](1) {$\coprod_j\newG[f_j][0][\partial,m]$}; 
        \node[right of=1,xshift=2.1cm](2) {$\newG[\id_Y][0][\partial,m]$}; 
        \node[below of=1](3) {$\coprod_j\newG[f_j][0][\id_{[m]}]$}; 
        \node[below of=2](4) {$\cG^0_m(f)$}; 
        \punctuation{4}{.};

        \draw[->] (1) to (2); 
        \draw[->] (1) to (3); 
        \draw[->] (2) to (4); 
        \draw[->] (3) to (4);
        \pushout{4};
    \end{tz}
\end{lemma}

\begin{proof}
    To show that the above square is a pushout in $(\Thnsset)^{(\tndnec{\repD[m+1]}{0}{m+1})^{\op}}$, it is enough to show that it is a pushout in $\Thnsset$ when evaluating at each object $T\hookrightarrow \repD[m+1]_{i,m+1}$ in $\tndnec{\repD[m+1]}{i}{m+1}$. This is a straightforward computation.
\end{proof}

We can now compute the straightening of the pushout-product map 
\[ \iota_m\widehat{\times} f\colon \partial\repD[m,Y]\amalg_{\partial\repD[m,X]} \repD[m,X]\to L\repD[m,Y]. \]

\begin{prop} \label{Stofpushprod}
    Let $m\geq 0$, $Y$ be a connected $\Thn$-space, and $f\colon X\hookrightarrow Y$ be a monomorphism in $\Thnsset$. For $0< i\leq m$, there is a natural isomorphism in $\Thnsset$
    \[ \St_{L\repD[m,Y]}(\iota_m\widehat{\times} f)(i)\cong \diag(\colim^{H^i_{m+1}}_{(\tndnec{\repD[m+1]}{i}{m+1})^{\op}} \iota\newG[\id_Y][i][\id_{[m]}]), \]
    and for $i=0$, there is a natural isomorphism in $\Thnsset$
    \[ \St_{L\repD[m,Y]}(\iota_m\widehat{\times} f)(0)\cong \diag(\colim^{H^0_{m+1}}_{(\tndnec{\repD[m+1]}{0}{m+1})^{\op}} \iota\cG^0_m(f)). \]
\end{prop}

\begin{proof} 
For $0<i\leq m$, we write $\newG[f][i][\partial,m]\coloneqq \newG[f][i][\id_{[m]}]$ and $\cG^i_m(f)\coloneqq \newG[\id_Y][i][\id_{[m]}]$. Let $X=\coprod_j X_j$ be a decomposition of $X$ into its connected components and write $f=\sum_j f_j\colon \coprod_j X_j\hookrightarrow Y$. Since the straightening functor $\St_{L\repD[m,Y]}\colon \sThnssetsliceshort{L\repD[m,Y]}\to [\Ch L\repD[m,Y]^{\op},\Thnsset]$ commutes with colimits, for $0\leq i\leq m$, we have a pushout in $\Thnsset$
\begin{tz}
        \node[](1) {$\coprod_j\St_{L\repD[m,Y]}([\iota_m,f_j])(i)$}; 
        \node[right of=1,xshift=4.1cm](2) {$\St_{L\repD[m,Y]}([\iota_m,\id_Y])(i)$}; 
        \node[below of=1](3) {$\coprod_j\St_{L\repD[m,Y]}([\id_{[m]},f_j])(i)$}; 
        \node[below of=2](4) {$\St_{L\repD[m,Y]}(\iota_m\widehat{\times} f)(i)$}; 

        \draw[->] (1) to (2); 
        \draw[->] (1) to (3); 
        \draw[->] (2) to (4); 
        \draw[->] (3) to (4);
        \pushout{4};
    \end{tz}
Now, by applying the colimit-preserving functor 
\[ \diag(\colim^{H^i_{m+1}}_{(\tndnec{\repD[m+1]}{i}{m+1})^{\op}} \iota(-))\colon (\Thnsset)^{(\tndnec{\repD[m+1]}{i}{m+1})^{\op}} \to \Thnsset \]
to the pushout in $(\Thnsset)^{(\tndnec{\repD[m+1]}{i}{m+1})^{\op}}$ 
\begin{tz}
        \node[](1) {$\coprod_j\newG[f_j][i][\partial,m]$}; 
        \node[right of=1,xshift=2cm](2) {$\newG[\id_Y][i][\partial,m]$}; 
        \node[below of=1](3) {$\coprod_j\newG[f_j][i][\id_{[m]}]$}; 
        \node[below of=2](4) {$\cG^i_m(f)$}; 

        \draw[->] (1) to (2); 
        \draw[->] (1) to (3); 
        \draw[->] (2) to (4); 
        \draw[->] (3) to (4);
        \pushout{4};
    \end{tz}
    obtained from \cref{lem:pushout} if $i=0$ and by direct inspection if $0<i\leq m$, we obtain a pushout in $\Thnsset$, which is isomorphic to the one above by
\cref{Stofalphafmono}. Hence, we get an isomorphism in $\Thnsset$
\[ \St_{L\repD[m,Y]}(\iota_m\widehat{\times} f)(i)\cong \diag(\colim^{H^i_{m+1}}_{(\tndnec{\repD[m+1]}{i}{m+1})^{\op}} \iota\cG^i_m(f)). \qedhere \]
\end{proof}

Later, we will need to compare the straightening of the pushout-product map $\iota_m\widehat{\times} f$ with that of the map $[\id_{[m]},\id_Y]\colon \repD[m,Y]\to L\repD[m,Y]$. The straightening of the latter admits the following description.

\begin{prop} \label{StofidmY}
    Let $m\geq 0$, and $Y$ be a connected $\Thn$-spaces. For $0\leq i\leq m$, there is a natural isomorphism in $\Thnsset$
    \[ \St_{L\repD[m,Y]}([\id_{[m]},\id_Y])(i)\cong \diag(\colim^{H^i_{m+1}}_{(\tndnec{\repD[m+1]}{i}{m+1})^{\op}} \iota\newG[\id_Y][i][\id_{[m]}]). \]
\end{prop}

\begin{proof}
    This is obtained by taking $\mapDelta=\id_{[m]}$ and $f=\id_Y$ in \cref{Stofalphafmono}.
\end{proof}

\begin{rmk} \label{rem:Stofpushprod}
    By \cref{Stofpushprod,StofidmY}, for $0< i\leq m$, there is an isomorphism in $\Thnsset$
    \[ \St_{L\repD[m,Y]}(\iota_m\widehat{\times} f)(i)\cong \St_{L\repD[m,Y]}([\id_{[m]},\id_Y])(i), \]
    and so the two straightenings only differ at $i=0$. Then the component of the $\Thnsset$-enriched natural transformation in $[\Ch L\repD[m,Y]^{\op},\Thnsset]$
    \[ \St_{L\repD[m,Y]}(\iota_m\widehat{\times}f)\colon \St_{L\repD[m,Y]}(\iota_m\widehat{\times}f)\to \St_{L\repD[m,Y]}([\id_{[m]},\id_Y]) \]
    at $0<i\leq m$ is an isomorphism, and at $0$ it is obtained by applying the functor \[ \diag(\colim^{H^i_{m+1}}_{(\tndnec{\repD[m+1]}{i}{m+1})^{\op}} \iota(-))\colon (\Thnsset)^{(\tndnec{\repD[m+1]}{i}{m+1})^{\op}} \to \Thnsset \] 
    to the canonical inclusion  $\cG^0_m(f)\hookrightarrow \newG[\id_Y][0][\id_{[m]}]$.
\end{rmk}

\subsection{Straightening of \texorpdfstring{$\repD[0,X]\to L\repD[m,X]$}{F[0,X]->LF[m,X]}} \label{subsec:Stoflastvertex}

We now compute the straightening of a map $[\langle m\rangle, \id_X]\colon \repD[0,X]\to L\repD[m,X]$ in $\sThnsset$ for $m\geq 0$ and $X$ a connected $\Thn$-space. 

\begin{lemma} \label{ConepickmX}
    There is a natural isomorphism in $\pcatThn$
    \[ \Cone[\id_X][\langle m\rangle]\cong L\repD[m,X]\amalg_{\repD[0]} L\repD[1,X], \]
    where the gluing happens along the vertices $m\in L\repD[m,X]$ and $0\in L\repD[1,X]$.
\end{lemma}

\begin{proof}
    By \cref{RmkPushout}, since $X$ is a connected $\Thn$-space, we get that $L\repD[0,X]=\repD[0]$. The result is then straightforward from the definition of $\Cone$ (see \cref{notationCone}) in the case where $\mapDelta=\langle m\rangle$ and $f=\id_X$.
    \end{proof}

    Recall the suspension functor $\Sigma\colon \Thnsset\to \Thncat$ sending an object $X\in \Thnsset$ to the directed $\Thnsset$-enriched category $\Sigma X$ with object set $\{0,1\}$ and hom $\Thn$-space $\Hom_{\Sigma X}(0,1)=X$. The following appears as \cite[Lemma 3.5.1]{MRR1}. 

    \begin{lemma} \label{Sh1issigma}
        There is a natural isomorphism in $\Thncat$
        \[ \Ch L\repD[1,X]\cong \Sigma X. \]
    \end{lemma}

    \begin{lemma} \label{CConepickmX}
    There is a natural isomorphism in $\Thncat$
    \[ \Ch\Cone[\id_X][\langle m\rangle]\cong \Ch L\repD[m,X]\amalg_{[0]} \Sigma X, \]
    where the gluing happens along the objects $m\in \Ch L\repD[m,X]$ and $0\in \Sigma X$.
\end{lemma}

\begin{proof}
    This is obtained by applying the left adjoint functor $\Ch\colon \pcatThn\to \Thncat$ to the isomorphism from \cref{ConepickmX} and using \cref{Sh1issigma}.
\end{proof}

\begin{prop} \label{StofpickmX}
    Let $m\geq 0$, and $X$ be a connected $\Thn$-space. For $0\leq i\leq m$, there is a natural isomorphism in $\Thnsset$
    \[ \St_{L\repD[m,X]}([\langle m\rangle,\id_X])(i)\cong \Hom_{\Ch L\repD[m,X]\amalg_{[0]} \Sigma X}(i,m+1). \]
\end{prop}

\begin{proof}
    By \cref{Stof[lX],CConepickmX}, we have isomorphisms in $\Thnsset$
    \[ \St_{L\repD[m,X]}([\langle m\rangle,\id_X])(i)\cong \Hom_{\Ch\Cone[\id_X][\langle m\rangle]}(i,m+1)\cong\Hom_{\Ch L\repD[m,X]\amalg_{[0]} \Sigma X}(i,m+1). \qedhere \]
\end{proof}

Later, we will need to compare the straightening of the map $[\langle m\rangle,\id_X]$ with that of the map $[\id_{[m]},\id_X]\colon \repD[m,X]\to L\repD[m,X]$. The straightening of the latter admits the following description.

\begin{lemma} \label{Coneid}
    There is a natural isomorphism in $\pcatThn$
    \[ \Cone[\id_X][\id_{[m]}]\cong L\repD[m+1,X]. \]
\end{lemma}

\begin{proof}
    This is straightforward from the definition of $\Cone$ (\cref{notationCone}) in the case where $\mapDelta=\id_{[m]}$ and $f=\id_X$.
\end{proof}

\begin{prop} \label{StofidmX}
    Let $m\geq 0$, and $X$ be a connected $\Thn$-spaces. For $0\leq i\leq m$, there is a natural isomorphism in $\Thnsset$
    \[ \St_{L\repD[m,Y]}([\id_{[m]},\id_X])(i)\cong \Hom_{\Ch L\repD[m+1,X]}(i,m+1). \]
\end{prop}

\begin{proof}
    By \cref{Stof[lX],Coneid}, we have isomorphisms in $\Thnsset$
    \[ \St_{L\repD[m,X]}([\id_{[m]},\id_X])(i)\cong \Hom_{\Ch\Cone[\id_X][\id_{[m]}]}(i,m+1)\cong\Hom_{\Ch L\repD[m+1,X]}(i,m+1). \qedhere \]
\end{proof}

\begin{rmk} \label{rem:Stoflanglem}
    By \cref{StofpickmX,StofidmX},  the component of the $\Thnsset$-enriched natural transformation in $[\Ch L\repD[m,X]^{\op},\Thnsset]$
    \[ \St_{L\repD[m,X]}([\langle m\rangle, \id_X])\colon \St_{L\repD[m,X]}([\langle m\rangle, \id_X])\to \St_{L\repD[m,X]}([\id_{[m]},\id_X]) \]
    at $0\leq i\leq m$ is the map in $\Thnsset$ 
    \[ \Hom_{\Ch L\repD[m,X]\amalg_{[0]} \Sigma X}(i,m+1)\to \Hom_{\Ch L\repD[m+1,X]}(i,m+1) \]
    induced by the action on hom $\Thn$-spaces of the canonical $\Thnsset$-enriched functor 
    \[ \Ch L\repD[m,X]\amalg_{[0]} \Sigma X\to \Ch L\repD[m+1,X]. \]
\end{rmk}

Finally, we observe that in the case where $m=0$, as $LF[0,X]\cong F[0]$ and $\Ch\repD[0]\cong [0]$, we have that $\St_{\repD[0]}$ is a functor 
\[ \St_{\repD[0]}\colon \sThnsset\cong \sThnssetslice{\repD[0]}\to [\Ch \repD[0]^{\op},\Thnsset]\cong \Thnsset. \]
Hence we get the following simplification of \cref{StofidmX} in the case where $m=0$.

\begin{prop} \label{prop:StF0}
    Let $X$ be a connected $\Thn$-space. There is a natural isomorphism in $\Thnsset$
    \[ \St_{\repD[0]}([0,X])\cong X. \]
\end{prop}

\begin{rmk} \label{rem:StF0off}
    Let $f\colon X\to Y$ be a map in $\Thnsset$ between connected $\Thn$-space. By \cref{prop:StF0}, the induced map in $\Thnsset$ 
    \[ \St_{\repD[0]}([0,f])\colon \St_{\repD[0]}(\repD[0,X])\to \St_{\repD[0]}(\repD[0,Y]) \]
    is simply given by the map $f\colon X\to Y$ itself.
\end{rmk}

\section{Straightening-unstraightening: Quillen equivalence}

In this section, we prove that the straightening-unstraightening adjunction gives a Quillen equivalence
\begin{tz}
\node[](1) {$\MSrightfib{W}$}; 
\node[right of=1,xshift=3.9cm](2) {$[\Ch W^{\op},\MSThnsset]_\proj$}; 
\punctuation{2}{,};

\draw[->] ($(2.west)-(0,5pt)$) to node[below,la]{$\Un_{W}$} ($(1.east)-(0,5pt)$);
\draw[->] ($(1.east)+(0,5pt)$) to node[above,la]{$\St_{W}$} ($(2.west)+(0,5pt)$);

\node[la] at ($(1.east)!0.5!(2.west)$) {$\bot$};
\end{tz}
for every object $W$ of $\pcatThn$. In \cref{subsec:reductiontrick}, we first give a useful reduction trick for proving that the straightening functor preserves (trivial) cofibrations. Then, in \cref{subsec:QPbefore,subsec:QPafter}, we prove that the above adjunction is a Quillen pair before and after localizing. Finally, in \cref{subsec:QE}, we prove that it is a Quillen equivalence by comparing it to the Quillen equivalence from \cref{subsec:Grothendieck}.

\subsection{Reduction trick} \label{subsec:reductiontrick}

In order to prove that the straightening functor is left Quillen, we will make use of the following reduction trick several times.

\begin{prop} \label{reductiontrick}
Let $f\colon W\to Z$ be a map in $\pcatThn$ and let $g\colon A\to B$ be a map in $\sThnssetslice{W}$. If the image of $g$ under $\St_W\colon \sThnssetslice{W}\to [\Ch W^{\op},\Thnsset]$ is a (trivial) cofibration in $[\Ch W^{\op},\MSThnsset]_\proj$, then the image of $g$ under the composite of functors $\St_Z\circ f_!\colon \sThnssetslice{W}\to [\Ch Z^{\op},\Thnsset]$ is a (trivial) cofibration in $[\Ch Z^{\op},\MSThnsset]_\proj$.
\end{prop}

\begin{proof}
Suppose that $g$ is a map in $\sThnssetslice{W}$ such that $\St_W(g)$ is a (trivial) cofibration in $[\Ch W^{\op},\MSThnsset]_\proj$. Since $(\Ch f)_!\colon [\Ch W^{\op},\MSThnsset]_\proj\to [\Ch Z^{\op},\MSThnsset]_\proj$ is left Quillen by \cref{LKEisQP}, then $\Thnsset$-natural transformation $(\Ch f)_! \St_W(g)$ is a (trivial) cofibration in $[\Ch Z^{\op},\MSThnsset]_\proj$. By \cref{basechangeSt}, we have a natural isomorphism $(\Ch f)_! \St_W(g)\cong \St_Z(f_!(g))$ and so $\St_Z(f_!(g))$ is a (trivial) cofibration in $[\Ch W^{\op},\MSThnsset]_\proj$, as desired. 
\end{proof}

\begin{rmk} \label{rem:reductiontrick}
As a consequence, if we want to show that a map $g\colon A\to B$ in $\sThnssetslice{W}$ is sent by the functor $\St_W\colon \sThnssetslice{W}\to [\Ch W^{\op},\Thnsset]$ to a (trivial) cofibration in $[\Ch W^{\op},\MSThnsset]_\proj$, it is enough to show that $g\colon A\to B$ seen as a map in $\sThnssetslice{B}$ (using the identity at $B$) is sent by $\St_B\colon \sThnssetslice{B}\to [\Ch B^{\op},\Thnsset]$ to a (trivial) cofibration in $[\Ch B^{\op},\MSThnsset]_\proj$. Indeed, this follows from the above proposition and the fact that $p_!(g)=g$, where $p\colon B\to W$ denotes the map defining $B$ as an object in $\sThnssetslice{W}$.
\end{rmk}

\subsection{Quillen pair before localizing} \label{subsec:QPbefore}

We first prove here that the straightening-un\-strai\-ght\-ening adjunction is a Quillen pair before localizing. 

\begin{prop} \label{QPwithinj}
Let $W$ be an object in $\pcatThn$. The adjunction $\St_W\dashv \Un_W$ is a Quillen pair
\begin{tz}
 \node[](1) {$\injsThnspaceslice{W}$}; 
\node[right of=1,xshift=3.8cm](2) {$[\Ch W^{\op},\MSThnsset]_\proj$}; 
\punctuation{2}{.};

\draw[->] ($(2.west)-(0,5pt)$) to node[below,la]{$\Un_W$} ($(1.east)-(0,5pt)$);
\draw[->] ($(1.east)+(0,5pt)$) to node[above,la]{$\St_W$} ($(2.west)+(0,5pt)$);

\node[la] at ($(1.east)!0.5!(2.west)$) {$\bot$};    
\end{tz}
\end{prop}

To prove the above result, by the reduction trick from \cref{rem:reductiontrick}, it is enough to show that a map in $\sThnssetslice{L\repD[m,Y]}$ 
\[ \iota_m\widehat{\times} f\colon \partial\repD[m,Y]\amalg_{\partial\repD[m,X]} \repD[m,X]\to L \repD[m,Y]\]
with $m\geq 0$ and $f\colon X\hookrightarrow Y$ a generating (trivial) cofibration in $\MSThnsset$ is sent by the functor $\St_{L\repD[m,Y]}\colon \sThnssetslice{L\repD[m,Y]}\to [\Ch L\repD[m,Y]^{\op},\Thnsset]$ to a (trivial) cofibration in $[\Ch L\repD[m,Y]^{\op},\MSThnsset]_\proj$.

\begin{lemma} \label{lem:leveltrivcof}
    Let $m\geq 0$, $Y$ be a connected $\Thn$-space, and $f\colon X\hookrightarrow Y$ be a (trivial) cofibration in $\Thnsset$. Then the induced map 
    \[ \cG^0_m(f)\to \newG[\id_Y][0][\id_{[m]}] \]
    is a (trivial) cofibration in $(\MSThnsset)^{(\tndnec{\repD[m+1]}{0}{m+1})^{\op}}_\inj$.
\end{lemma}

\begin{proof}
    By definition of the injective model structure, it is enough to show that, for each object $T\hookrightarrow \repD[m+1]_{0,m+1}$ of $\tndnec{\repD[m+1]}{0}{m+1}$, the induced map 
    \[ \cG^0_m(f)(T)\to \newG[\id_Y][0][\id_{[m]}](T) \]
    is a (trivial) cofibration in $\MSThnsset$. However, recalling \cref{notn:Falphatnd,notation:Gmf}, a direct computation shows that
    \[ \cG^0_m(f)(T)\to \newG[\id_Y][0][\id_{[m]}](T)= \begin{cases}
    \prod_{B(T)} Y\xrightarrow{\id} \prod_{B(T)} Y & \text{if } T\neq\repD[m+1] \\
    X\xrightarrow{f} Y & \text{if } T=\repD[m+1].
    \end{cases}
    \]
    Hence we get the result, as $f$ is by assumption a (trivial) cofibration in $\MSThnsset$. 
\end{proof}

\begin{lemma} \label{lem:projcof}
    For $m\geq 0$, the functor
\[ \diag(\colim^{H^0_{m+1}}_{(\tndnec{\repD[m+1]}{0}{m+1})^{\op}} \iota (-))\colon (\MSThnsset)^{(\tndnec{\repD[m+1]}{0}{m+1})^{\op}}_\inj\to \MSThnsset \]
is left Quillen.
\end{lemma}

\begin{proof}
    This is obtained as a combination of \cite[Propositions 1.5.3 and 3.3.9]{MRR1}, noticing that the functor $H^0_{m+1}$ from \cref{notationH} coincides with the functor $H_{m+1}$ from \cite[Notation 3.3.3]{MRR1}. 
\end{proof}

\begin{lemma} \label{prop:trivcofat0}
    Let $m\geq 0$, $Y$ be a connected $\Thn$-space, and $f\colon X\hookrightarrow Y$ be a (trivial) cofibration in $\Thnsset$. Then the map 
    \[ \St_{L\repD[m,Y]}(\iota_m\widehat{\times} f)_0\colon \St_{L\repD[m,Y]}(\iota_m\widehat{\times} f)(0)\to \St_{L\repD[m,Y]}([\id_{[m]},\id_Y])(0) \]
    is a (trivial) cofibration in $\Thnsset$.
\end{lemma}

\begin{proof}
    By \cref{rem:Stofpushprod}, the above map in $\Thnsset$ can be computed as the map 
    \[ \diag(\colim^{H^0_{m+1}}_{(\tndnec{\repD[m+1]}{0}{m+1})^{\op}} \iota (\cG^0_m(f)\to \newG[\id_Y][0][\id_{[m]}])). \]
   Since the map $\cG^0_m(f)\to \newG[\id_Y][0][\id_{[m]}]$ is a (trivial) cofibration in $(\MSThnsset)^{(\tndnec{\repD[m+1]}{0}{m+1})^{\op}}_\inj$ by \cref{lem:leveltrivcof}, then by \cref{lem:projcof} the above map is a (trivial) cofibration in $\MSThnsset$.
\end{proof}

\begin{lemma} \label{lem:proj(triv)cof}
    Let $\cC$ be a directed $\Thnsset$-enriched category with object set $\{0,1,\ldots,m\}$, and $\eta\colon F\to G$ be a $\Thnsset$-enriched natural transformation in $[\cC^{\op},\Thnsset]$. Suppose that:
    \begin{itemize}[leftmargin=0.6cm]
        \item for $0<i\leq m$, the map $\eta_i\colon F(i)\to G(i)$ is an isomorphism in $\Thnsset$,
        \item the map $\eta_0\colon F(0)\to G(0)$ is a (trivial) cofibration in $\MSThnsset$.
    \end{itemize}
    Then $\eta\colon F\to G$ is a (trivial) cofibration in $[\cC^{\op},\MSThnsset]_\proj$.
\end{lemma}

\begin{proof}
    We deal with the case where $\eta_0$ is a cofibration in $\MSThnsset$; the case of a trivial cofibration proceeds similarly. 

    We show that the $\Thnsset$-enriched natural transformation $\eta\colon F\to G$ has the left lifting property with respect to all trivial fibrations in $[\cC^{\op},\MSThnsset]_\proj$. Let $\chi\colon P\to Q$ be a trivial fibration in $[\cC^{\op},\MSThnsset]_\proj$, i.e., the map $\chi_i\colon P(i)\to Q(i)$ is a trivial fibration in $\MSThnsset$, for all $0\leq i\leq m$. Consider a commutative square in $[\cC^{\op},\Thnsset]$ as below left. 
    \begin{tz}
         \node[](1) {$F$}; 
        \node[right of=1,xshift=.55cm](2) {$P$}; 
        \node[below of=1](3) {$G$}; 
        \node[below of=2](4) {$Q$}; 

        \draw[->] (1) to node[above,la]{$\beta$} (2); 
        \draw[->] (1) to node[left,la]{$\eta$} (3); 
        \draw[->] (2) to node[right,la]{$\chi$} (4); 
        \draw[->] (3) to node[below,la]{$\gamma$} (4);

        \draw[->,dashed] (3) to node[left,la,yshift=3pt]{$\delta$} (2);

        \node[right of=2,xshift=2cm](1) {$F(0)$}; 
        \node[right of=1,xshift=1cm](2) {$P(0)$}; 
        \node[below of=1](3) {$G(0)$}; 
        \node[below of=2](4) {$Q(0)$}; 

        \draw[->] (1) to node[above,la]{$\beta_0$} (2); 
        \draw[->] (1) to node[left,la]{$\eta_0$} (3); 
        \draw[->] (2) to node[right,la]{$\chi_0$} (4); 
        \draw[->] (3) to node[below,la]{$\gamma_0$} (4);

        \draw[->,dashed] (3) to node[left,la,yshift=3pt]{$\delta_0$} (2);
    \end{tz}
    We construct a $\Thnsset$-enriched natural transformation $\delta\colon G\to P$ making the above left diagram commute. For $i=0$, we set $\delta_0$ to be a lift in the above right commutative square in $\Thnsset$, which exists since by assumption, and, for $0<i\leq m$, we set $\delta_i$ to be the composite
    \[ \delta_i\colon G(i)\xrightarrow{\eta_i^{-1}} F(i)\xrightarrow{\beta_i} P(i). \]

    We show that the $\delta_i$'s assemble into a $\Thnsset$-enriched natural transformation. Since $\eta$ and~$\beta$ are $\Thnsset$-enriched natural transformations, the enriched naturality condition for $\delta$ clearly holds in the case where $0<i\leq j\leq m$, and, for $0<i\leq m$, it is obtained from the commutativity of the following diagram.
    \begin{tz}
         \node[](1) {$G(i)\times \Hom_\cC(0,i)$}; 
        \node[right of=1,xshift=2.2cm](2) {$G(0)$}; 
        \node[below of=1](3) {$F(i)\times \Hom_\cC(0,i)$}; 
        \node[below of=2](4) {$F(0)$};
        \node[below of=3](5) {$P(i)\times \Hom_\cC(0,i)$}; 
        \node[below of=4](6) {$P(0)$}; 

        \draw[->] (1) to node[above,la]{$G_{0,i}$} (2); 
        \draw[->] (1) to node[left,la]{$\eta_i^{-1}\times \id$} (3); 
        \draw[->] (4) to node[right,la]{$\eta_0$} (2); 
        \draw[->] (3) to node[left,la]{$\beta_i\times \id$} (5); 
        \draw[->] (4) to node[right,la]{$\beta_0$} (6); 
        \draw[->] (3) to node[below,la]{$F_{0,i}$} (4);
        \draw[->] (5) to node[below,la]{$P_{0,i}$} (6);
        \draw[->,bend right=50] ($(1.south west)+(5pt,0)$) to node[left,la]{$\delta_i\times \id$} ($(5.north west)+(5pt,0)$); 
        \draw[->,bend left=50] (2) to node[right,la]{$\delta_0$} (6); 
    \end{tz}
    So we have a $\Thnsset$-enriched natural transformation $\delta\colon G\to P$ as desired. 
    \end{proof}

\begin{prop} \label{StLprestrivcof}
    Let $m\geq 0$, $Y$ be a connected $\Thn$-space, and $f\colon X\hookrightarrow Y$ be a (trivial) cofibration in $\Thnsset$. Then the induced $\Thnsset$-enriched natural transformation 
    \[ \St_{L\repD[m,Y]}(\iota_m\widehat{\times} f)\colon \St_{L\repD[m,Y]}(\iota_m\widehat{\times} f)\to \St_{L\repD[m,Y]}([\id_{[m]},\id_Y]) \]
    is a (trivial) cofibration in $[\Ch L\repD[m,Y]^{\op},\MSThnsset]_\proj$.
\end{prop}

\begin{proof}
    For $0<i\leq m$, by \cref{rem:Stofpushprod}, the map
    \[ \St_{L\repD[m,Y]}(\iota_m\widehat{\times} f)_i\colon \St_{L\repD[m,Y]}(\iota_m\widehat{\times} f)(i)\to \St_{L\repD[m,Y]}([\id_{[m]},\id_Y])(i) \]
    is an isomorphism in $\Thnsset$, and by \cref{prop:trivcofat0} the map 
    \[ \St_{L\repD[m,Y]}(\iota_m\widehat{\times} f)_0\colon \St_{L\repD[m,Y]}(\iota_m\widehat{\times} f)(0)\to \St_{L\repD[m,Y]}([\id_{[m]},\id_Y])(0) \]
    is a (trivial) cofibration in $\MSThnsset$. Hence the result follows from \cref{lem:proj(triv)cof}.
\end{proof}

\begin{rmk} \label{rem:gencofinjThn}
    Recall from \cite[Theorem 15.6.27]{Hirschhorn} that a set of generating cofibrations in $\injThnspace$ is given by the monomorphisms
    \[ X\xhookrightarrow{f} Y\coloneqq (\partial\repThn\hookrightarrow \repThn)\widehat{\times}(\partial\repS\hookrightarrow \repS) \]
    with $\defThn\in \Thn$ and $\defS\geq 0$. 
\end{rmk}

We are now ready to show that the straightening functor is left Quillen.

\begin{prop} \label{StWprescof}
    Let $W$ be an object in $\pcatThn$. Then the functor 
    \[ \St_W\colon \injsThnspaceslice{W}\to [\Ch W^{\op},\MSThnsset]_\proj \]
    preserves cofibrations. 
\end{prop}

\begin{proof}
    By \cref{rem:gencofinjslice,rem:gencofinjThn}, it is enough to show that the functor $\St_W$ sends the generating cofibrations 
    \[ \partial\repD[m,Y]\amalg_{\partial\repD[m,X]} \repD[m,X]\xrightarrow{\iota_m\widehat{\times} f} \repD[m,Y] \to W \]
    in $\injsThnspaceslice{W}$ with $m\geq 0$, $f\colon X\hookrightarrow Y$ a monomorphism in $\Thnsset$ of the form 
    \[ X\xhookrightarrow{f} Y\coloneqq (\partial\repThn\hookrightarrow \repThn)\widehat{\times}(\partial\repS\hookrightarrow \repS), \]
    and $\repD[m,Y]\to W$ a map in $\sThnsset$, to a cofibration in $[\Ch W^{\op},\MSThnsset]_\proj$. By \cref{rem:reductiontrick}, it is enough to show that, for $m\geq 1$ and $f\colon X\hookrightarrow Y$ as above, the functor $\St_{L\repD[m,Y]}\colon \sThnssetsliceshort{L\repD[m,Y]}\to [\Ch L\repD[m,Y]^{\op},\Thnsset]$ sends the map 
    \[ \partial\repD[m,Y]\amalg_{\partial\repD[m,X]} \repD[m,X]\xrightarrow{\iota_m\widehat{\times} f} \repD[m,Y] \xrightarrow{[\id_{[m]},\id_Y]} L\repD[m,Y] \]
    to a cofibration in $[\Ch L\repD[m,Y]^{\op},\MSThnsset]_\proj$.

    Since by \cref{rem:repareconnected} the $\Thn$-space $Y=\repThn\times \repS=\repD[\defThn,\defS]$ is connected, by \cref{StLprestrivcof} the induced $\Thnsset$-enriched natural transformation 
    \[ \St_{L\repD[m,Y]}(\iota_m\widehat{\times} f)\colon \St_{L\repD[m,Y]}(\iota_m\widehat{\times} f)\to \St_{L\repD[m,Y]}([\id_{[m]},\id_Y]) \]
    is a cofibration in $[\Ch L\repD[m,Y]^{\op},\MSThnsset]_\proj$, as desired. 
\end{proof}

\begin{rmk} \label{rem:gentrivcofinjThn}
    Recall from \cite[Theorem 15.6.27]{Hirschhorn} that a set of generating trivial cofibrations in $\injThnspace$ is given by the monomorphisms
    \[ X\xhookrightarrow{f} Y\coloneqq (\partial\repThn\hookrightarrow \repThn)\widehat{\times}(\Lambda^t[\defS]\hookrightarrow \repS) \]
    with $\defThn\in \Thn$, $\defS\geq 1$, and $0\leq t\leq \defS$. 
\end{rmk}

\begin{rmk} \label{rem:trivcofinloc}
    Note that, since $\MSThnsset$ is a localization of $\injThnspace$, a trivial cofibration $f\colon X\hookrightarrow Y$ in $\injThnspace$ is also a trivial cofibration in $\MSThnsset$.
\end{rmk}

\begin{prop} \label{StWprestrivcof}
    Let $W$ be an object in $\pcatThn$. Then the functor 
    \[ \St_W\colon \injsThnspaceslice{W}\to [\Ch W^{\op},\MSThnsset]_\proj \]
    preserves trivial cofibrations. 
\end{prop}

\begin{proof}
    By \cref{rem:gencofinjslice,rem:gentrivcofinjThn}, it is enough to show that the functor $\St_W$ sends the generating trivial cofibrations 
    \[ \partial\repD[m,Y]\amalg_{\partial\repD[m,X]} \repD[m,X]\xrightarrow{\iota_m\widehat{\times} f} \repD[m,Y] \to W \]
    in $\injsThnspaceslice{W}$ with $m\geq 0$, $f\colon X\hookrightarrow Y$ a monomorphism in $\Thnsset$ of the form 
    \[ X\xhookrightarrow{f} Y\coloneqq (\partial\repThn\hookrightarrow \repThn)\widehat{\times}(\Lambda^t[\defS]\hookrightarrow \repS), \]
    and $\repD[m,Y]\to W$ a map in $\sThnsset$, to a trivial cofibration in $[\Ch W^{\op},\MSThnsset]_\proj$. By \cref{reductiontrick}, it is enough to show that, for $m\geq 1$ and $f\colon X\hookrightarrow Y$ as above, the functor $\St_{L\repD[m,Y]}\colon \sThnssetsliceshort{L\repD[m,Y]}\to [\Ch L\repD[m,Y]^{\op},\Thnsset]$ sends the map 
    \[ \partial\repD[m,Y]\amalg_{\partial\repD[m,X]} \repD[m,X]\xrightarrow{\iota_m\widehat{\times} f} \repD[m,Y] \xrightarrow{[\id_{[m]},\id_Y]} L\repD[m,Y] \]
    to a trivial cofibration in $[\Ch L\repD[m,Y]^{\op},\MSThnsset]_\proj$. 
    
    Since by \cref{rem:repareconnected} the $\Thn$-space $Y=\repThn\times \repS=\repD[\defThn,\defS]$ is connected and by \cref{rem:trivcofinloc} the map $f$ is a trivial cofibration in $\MSThnsset$, by \cref{StLprestrivcof} the induced $\Thnsset$-enriched natural transformation 
    \[ \St_{L\repD[m,Y]}(\iota_m\widehat{\times} f)\colon \St_{L\repD[m,Y]}(\iota_m\widehat{\times} f)\to \St_{L\repD[m,Y]}([\id_{[m]},\id_Y]) \]
    is a trivial cofibration in $[\Ch L\repD[m,Y]^{\op},\MSThnsset]_\proj$, as desired.
\end{proof}

\begin{proof}[Proof of \cref{QPwithinj}]
    The functor
   $\St_W\colon \injsThnspaceslice{W}\to [\Ch W^{\op},\MSThnsset]_\proj$
    preserves (trivial) cofibrations by \cref{StWprescof,StWprestrivcof}, and so it is left Quillen.
\end{proof}

\subsection{Quillen pair after localizing} \label{subsec:QPafter}

We now prove that the desired straightening-un\-strai\-ght\-ening adjunction is a Quillen pair with respect to the desired localization.

\begin{prop} \label{QPwithloc}
Let $W$ be an object in $\pcatThn$. The adjunction $\St_W\dashv \Un_W$ is a Quillen pair enriched over $\MSThnsset$
\begin{tz}
 \node[](1) {$\MSrightfib{W}$}; 
\node[right of=1,xshift=4cm](2) {$[\Ch W^{\op},\MSThnsset]_\proj$};
\punctuation{2}{.};

\draw[->] ($(2.west)-(0,5pt)$) to node[below,la]{$\Un_W$} ($(1.east)-(0,5pt)$);
\draw[->] ($(1.east)+(0,5pt)$) to node[above,la]{$\St_W$} ($(2.west)+(0,5pt)$);

\node[la] at ($(1.east)!0.5!(2.west)$) {$\bot$};    
\end{tz}
\end{prop}

To prove the above result, by \cite[Theorem 3.3.20(1)(a)]{Hirschhorn}, \cref{QPwithinj}, and the reduction trick from \cref{rem:reductiontrick}, it is enough to show that a map in $\sThnssetsliceshort{L\repD[m,X]}$ 
\[ [\langle m\rangle,\id_X]\colon \repD[0,X]\to \repD[m,X] \]
for $m\geq 1$ and $X$ a connected $\Thn$-space is sent by the functor \[ \St_{L\repD[m,X]}\colon \sThnssetsliceshort{L\repD[m,X]}\to [\Ch L\repD[m,X]^{\op},\Thnsset] \]
to a weak equivalence in $[\Ch L\repD[m,X]^{\op},\MSThnsset]_\proj$; and that a map in $\sThnsset$ 
\[ [\id_{[0]},f]\colon \repD[0,X]\to \repD[0,Y] \]
for $f\colon X\hookrightarrow Y$ a trivial cofibration in $\MSThnsset$ between connected $\Thn$-spaces is sent by the functor $\St_{\repD[0]}\colon \sThnsset\to \Thnsset$ to a weak equivalence in $\MSThnsset$.

We first deal with the map $[\langle m\rangle,\id_X]$. Recall the functor $\Sigma_m\colon \Thnsset\to \Thncat$ which sends an object $X\in \Thnsset$ to the pushout $\Sigma_mX\coloneqq \Sigma X\amalg_{[0]}\ldots \amalg_{[0]} \Sigma X$ of $m$ copies of $\Sigma X$ along consecutive sources and targets. Then by \cite[Lemma 4.2.3]{MRR1} combined with \cite[Corollary 3.5.2]{MRR1} we have the following result. 

\begin{lemma} \label{lem:SigmavsSh}
    Let $m\geq 1$, and $X$ be a connected $\Thn$-space. The $\Thnsset$-enriched functor 
    \[ \Sigma_m X\to \Ch L\repD[m,X]\]
    is a weak equivalence in $\MSThncat$. 
\end{lemma}

\begin{lemma} \label{lem:Sigmavspushout}
    Let $m\geq 1$, and $X$ be a connected $\Thn$-space. The $\Thnsset$-enriched functor 
    \[ \Sigma_{m+1} X\to \Ch L\repD[m,X]\amalg_{[0]} \Sigma X\]
    is a weak equivalence in $\MSThncat$. 
\end{lemma}

\begin{proof}
    Consider the following diagram of spans in $\Thncat$
    \begin{tz}
\node[](1') {$\Sigma_m X$};
\node[right of=1',xshift=.9cm](2') {$[0]$};
\node[right of=2',xshift=.8cm](3') {$\Sigma X$};
\draw[->] (2') to node[above,la]{$m$} (1'); 
\draw[right hook->] (2') to node[above,la]{$0$} (3'); 

\node[below of=1'](1) {$\Ch L\repD[m,X]$};
\node[below of=2'](2) {$[0]$};
\node[below of=3'](3) {$\Sigma X$};
\draw[->] (2) to node[below,la]{$m$} (1); 
\draw[right hook->] (2) to node[below,la]{$0$} (3); 

\draw[->] (1') to node[left,la]{$\simeq$} (1);
\draw[d] (2') to (2);
\draw[d] (3') to (3);
    \end{tz}
    where $0\colon [0]\to \Sigma X$ is a cofibration in $\MSThncat$ and $\Sigma_m X\to \Ch L\repD[m,X]$ is a weak equivalence in $\MSThncat$ by \cref{lem:SigmavsSh}. In particular, the pushout in $\Thncat$ of each span is a homotopy pushout in $\MSThncat$ and every vertical $\Thnsset$-enriched functor is a weak equivalence in $\MSThncat$. Hence the induced $\Thnsset$-enriched functor between pushouts
    \[ \Sigma_{m+1} X\to \Ch L\repD[m,X]\amalg_{[0]} \Sigma X\]
    is also a weak equivalence in $\MSThncat$, as desired.
\end{proof}

\begin{prop} \label{prop:pushoutvsSh}
    Let $m\geq 1$, and $X$ be a connected $\Thn$-space. The $\Thnsset$-enriched functor 
    \[ \Ch L\repD[m,X]\amalg_{[0]} \Sigma X\to \Ch L\repD[m+1,X] \]
    is a weak equivalence in $\MSThncat$.
\end{prop}

\begin{proof}
    We have a commutative triangle in $\MSThncat$
    \begin{tz}
        \node[](1) {$\Sigma_{m+1} X$}; 
        \node[right of=1,xshift=2.5cm](2) {$\Ch L\repD[m,X]\amalg_{[0]} \Sigma X$}; 
        \node[below of=2](3) {$\Ch L\repD[m+1,X]$}; 

        \draw[->] (1) to node[above,la]{$\simeq$} (2); 
        \draw[->] (1) to node[below,la]{$\simeq$} (3); 
        \draw[->] (2) to (3);
    \end{tz}
    where $\Sigma_{m+1} X\to \Ch L\repD[m+1,X]$ and $\Sigma_{m+1} X\to \Ch L\repD[m,X]\amalg_{[0]} \Sigma X$ are weak equivalences in $\MSThncat$ by \cref{lem:SigmavsSh,lem:Sigmavspushout}. Hence, by $2$-out-of-$3$, the $\Thnsset$-enriched functor 
    \[ \Ch L\repD[m,X]\amalg_{[0]} \Sigma X\to \Ch L\repD[m+1,X] \]
    is a weak equivalence in $\MSThncat$, as desired.
\end{proof}

\begin{prop} \label{prop:StLsendsrightfib}
    Let $m\geq 1$, and $X$ be a connected $\Thn$-space. The induced $\Thnsset$-enriched natural transformation
    \[ \St_{L\repD[m,X]}([\langle m\rangle,\id_X])\colon \St_{L\repD[m,X]}([\langle m\rangle,\id_X])\to \St_{L\repD[m,X]}([\id_{[m]},\id_X]) \]
    is a trivial cofibration in $[\Ch L\repD[m,X]^{\op},\MSThnsset]_\proj$.
\end{prop}

\begin{proof}
    By \cref{StWprescof}, we know that the above $\Thnsset$-enriched natural transformation is a cofibration in $[\Ch L\repD[m,X]^{\op},\MSThnsset]_\proj$. Hence it remains to show that it is also a weak equivalence in $[\Ch L\repD[m,X]^{\op},\MSThnsset]_\proj$. 
    
    For $0\leq i\leq m$, by \cref{rem:Stoflanglem} the map  
    \[ \St_{L\repD[m,X]}([\langle m\rangle,\id_X])_i\colon \St_{L\repD[m,X]}([\langle m\rangle,\id_X])(i)\to \St_{L\repD[m,X]}([\id_{[m]},\id_X])(i) \]
    is given by the map in $\Thnsset$
    \begin{equation}\label{maponhom} \Hom_{\Ch L\repD[m,X]\amalg_{[0]} \Sigma X}(i,m+1)\to \Hom_{\Ch L\repD[m+1,X]}(i,m+1) \end{equation}
    induced by the weak equivalence $\Ch L\repD[m,X]\amalg_{[0]} \Sigma X\to \Ch L\repD[m+1,X]$ in $\MSThncat$ from \cref{prop:pushoutvsSh}. Hence, by definition of the weak equivalences in $\MSThncat$, we get that~\eqref{maponhom} is a weak equivalence in $\MSThnsset$. This shows that
    \[ \St_{L\repD[m,X]}([\langle m\rangle,\id_X])\colon \St_{L\repD[m,X]}([\langle m\rangle,\id_X])\to \St_{L\repD[m,X]}([\id_{[m]},\id_X]) \]
    is a weak equivalence in $[\Ch L\repD[m,X]^{\op},\MSThnsset]_\proj$, as desired.
\end{proof}

\begin{prop} \label{StWofrightfib}
    Let $W$ be an object in $\pcatThn$. Then the  straightening functor $\St_W\colon \sThnssetslice{W}\to [\Ch W^{\op},\Thnsset]$ sends a map 
    \[ \repD[0,\defThn,0]\xrightarrow{[\langle m\rangle,\id_{\repD[\defThn,0]}]} \repD[m,\defThn,0]\to W \]
    with $m\geq 1$, $\defThn\in \Thn$, and $\repD[m,\defThn,0]\to W$ a map in $\sThnsset$, to a trivial cofibration in~$[\Ch W^{\op},\MSThnsset]_\proj$.
\end{prop}

\begin{proof}
    Note that, by \cref{reductiontrick}, it is enough to show that the straightening functor $\St_{L\repD[m,\defThn,0]}\colon \sThnssetsliceshort{L\repD[m,\defThn,0]}\to [\Ch L\repD[m,\defThn,0]^{\op},\Thnsset]$ sends the map 
    \[ \repD[0,\defThn,0]\xrightarrow{[\langle m\rangle,\id_{\repD[\defThn,0]}]} \repD[m,\defThn,0]\xrightarrow{[\id_{[m]},\id_{\repD[\defThn,0]}]} \repD[m,\defThn,0] \]
    to a trivial cofibration in $[\Ch L\repD[m,\defThn,0]^{\op},\MSThnsset]_\proj$. Since by \cref{rem:repareconnected} the representable $F[\defThn,0]=\repThn$ is connected, by \cref{prop:StLsendsrightfib} the induced $\Thnsset$-enriched natural transformation 
    \[ \St_{L\repD[m,\defThn,0]}([\langle m\rangle,\id_{\repD[\defThn,0]}])\colon \St_{L\repD[m,\defThn,0]}([\langle m\rangle,\id_{\repD[\defThn,0]}])\to \St_{L\repD[m,\defThn,0]}([\id_{[m]},\id_{\repD[\defThn,0]}]) \]
    is a trivial cofibration in $[\Ch L\repD[m,\defThn,0]^{\op},\MSThnsset]_\proj$, as desired.
\end{proof}

We now deal with the map $[\id_{[0]},f]$. 

\begin{prop}
    Let $f\colon X\hookrightarrow Y$ be a trivial cofibration in $\MSThnsset$ between connected $\Thn$-spaces. Then the induced map 
    \[ \St_{\repD[0]}([\id_{[0]},f])\colon \St_{\repD[0]}(\repD[0,X])\to \St_{\repD[0]}(\repD[0,Y]) \]
    is a trivial cofibration in $\MSThnsset$. 
\end{prop}

\begin{proof}
    This is straightforward from the computation in \cref{rem:StF0off}.
\end{proof}

\begin{rmk} \label{Sisconnected}
    Note that all maps $f\colon X\to Y$ in $S_{\CSThn}$ are such that $X$ and $Y$ are connected $\Thn$-spaces. Indeed, this can be shown by induction on $n\geq 1$. 
\end{rmk}

\begin{prop} \label{StWofrec}
    Let $W$ be an object in $\pcatThn$. Then the straightening functor $\St_W\colon \sThnssetslice{W}\to [\Ch W^{\op},\Thnsset]$ sends a map 
    \[ \repD[0,X]\xrightarrow{[\id_{[0]},f]} \repD[0,Y]\to W \]
    with $f\colon X\hookrightarrow Y$ a map in $S_{\CSThn}$, and $[0,Y]\to W$ a map in $\sThnsset$, to a trivial cofibration in $[\Ch W^{\op},\MSThnsset]_\proj$.
\end{prop}

\begin{proof}
    Note that, by \cref{reductiontrick}, it is enough to show that the straightening functor $\St_{\repD[0]}\colon \sThnsset\to \Thnsset$ sends the map 
    \[ [\id_{[0]},f]\colon \repD[0,X]\to \repD[0,Y] \]
    to a trivial cofibration in $\MSThnsset$. Since by \cref{Sisconnected} the $\Thn$-spaces $X$ and $Y$ are connected, by \cref{prop:StLsendsrightfib} the induced map 
    \[ \St_{\repD[0]}([\id_{[0]},f])\colon \St_{\repD[0]}([0,X])\to \St_{\repD[0]}([0,X]) \]
    is a trivial cofibration in $\MSThnsset$, as desired.    
\end{proof}

\begin{proof}[Proof of \cref{QPwithloc}]
    As the functor $\St_W\colon \injsThnspaceslice{W}\to [\Ch W^{\op},\MSThnsset]_\proj$ is left Quillen by \cref{QPwithinj}, to show that the functor 
    \[ \St_W\colon \MSrightfib{W}\to [\Ch W^{\op},\MSThnsset]_\proj \]
    from the localization is also left Quillen, it is enough by \cite[Theorem 3.3.20(1)(a)]{Hirschhorn} to show that it sends the maps from \cref{subsec:MSrightfib} with respect to which we localize the model structure $\injsThnspaceslice{W}$ to obtain the model structure $\MSrightfib{W}$ to weak equivalence in $[\Ch W^{\op},\MSThnsset]_\proj$. But this follows from \cref{StWofrightfib,StWofrec}. 

    The fact that the Quillen pair is enriched follows from the $\MSThnsset$-enrichment of both model structures given by \cref{MSrightfibenriched,MSprojenriched} together with \cref{lem:Stprestensors} showing that the functor $\St_W$ is compatible with these enrichments.
\end{proof}

\subsection{Quillen equivalence} \label{subsec:QE}

Finally, we prove that the straightening-unstraightening Quillen pair is further a Quillen equivalence. 

We first prove that the functor $\St_{\Nh\cC}\colon \MSrightfib{\Nh\cC}\to [\Ch(\Nh\cC)^{\op},\MSThnsset]_\proj$ is a Quillen equivalence in the case where $\cC$ is a fibrant $\MSThnsset$-enriched category, by showing that the following square of left Quillen functors
 \begin{tz}
\node[](1) {$[\cC^{\op},\MSThnsset]_\proj$}; 
\node[below of=1](2) {$[\cC^{\op},\MSThnsset]_\proj$}; 
\node[right of=1,xshift=3.7cm](3) {$\MSrightfib{\Nh\cC}$}; 
\node[below of=3](4) {$[\Ch(\Nh\cC)^{\op},\MSThnsset]_\proj$};

\draw[->] (1) to node[above,la]{$\int^\Nh_\cC$} (3); 
\draw[->] (3) to node(a)[right,la]{$\St_{\Nh\cC}$} (4); 
\draw[->] (4) to node[below,la]{$(\varepsilon_\cC)_!$} (2);
\draw[->] (1) to node(b)[left,la]{$\id$} (2);

\cell[la,above,xshift=3pt][n][.5]{a}{b}{$\simeq$};
\end{tz}
commutes up to a natural isomorphism at the level of homotopy categories, where the $\Thnsset$-enriched functor $\varepsilon_\cC\colon \Ch\Nh\cC\to \cC$ is the component at $\cC$ of the (derived) counit of the adjunction $\Ch\dashv \Nh$. Note that we already know that every functor in the above square except for $\St_{\Nh\cC}$ is a Quillen equivalence, as we will also recall later. 

Let us fix a $\Thnsset$-enriched category $\cC$ and an object $c\in \cC$. The latter can also be regarded as an object $c\colon \repD[0]\to \Nh\cC$ in $\sThnssetslice{\Nh\cC}$.

\begin{lemma} \label{Stofc}
    There is an isomorphism in $[\Ch(\Nh\cC)^{\op}, \Thnsset]$
    \[ \St_{\Nh\cC}(c)\cong \Hom_{\Ch\Nh\cC}(-,c). \]
\end{lemma}

\begin{proof}
    By applying the colimit-preserving functor $\Ch\colon \pcatThn\to \Thncat$ to the pushout of \cref{pushoutlowersigma} in the case where $W=\Nh\cC$ and $\sigma=c\colon \repD[0]\to \Nh\cC$ and using \cref{Sh1issigma}, we obtain the following pushout in $\Thncat$. 
    \begin{tz}
        \node[](1) {$[0]$}; 
        \node[right of=1,xshift=1.3cm](2) {$\Ch\Nh\cC$}; 
        \node[below of=1](3) {$\Sigma(\repD[0])$}; 
        \node[below of=2](4) {$\Ch(\Nh\cC_c)$}; 

        \draw[->] (1) to node[above,la]{$c$} (2); 
        \draw[->] (1) to node[left,la]{$0$} (3); 
        \draw[->] (2) to node[right,la]{$\Ch(\iota_c)$} (4); 
        \draw[->] (3) to (4);
        \pushout{4};
    \end{tz}
    Then, by definition, we get that $\St_{\Nh\cC}(c)$ is the composite
    \[ \St_{\Nh\cC}(c)\colon \Ch(\Nh\cC)^{\op}\xrightarrow{\Ch(\iota_c)} \Ch(\Nh\cC_c)^{\op}\xrightarrow{\Hom_{\Ch(\Nh\cC_c)(-,\top)}} \Thnsset. \]
    Given an object $a\in \cC$, a direct computation shows that there is an isomorphism in $\Thnsset$
    \[ \Hom_{\Ch(\Nh\cC_c)}(a,\top)\cong \Hom_{\Ch\Nh\cC}(a,c), \]
    as the pushout $\Ch(\Nh\cC_c)$ is obtained from $\Ch\Nh\cC$ by adding a unique object $\top$ and a unique morphism $c\to \top$ together with free composites with it. Hence this gives isomorphisms in $[\Ch(\Nh\cC)^{\op}, \Thnsset]$
    \[ \St_{\Nh\cC}(c)\cong \Hom_{\Ch(\Nh\cC_c)}(-,\top)\circ \Ch(\iota_c)\cong \Hom_{\Ch\Nh\cC}(-,c). \qedhere\]
\end{proof}

\begin{lemma} \label{lem:computeLKE}
    There is an isomorphism in $[\cC^{\op}, \Thnsset]$
    \[ (\varepsilon_\cC)_!\St_{\Nh\cC}(c)\cong \Hom_{\cC}(-,c). \]
\end{lemma}

\begin{proof}
    This follows from \cref{Stofc} and the fact that the left Kan extension of the representable $\Hom_{\Ch \Nh\cC}(-,c)$ along $\varepsilon_\cC$ is the representable $\Hom_{\cC}(-,c)$ by \cite[(4.32)]{Kelly}, as the $\Thnsset$-enriched functor $\varepsilon_\cC\colon \Ch\Nh\cC\to \cC$ is the identity on objects. 
\end{proof}

As a consequence, the map $\id_c\colon \repD[0]\to \int_\cC^\Nh\Hom_\cC(-,c)$ in $\Thnsset$ induces a $\Thnsset$-enriched natural transformation in $[\cC^{\op},\Thnsset]$ 
\[ \textstyle(\varepsilon_\cC)_!\St_{\Nh\cC}(\id_c)\colon \Hom_\cC(-,c)\cong (\varepsilon_\cC)_!\St_{\Nh\cC}(c)\to (\varepsilon_\cC)_!\St_{\Nh\cC} \int_\cC^\Nh \Hom_\cC(-,c). \]
However, this construction is not natural in $c\in \cC$, and so in \cref{sec:comparisonmap}, we construct a retraction of the above, which happens to be natural in $c\in \cC$. Namely, we prove the following. 

\begin{prop} \label{natphic}
    There is a $\Thnsset$-enriched natural transformation in $[\cC^{\op},\Thnsset]$
    \[ \textstyle\varphi_c\colon (\varepsilon_\cC)_!\St_{\Nh\cC} \int_\cC^\Nh \Hom_\cC(-,c)\to \Hom_\cC(-,c) \]
    which is natural in $c\in \cC$, and is a retraction of the $\Thnsset$-enriched natural transformation $(\varepsilon_\cC)_!\St_{\Nh\cC}(\id_c)\colon \Hom_\cC(-,c)\cong (\varepsilon_\cC)_!\St_{\Nh\cC}(c)\to (\varepsilon_\cC)_!\St_{\Nh\cC} \int_\cC^\Nh \Hom_\cC(-,c)$.
\end{prop}

We now prove that the comparison map is a weak equivalence, hence inducing an isomorphism at the level of homotopy categories. 

\begin{lemma} \label{Intciswe}
    The map 
    \[ \textstyle\id_c\colon \repD[0]\to \int_\cC^\Nh\Hom_\cC(-,c) \]
    is a weak equivalence in $\MSrightfib{\Nh\cC}$. 
\end{lemma}

\begin{proof}
    By \cite[Example 3.23 and Theorem 5.25]{RasekhD}, the map $\textstyle\id_c\colon \repD[0]\to \int_\cC^N\Hom_\cC(-,c)$ is a weak equivalence in $\MSrightfib{N\cC}$. Hence, by applying the left Quillen functor $\varphi_!\colon \MSrightfib{N\cC}\to \MSrightfib{\Nh\cC}$ induced by the canonical map $\varphi\colon N\cC\to \Nh\cC$, we get a weak equivalence in $\MSrightfib{\Nh\cC}$
    \[ \textstyle\id_c\colon \varphi_!\repD[0]\to \varphi_!\int_\cC^N\Hom_\cC(-,c). \]
    Now, by \cref{triangleint}, we have a weak equivalence in $\MSrightfib{\Nh\cC}$
    \[ \textstyle\varphi_!\int_\cC^N\Hom_\cC(-,c)\to \int_\cC^\Nh\Hom_\cC(-,c) \]
    and hence we obtain that the composite 
    \[ \textstyle\id_c\colon \repD[0]\xrightarrow{\id_c}\varphi_!\int_\cC^N\Hom_\cC(-,c)\to \int_\cC^\Nh\Hom_\cC(-,c)\]
    is also a weak equivalence in $\MSrightfib{\Nh\cC}$, as desired.
\end{proof}

\begin{prop} \label{varphicwe}
    The $\Thnsset$-enriched natural transformations
    \[ \textstyle(\varepsilon_\cC)_!\St_{\Nh\cC}(\id_c)\colon \Hom_\cC(-,c)\cong (\varepsilon_\cC)_!\St_{\Nh\cC}(c)\to (\varepsilon_\cC)_!\St_{\Nh\cC} \int_\cC^\Nh \Hom_\cC(-,c) \]
    and 
    \[ \textstyle\varphi_c\colon (\varepsilon_\cC)_!\St_{\Nh\cC} \int_\cC^\Nh \Hom_\cC(-,c)\to \Hom_\cC(-,c)\]
    are weak equivalences in $[\cC^{\op},\MSThnsset]_\proj$.
\end{prop}

\begin{proof}
    Since $(\varepsilon_\cC)_!\St_{\Nh\cC}\colon \MSrightfib{\Nh\cC}\to [\cC^{\op},\MSThnsset]_\proj$ preserves weak equivalences by \cref{LKEisQP,QPwithloc}, the fact that $(\varepsilon_\cC)_!\St_{\Nh\cC}(\id_c)$ is a weak equivalence in $[\cC^{\op},\MSThnsset]_\proj$ follows directly from \cref{Intciswe}. Then by $2$-out-of-$3$ we get that $\varphi_c$ is also a weak equivalence in $[\cC^{\op},\MSThnsset]_\proj$, as it is a retraction of $(\varepsilon_\cC)_!\St_{\Nh\cC}(\id_c)$ by \cref{natphic}.
\end{proof}

\begin{rmk} \label{cofibranthtpycol}
    Suppose that $\cM$ is a combinatorial model category and that $\cI$ is a set of generating cofibrations between cofibrant objects. Write $\cI_0=\{A,B\mid A\to B\in \cJ\}\subseteq \Ob\cM$. Recall that every cofibrant object in $\cM$ can be obtained as a retract of transfinite compositions of pushouts along generating cofibrations in $\cI$. Since pushouts of cofibrant objects along cofibrations and transfinite compositions of cofibrations between cofibrant objects are homotopy pushouts by \cite[Propositions 14.5 and 14.10]{duggerhocolimnotes}, this tells us that every cofibrant object is a retract of a homotopy colimit of objects in $\cI_0$. 
\end{rmk}

We are now ready to show the desired result.

\begin{prop} \label{prop:wetoid}
    The following diagram of left Quillen functors commutes up to a natural weak equivalence.
    \begin{tz}
\node[](1) {$[\cC^{\op},\MSThnsset]_\proj$}; 
\node[below of=1](2) {$[\cC^{\op},\MSThnsset]_\proj$}; 
\node[right of=1,xshift=3.7cm](3) {$\MSrightfib{\Nh\cC}$}; 
\node[below of=3](4) {$[\Ch(\Nh\cC)^{\op},\MSThnsset]_\proj$};

\draw[->] (1) to node[above,la]{$\int^\Nh_\cC$} (3); 
\draw[->] (3) to node[right,la]{$\St_{\Nh\cC}$} (4); 
\draw[->] (4) to node[below,la]{$(\varepsilon_\cC)_!$} (2);
\draw[->] (1) to node[left,la]{$\id$} (2);

\cell[la,above,xshift=3pt][n][.5]{a}{b}{$\simeq$};
\end{tz}
\end{prop}

\begin{proof} 
    First note that the functors in the diagram are all left Quillen by \cref{LKEisQP,intQP,QPwithloc}. Hence, it is enough to show that there is a weak equivalence in $[\cC^{\op},\MSThnsset]_\proj$ \[ \textstyle (\varepsilon_\cC)_!\St_{\Nh\cC} \int_\cC^\Nh F\to F \]
    at any cofibrant object $F\colon \cC^{\op}\to\Thnsset$ in $[\cC^{\op},\MSThnsset]_\proj$ that is natural in $F$. However, by combining \cref{gencofprojective,cofibranthtpycol}, every cofibrant object is a retract of a homotopy colimit of objects of the form $\Hom_\cC(-,c)\otimes X$, with $c\in \cC$ an object and $X\in \Thnsset$, and so it is enough to show that, for every object $c\in \cC$ and every object $X\in \Thnsset$, there is a weak equivalence in~$[\cC^{\op},\MSThnsset]_\proj$
    \[ \textstyle\varphi_{c,X}\colon (\varepsilon_\cC)_! \St_{\Nh\cC} (\int^\Nh_\cC\Hom_\cC(-,c)\otimes X)\to \Hom_\cC(-,c)\otimes X \]
    that is natural in $c$ and $X$. As all functors involved preserve tensors by \cref{subsec:MSproj,lem:intprestensors,lem:Stprestensors}, we define $\varphi_{c,X}$ to be
    \[ \textstyle (\varepsilon_\cC)_! \St_{\Nh\cC} (\int^\Nh_\cC\Hom_\cC(-,c)\otimes X)\cong (\varepsilon_\cC)_! \St_{\Nh\cC} (\int^\Nh_\cC\Hom_\cC(-,c))\otimes X\xrightarrow{\varphi_c\otimes X} \Hom_\cC(-,c)\otimes X. \]
    Then $\varphi_{c,X}$ is clearly natural in $X\in \Thnsset$ and it is natural in $c\in \cC$ by \cref{natphic}. Moreover, as the model structure $[\cC^{\op},\MSThnsset]_\proj$ is enriched over $\MSThnsset$ by \cref{MSprojenriched} and $\varphi_c$ is a weak equivalence in $[\cC^{\op},\MSThnsset]_\proj$ by \cref{varphicwe}, then so is $\varphi_{c,X}=\varphi_c\otimes X$, which concludes the proof. 
\end{proof}

From the commutativity of the above square at the level of homotopy categories, we can deduce the following. 

\begin{thm} \label{prop:StNCQE}
Let $\cC$ be a fibrant $\MSThnsset$-enriched category. Then the Quillen pair $\St_{\Nh\cC}\dashv \Un_{\Nh\cC}$ is a Quillen equivalence enriched over $\MSThnsset$
\begin{tz}
\node[](1) {$\MSrightfib{\Nh\cC}$}; 
\node[right of=1,xshift=4.2cm](2) {$[\Ch(\Nh\cC)^{\op},\MSThnsset]_\proj$};
\punctuation{2}{.};

\draw[->] ($(2.west)-(0,5pt)$) to node[below,la]{$\Un_{\Nh\cC}$} ($(1.east)-(0,5pt)$);
\draw[->] ($(1.east)+(0,5pt)$) to node[above,la]{$\St_{\Nh\cC}$} ($(2.west)+(0,5pt)$);

\node[la] at ($(1.east)!0.5!(2.west)$) {$\bot$};
\end{tz}
\end{thm}

\begin{proof}
    We have a triangle of left Quillen functors
    \begin{tz}
\node[](1) {$[\cC^{\op},\MSThnsset]_\proj$}; 
\node[below of=1](2) {$[\cC^{\op},\MSThnsset]_\proj$}; 
\node[right of=1,xshift=3.7cm](3) {$\MSrightfib{\Nh\cC}$}; 
\node[below of=3](4) {$[\Ch(\Nh\cC)^{\op},\MSThnsset]_\proj$};

\draw[->] (1) to node[above,la]{$\int^\Nh_\cC$} (3); 
\draw[->] (3) to node[right,la]{$\St_{\Nh\cC}$} (4); 
\draw[->] (4) to node[below,la]{$(\varepsilon_\cC)_!$} (2);
\draw[->] (1) to node[left,la]{$\id$} (2);

\cell[la,above,xshift=3pt][n][.5]{a}{b}{$\simeq$};
\end{tz}
which commutes up to isomorphism at the level of homotopy categories by \cref{prop:wetoid}. Since $\cC$ is fibrant in $\MSThncat$, the $\Thnsset$-enriched functor $\varepsilon_\cC\colon \Ch\Nh\cC\to \cC$ is a weak equivalence in $\MSThncat$ as it is the component of the derived counit of the Quillen equivalence from \cref{injQuillen}. Hence the functor $(\varepsilon_\cC)_!$ is a Quillen equivalence by \cref{LKEisQP}. Moreover, by \cref{prop:intisQE} the functor $\int_\cC^\Nh$ is a Quillen equivalence. Hence, by $2$-out-of-$3$, we conclude that the functor $\St_{\Nh\cC}$ is also a Quillen equivalence, as desired.
\end{proof}

We are now ready to deduce that, for any object $W\in \pcatThn$, the straightening functor $\St_W\colon [\Ch W^{\op},\MSThnsset]_\proj\to \MSrightfib{W}$ is a Quillen equivalence. For this, we use the following change of base lemma. 

\begin{lemma} \label{StWQEStZ}
    Let $f\colon W\to Z$ be a weak equivalence in $\pcatinj$. The Quillen pair $\St_W\dashv \Un_W$ is a Quillen equivalence
\begin{tz}
\node[](1) {$\MSrightfib{W}$}; 
\node[right of=1,xshift=3.9cm](2) {$[\Ch W^{\op},\MSThnsset]_\proj$}; 

\draw[->] ($(2.west)-(0,5pt)$) to node[below,la]{$\Un_{W}$} ($(1.east)-(0,5pt)$);
\draw[->] ($(1.east)+(0,5pt)$) to node[above,la]{$\St_{W}$} ($(2.west)+(0,5pt)$);

\node[la] at ($(1.east)!0.5!(2.west)$) {$\bot$};
\end{tz}
if and only if the Quillen pair $\St_Z\dashv \Un_Z$ is a Quillen equivalence.
\begin{tz}
\node[](1) {$\MSrightfib{Z}$}; 
\node[right of=1,xshift=3.9cm](2) {$[\Ch Z^{\op},\MSThnsset]_\proj$}; 

\draw[->] ($(2.west)-(0,5pt)$) to node[below,la]{$\Un_{Z}$} ($(1.east)-(0,5pt)$);
\draw[->] ($(1.east)+(0,5pt)$) to node[above,la]{$\St_{Z}$} ($(2.west)+(0,5pt)$);

\node[la] at ($(1.east)!0.5!(2.west)$) {$\bot$};
\end{tz}
\end{lemma}

\begin{proof}
    First, note that, since $\Ch\colon \pcatinj\to \MSThncat$ preserves weak equivalences by \cref{injQuillen}, the $\Thnsset$-enriched functor $\Ch f\colon \Ch W\to \Ch Z$ is a weak equivalence in $\MSThncat$. Then, by \cref{basechangeSt}, we have a commutative square of left Quillen functors 
\begin{tz}
        \node[](1) {$\sThnssetslice{W}$}; 
        \node[right of=1,xshift=2.8cm](2) {$[\Ch W^{\op},\Thnsset]$}; 
        \node[below of=1](3) {$\sThnssetslice{Z}$}; 
        \node[below of=2](4) {$[\Ch Z^{\op},\Thnsset]$}; 

        \draw[->] (1) to node[above,la]{$\St_W$} (2); 
        \draw[->] (1) to node[left,la]{$f_!$} (3); 
        \draw[->] (2) to node[right,la]{$(\Ch f)_!$} (4); 
        \draw[->] (3) to node[below,la]{$\St_Z$} (4);
         \cell[la,above][n][.5]{2}{3}{$\cong$};
    \end{tz}
    where $f_!$ and $(\Ch f)_!$ are Quillen equivalences by \cref{postcompisQP,LKEisQP}. Hence, by $2$-out-of-$3$, we conclude that the functor $\St_W$ is a Quillen equivalence if and only if the functor $\St_Z$ is a Quillen equivalence.
\end{proof}

\begin{thm} \label{stunmainQE}
Let $W$ be an object in $\pcatThn$. The Quillen pair $\St_W\dashv \Un_W$ is a Quillen equivalence enriched over $\MSThnsset$
\begin{tz}
\node[](1) {$\MSrightfib{W}$}; 
\node[right of=1,xshift=3.9cm](2) {$[\Ch W^{\op},\MSThnsset]_\proj$}; 
\punctuation{2}{.};

\draw[->] ($(2.west)-(0,5pt)$) to node[below,la]{$\Un_{W}$} ($(1.east)-(0,5pt)$);
\draw[->] ($(1.east)+(0,5pt)$) to node[above,la]{$\St_{W}$} ($(2.west)+(0,5pt)$);

\node[la] at ($(1.east)!0.5!(2.west)$) {$\bot$};
\end{tz}
\end{thm}

\begin{proof}
    The component of the derived unit of the Quillen equivalence from \cref{injQuillen} at $W$ gives a weak equivalence $W\to \Nh\widehat{\Ch W}$ in $\pcatinj$, where $\Ch W\to \widehat{\Ch W}$ is a fibrant replacement in $\MSThncat$. Then, by \cref{prop:StNCQE} the functor $\St_{\Nh\widehat{\Ch W}}$ is a Quillen equivalence, and so by \cref{StWQEStZ} we get that the functor $\St_W$ is also a Quillen equivalence. 
\end{proof}

\appendix

\section{Straightening-unstraightening: technical results}

In this section, we prove the technical results stated in \cref{subsec:propofSt}. In \cref{sec:naturalityofSt}, we prove \cref{basechangeSt}; in \cref{sec:Stofgeneral}, \cref{Stof[lX]}; in \cref{sec:enrichmentofSt}, \cref{lem:Stprestensors}; and, finally, in \cref{sec:Stofmono}, \cref{Stofalphafmono}.

\subsection{Naturality of straightening} \label{sec:naturalityofSt}

We aim to show \cref{basechangeSt} which states that the straightening functor is natural in $W$. For this, we first need the following result. Recall the notion of a cofinal functor from \cite[\textsection IX.3]{MacLane}.

\begin{lemma} \label{prop:BeckChevalley1}
    Consider a commutative square in $\cat$ as below left.
    \begin{tz}
        \node[](1) {$\cA$}; 
        \node[below of=1](2) {$\cC$}; 
        \node[right of=1,xshift=.5cm](3) {$\cB$}; 
        \node[below of=3](4) {$\cD$}; 

        \draw[->] (1) to node[left,la]{$I$} (2); 
        \draw[->] (1) to node[above,la]{$F$} (3); 
        \draw[->] (2) to node[below,la]{$G$} (4); 
        \draw[->] (3) to node[right,la]{$J$} (4);

        \node[right of=3,xshift=2cm](1) {$\set^{\cA^{\op}}$}; 
        \node[below of=1](2) {$\set^{\cC^{\op}}$}; 
        \node[right of=1,xshift=1.3cm](3) {$\set^{\cB^{\op}}$}; 
        \node[below of=3](4) {$\set^{\cD^{\op}}$};

        \cell[][n][.5]{1}{4}{};

        \draw[->] (2) to node[left,la]{$I^*$} (1); 
        \draw[->] (1) to node[above,la]{$F_!$} (3); 
        \draw[->] (2) to node[below,la]{$G_!$} (4); 
        \draw[->] (4) to node[right,la]{$J^*$} (3);
    \end{tz}
    Suppose that, for every object $b\in \cB$, the induced functor between comma categories
    \[ J\downarrow I\colon b\downarrow F\to Jb\downarrow G, \quad (a\in \cA, b\xrightarrow{f} Fa\in \cB)\mapsto (Ia\in \cC,Jb\xrightarrow{Jf} JFa=GIa\in \cD)\]
    is cofinal. Then the canonical natural transformation in the above right square in $\cat$ is a natural isomorphism $F_! I^*\cong J^* G_!$, i.e., the Beck-Chevalley condition is satisfied.
\end{lemma}

\begin{proof}
    Let $H\colon \cC^{\op}\to \set$ be a functor and $b\in \cB$ be an object. By the pointwise formula for left Kan extension, we have natural isomorphisms in $\cE$ 
    \[ J^* G_! H(b)=G_! H(Jb)\cong \colim((Jb\downarrow G)^{\op}\to \cC^{\op}\xrightarrow{H} \set), \]
    \begin{align*}
        F_! I^* H(b) =F_! (H\circ I^{\op})(b) &\cong \colim((b\downarrow F)^{\op}\to \cA^{\op}\xrightarrow{I^{\op}} \cC^{\op}\xrightarrow{H} \set) \\
        &\cong \colim((b\downarrow F)^{\op}\xrightarrow{(J\downarrow I)^{\op}}(Jb\downarrow G)^{\op}\to \cC^{\op}\xrightarrow{H} \set).
    \end{align*} 
    Hence the result follows from the fact that the functor $(J\downarrow I)^{\op}\colon (b\downarrow F)^{\op}\to (Jb\downarrow G)^{\op}$ is final as its opposite $J\downarrow I\colon b\downarrow F\to Jb\downarrow G$ is cofinal by assumption.
\end{proof}

Recall that $\Thncat$ is a full subcategory of $\cat^{\DThnop}$. Hence a $\Thnsset$-enriched category $\cC$ can be thought of as a functor $\cC\colon \Thnop\times \Dop\to \cat$, and a $\Thnsset$-enriched functor $F\colon \cC\to \cD$ as a natural transformation between such functors.

\begin{lemma} \label{prop:BeckChevalley2}
    Consider a commutative square in $\Thncat$ as below left.
    \begin{tz}
        \node[](1) {$\cA$}; 
        \node[below of=1](2) {$\cC$}; 
        \node[right of=1,xshift=.5cm](3) {$\cB$}; 
        \node[below of=3](4) {$\cD$}; 

        \draw[->] (1) to node[left,la]{$I$} (2); 
        \draw[->] (1) to node[above,la]{$F$} (3); 
        \draw[->] (2) to node[below,la]{$G$} (4); 
        \draw[->] (3) to node[right,la]{$J$} (4);

        \node[right of=3,xshift=2cm](1) {$[\cA^{\op},\Thnsset]$}; 
        \node[below of=1](2) {$[\cC^{\op},\Thnsset]$}; 
        \node[right of=1,xshift=2.7cm](3) {$[\cB^{\op},\Thnsset]$}; 
        \node[below of=3](4) {$[\cD^{\op},\Thnsset]$};

        \cell[][n][.5]{1}{4}{};

        \draw[->] (2) to node[left,la]{$I^*$} (1); 
        \draw[->] (1) to node[above,la]{$F_!$} (3); 
        \draw[->] (2) to node[below,la]{$G_!$} (4); 
        \draw[->] (4) to node[right,la]{$J^*$} (3);
    \end{tz}
    Suppose that, for every object $b\in \cB$ and every $\defThn\in \Thn$, $\defS\geq 0$, the induced functor between comma categories 
    \[ J_{\defThn,\defS}\downarrow I_{\defThn,\defS}\colon b\downarrow F_{\defThn,\defS}\to Jb\downarrow G_{\defThn,\defS}\]
    is cofinal. Then the canonical natural transformation in the above right square in $\cat$ is a natural isomorphism $F_! I^*\cong J^* G_!$, i.e., the Beck-Chevalley condition is satisfied.
\end{lemma}

\begin{proof} 
    Let $H\colon \cC^{\op}\to \Thnsset$ be a $\Thnsset$-enriched functor and $b\in \cB$. By using the pointwise formula for $\Thnsset$-enriched left Kan extensions from \cite[(4.18) and (4.25)]{Kelly} and applying the evaluation functor $\ev_{\defThn,\defS}\colon \Thnsset\to \sset$ for every $\defThn\in \Thn$ and $\defS\geq 0$, we see that there are natural isomorphisms of sets 
    \[ \ev_{\defThn,\defS}(F_! I^* H(b))\cong \ev_{\defThn,\defS}((F_{\defThn,\defS})_!I_{\defThn,\defS}^*H_{\defThn,\defS}(b)), \quad \ev_{\defThn,\defS}(J^* G_!H(b))\cong \ev_{\defThn,\defS}(J_{\defThn,\defS}^* (G_{\defThn,\defS})_! H_{\defThn,\defS}(b)). \]
    Hence the result follows from applying \cref{prop:BeckChevalley1} and using that isomorphisms in $\Thnsset$ are levelwise.
\end{proof}

Let us fix an object $W\in \pcatThn$ and a map $\sigma\colon \repD[m,\defThn,\defS]\to W$ in $\sThnsset$ with $m,\defS\geq 0$ and $\defThn\in \Thn$. Consider the following pushout square in $\Thncat$.
\begin{tz}
        \node[](1) {$\Ch L\repD[m,\defThn,\defS]$}; 
        \node[right of=1,xshift=2cm](2) {$\Ch W$}; 
        \node[below of=1](3) {$\Ch L\repD[m+1,\defThn,\defS]$}; 
        \node[below of=2](4) {$\Ch W_\sigma$}; 
        \pushout{4};

        \draw[->] (1) to node[above,la]{$\Ch \sigma$} (2); 
        \draw[->] (1) to node[left,la]{$\Ch L [d^{m+1},\defThn,\defS]$} (3); 
        \draw[->] (2) to node[right,la]{$\Ch \iota_\sigma$} (4); 
        \draw[->] (3) to node[below,la]{$\Ch \sigma'$} (4);
    \end{tz}

\begin{lemma} \label{cofinality}
    For $\defThn'\in \Thn$, $\defS'\geq 0$, and $b\in \Ch W$ an object, the functor 
    \[ (\Ch\iota_\sigma)_{\defThn',\defS'}\downarrow(\Ch L[d^{m+1},\defThn,\defS])_{\defThn',\defS'}\colon b\downarrow (\Ch\sigma)_{\defThn',\defS'}\to (\Ch\iota_\sigma) b \downarrow (\Ch\sigma')_{\defThn',\defS'}\]
    is cofinal.
\end{lemma}

\begin{proof} 
    Let $0\leq i\leq m+1$ and $f\colon (\Ch \iota_\sigma) b\to (\Ch \sigma')i$ be a morphism in $(\Ch W_\sigma)_{\defThn',\defS'}$. We need to show that the comma category $((\Ch\iota_\sigma)_{\defThn',\defS'}\downarrow(\Ch L[d^{m+1},\defThn,\defS])_{\defThn',\defS'})\downarrow (i,f)$ is non-empty and connected. To see this, we will show the following: 
    \begin{enumerate}[leftmargin=0.6cm,label=(\roman*)]
        \item there exist $0\leq j\leq m$, a morphism $g\colon b\to (\Ch\sigma)j$ in $(\Ch W)_{\defThn',\defS'}$, and a morphism $\varphi\colon j\to i$ in $(\Ch L\repD[m+1,\defThn,\defS])_{\defThn',\defS'}$ such that $f=(\Ch \sigma')\varphi\circ (\Ch \iota_\sigma) g$, 
        \item given another such factorization $(j',g',\varphi')$ of $f$, then there exists a morphism $\beta\colon j'\to j$ in $(\Ch L\repD[m,\defThn,\defS])_{\defThn',\defS'}$ such that $\varphi \circ \beta=\varphi'$ and $g=(\Ch \sigma)\beta\circ g'$.
    \end{enumerate}

    When $0\leq i\leq m$, as the functor $(\Ch \iota_\sigma)_{\defThn',\defS'}\colon (\Ch W)_{\defThn',\defS'}\to (\Ch W_\sigma)_{\defThn',\defS'}$ is fully faithful as a pushout of a fully faithful functor, there exists a morphism $g\colon b\to (\Ch\sigma)i$ such that $f=(\Ch \iota_\sigma)g$, and hence the tuple $(i,g,\id_i)$ gives a factorization of $f$ as in (i). It is in fact a terminal object in the comma category $((\Ch\iota_\sigma)_{\defThn',\defS'}\downarrow(\Ch L[d^{m+1},\defThn,\defS])_{\defThn',\defS'})\downarrow (i,f)$ and so it satisfies (ii) as well.

    It remains to consider the case where $i=m+1$. In order to show (i) and (ii), using the same notations as above, we will construct a commutative diagram in $(\Ch W_\sigma)_{\defThn',\defS'}$ of the following form.
    \begin{tz}
    \node[](1) {$(\Ch\iota_\sigma)b$}; 
    \node[above of=1,xshift=4cm](2) {$(\Ch \iota_\sigma)(\Ch \sigma)j'=(\Ch\sigma')j'$}; 
    \node[below of=2,xshift=3.7cm](3) {$\top$}; 
    \node[below of=1,xshift=4cm](4) {$(\Ch\iota_\sigma)(\Ch\sigma)j=(\Ch\sigma')j$}; 
    
    \draw[->] (1) to node[above,la,xshift=-1.2cm]{$(\Ch \iota_\sigma)g'=(T_1,\gamma_1)$} (2); 
    \draw[->] (2) to node[above,la,xshift=1.2cm]{$(\Ch\sigma')\varphi'=(T_2,\gamma_2)$} (3); 
    \draw[->] (1) to node[below,la,xshift=-1.2cm]{$(\Ch \iota_\sigma)g=(T',\gamma')$} (4); 
    \draw[->] (4) to node[below,la,xshift=1.2cm]{$(\Ch\sigma')\varphi=(B_\omega^T,\gamma_\omega)$} (3); 
    \draw[->] (2) to node[over,la,pos=0.7]{$(\Ch\iota_\sigma)(\Ch\sigma)\beta=(T_1',\gamma_1')$} (4);
    \draw[->,w](1) to node[above,la,pos=0.65]{$f=(T,\gamma)$} (3);
\end{tz}
    
    Recall from \cref{cor:computationhomsC} that 
    \[ \Hom_{\Ch W_\sigma}((\Ch\iota_\sigma) b,\top)\cong \diag(\colim_{T\in \catnec{(W_\sigma)_{-,\star,\star}}{(\Ch\iota_\sigma) b}{\top}} \Hom_{\CL T}(\alpha,\omega)), \]
    where $\top=(\Ch \sigma')(m+1)$. Hence a morphism $f\colon (\Ch \iota_\sigma)b\to \top$ in $(\Ch W_\sigma)_{-,\defThn',\defS'}$, i.e., an element of $\Hom_{\Ch W_\sigma}((\Ch\iota_\sigma) b,\top)_{\defThn',\defS'}$, consists of a pair $(T,\gamma)$ of a necklace $T\to ((W_\sigma)_{\defThn',\defS'})_{(\Ch \iota_\sigma)b,\top}$ and an element $\gamma\in \Hom_{\CL T}(\alpha,\omega)_{\defS'}$. 

    Now note that, if $\tau\colon \repD[m']\to (W_\sigma)_{-,\defThn',\defS'}$ is such that $\tau(m')=\top$, then by definition of $(W_\sigma)_{-,\defThn',\defS'}$ as the pushout $(L\repD[m+1,\defThn,\defS]))_{-,\defThn',\defS'}\amalg_{(L\repD[m,\defThn,\defS])_{-,\defThn',\defS'}} W_{-,\defThn',\defS'}$ in $\Dset$, the $m'$-simplex $\tau$ factors through $\sigma'_{-,\defThn',\defS'}\colon (L\repD[m+1,\defThn,\defS]))_{-,\defThn',\defS'}\to (W_\sigma)_{-,\defThn',\defS'}$. In particular, the last bead $B_\omega^T\hookrightarrow T\to ((W_\sigma)_{\defThn',\defS'})_{(\Ch \iota_\sigma)b,\top}$ satisfies this condition and so it factors through a simplex $B_\omega^T\to (L\repD[m+1,\defThn,\defS]_{-,\defThn',\defS'})_{j,m+1}$ for some $0\leq j\leq m$.

    We write $T=T'\vee B_\omega^T$. Then by \cite[Corollary 3.10]{DuggerSpivakRigidification}, there is an isomorphism of sets 
    \begin{equation} \label{splithom} \Hom_{\CL T}(\alpha,\omega)_{\defS'}\cong \Hom_{\CL T'}(\alpha,\omega)_{\defS'}\times \Hom_{\CL B_\omega^T}(\alpha,\omega)_{k'} 
    \end{equation}
    and so the element $\gamma\in \Hom_{\CL T}(\alpha,\omega)_{\defS'}$ corresponds to a pair $(\gamma',\gamma_\omega)$ with $\gamma'\in \Hom_{\CL T'}(\alpha,\omega)_{\defS'}$ and $\gamma_\omega\in \Hom_{\CL B_\omega^T}(\alpha,\omega)_{k'}$. Using \cref{cor:computationhomsC}, the pair $(T',\gamma')$ corresponds to an element of $\Hom_{\Ch W_\sigma}((\Ch \iota_\sigma)b,(\Ch \iota_\sigma)(\Ch \sigma)j)_{\defThn',\defS'}$, and so to an element $g\in \Hom_{\Ch W}(b,(\Ch \sigma)j)_{\defThn',\defS'}$ by fully faithfulness of $(\Ch \iota_\sigma)_{\defThn',\defS'}$, and the pair $(B_\omega^T,\gamma_\omega)$ to an element of $\varphi\in \Hom_{\Ch L\repD[m+1,\defThn,\defS]}(j,m+1)_{\defThn',\defS'}$ by the above discussion. Moreover, as $T=T'\vee B_{\omega}^T$ and $\gamma$ is sent to $(\gamma',\gamma_\omega)$ by the isomorphism~\eqref{splithom}, we get that $f=(\Ch \sigma')\varphi\circ (\Ch \iota_\sigma) g$, proving (i).

    Now, if $(j',g',\varphi')$ is another factorization of $f$ of the desired form, then using \cref{cor:computationhomsC} the morphism $g'$ consists of a pair $(T_1,\gamma_1)$ of a necklace $T_1\to (W_{-,\defThn',\defS'})_{b,(\Ch\sigma)j'}$ and an element $\gamma_1\in \Hom_{\CL T_1}(\alpha,\omega)_{\defS'}$ and the morphism $\varphi'$ consists of a pair $(T_2,\gamma_2)$ of a necklace $T_2\to ( L\repD[m+1,\defThn,\defS]_{-,\defThn',\defS'})_{j',m+1}$ and an element $\gamma_2\in \Hom_{\CL T_2}(\alpha,\omega)_{\defS'}$. Moreover, as $f=(\Ch\sigma')\varphi'\circ (\Ch\iota_\sigma) g'$, we get that $T=T_1\vee T_2$ and that the pair $(\gamma_1,\gamma_2)$ corresponds to $\gamma\in \Hom_{\CL T}(\alpha,\omega)_{\defS'}$ under the isomorphism of sets from \cite[Corollary 3.10]{DuggerSpivakRigidification}
    \begin{equation} \label{splithom2} \Hom_{\CL T}(\alpha,\omega)_{\defS'}\cong \Hom_{\CL T_1}(\alpha,\omega)_{\defS'}\times \Hom_{\CL T_2}(\alpha,\omega)_{k'}. \end{equation}
    
    As $T'\hookrightarrow T$ is by definition the subnecklace of $T$ containing all beads of $T$ except for $B_\omega^T$, we get that $T_1\hookrightarrow T'$. Write $T'=T_1\vee T_1'$. As $T_1\vee T_1'\vee B_\omega^T=T'\vee B_\omega^T=T=T_1\vee T_2$ we also have that $T_2=T_1'\vee B_\omega^T$. Then by \cite[Corollary 3.10]{DuggerSpivakRigidification} there are isomorphisms of sets 
    \begin{align*} 
    \Hom_{\CL T'}(\alpha,\omega)_{\defS'}\times \Hom_{\CL B_\omega^T}(\alpha,\omega)_{\defS'} &\cong \Hom_{\CL T_1}(\alpha,\omega)_{\defS'}\times \Hom_{\CL T'_1}(\alpha,\omega)_{k'}\times \Hom_{\CL B_\omega^T}(\alpha,\omega)_{\defS'} \\
    &\cong \Hom_{\CL T_1}(\alpha,\omega)_{\defS'}\times \Hom_{\CL T_2}(\alpha,\omega)_{\defS'}, \end{align*}
    and so there is a unique element $\gamma'_1\in \Hom_{\CL T'_1}(\alpha,\omega)_{k'}$ such that the pair $(\gamma_1,\gamma'_1)$ corresponds to $\gamma'$ and the pair $(\gamma'_1,\gamma_\omega)$ to $\gamma_2$, where we used that the pairs $(\gamma',\gamma_\omega)$ and $(\gamma_1,\gamma_2)$ both correspond to $\gamma\in \Hom_{\CL T}(\alpha,\omega)_{\defS'}$ under the isomorphisms \eqref{splithom} and \eqref{splithom2}, respectively. In particular, this implies that $\varphi \circ \beta=\varphi'$ and $g=(\Ch \sigma)\beta\circ g'$, proving (ii). 
\end{proof}

\begin{lemma} \label{lem:Stvssigma}
    There is an isomorphism in $[\Ch W^{\op},\Thnsset]$
    \[ (\Ch \sigma)_! \St_{L\repD[m,\defThn,\defS]}(\id_{\repD[m,\defThn,\defS]}) \cong \St_W(\sigma). \]
\end{lemma}

\begin{proof}
    First recall that, by the definition of the straightening functor on representables from \cref{subsec:defStUn}, we have the following isomorphism in $[\Ch L\repD[m,\defThn,\defS]^{\op},\Thnsset]$
    \[ \St_{L\repD[m,\defThn,\defS]}(\id_{\repD[m,\defThn,\defS]})\cong (\Ch L[d^{m+1},\defThn,\defS])^*\Hom_{\Ch L\repD[m+1,\defThn,\defS]}(-,m+1)\]
    and the following isomorphism in $[\Ch W^{\op},\Thnsset]$
    \[ \St_W(\sigma)=(\Ch \iota_\sigma)^* \Hom_{\Ch W_\sigma}(-,\top)\cong (\Ch \iota_\sigma)^* (\Ch \sigma')_! \Hom_{\Ch L\repD[m+1,\defThn,\defS]}(-,m+1), \]
    where we use that the left Kan extension of the representable $\Hom_{\Ch L\repD[m+1,\defThn,\defS]}(-,m+1)$ along $\Ch\sigma'\colon \Ch L\repD[m+1,\defThn,\defS]\to \Ch W_\sigma$ is the representable $\Hom_{\Ch W_\sigma}(-,\top)$ by \cite[(4.32)]{Kelly}. Then, by combining \cref{prop:BeckChevalley2,cofinality}, we get that the following square in $\cat$
    \begin{tz}
        \node[](1) {$[\Ch L\repD[m,\defThn,\defS]^{\op},\Thnsset]$}; 
        \node[below of=1](2) {$[\Ch L\repD[m+1,\defThn,\defS]^{\op},\Thnsset]$}; 
        \node[right of=1,xshift=4.2cm](3) {$[\Ch W^{\op},\Thnsset]$}; 
        \node[below of=3](4) {$[(\Ch W_\sigma)^{\op},\Thnsset]$}; 

        \draw[->] (2) to node[left,la]{$(\Ch L[d^{m+1},\defThn,\defS])^*$} (1); 
        \draw[->] (1) to node[above,la]{$(\Ch \sigma)_!$} (3); 
        \draw[->] (2) to node[below,la]{$(\Ch \sigma')_!$} (4); 
        \draw[->] (4) to node[right,la]{$(\Ch \iota_\sigma)^*$} (3);

         \cell[la,above][n][.5]{1}{4}{$\cong$};
    \end{tz}
    commutes up to a natural isomorphism $(\Ch \sigma)_!(\Ch L[d^{m+1},\defThn,\defS])^*\cong (\Ch \iota_\sigma)^* (\Ch \sigma')_!$. By taking the component of this isomorphism at the representable $\Hom_{\Ch L\repD[m+1,\defThn,\defS]}(-,m+1)$, we get isomorphisms in $[\Ch W^{\op},\Thnsset]$
    \begin{align*}
        (\Ch \sigma)_! \St_{L\repD[m,\defThn,\defS]}(\id_{\repD[m,\defThn,\defS]}) 
        & = (\Ch \sigma)_! (\Ch L[d^{m+1},\defThn,\defS])^*\Hom_{\Ch L\repD[m+1,\defThn,\defS]}(-,m+1) \\
        & \cong (\Ch \iota_\sigma)^* (\Ch \sigma')_! \Hom_{\Ch L\repD[m+1,\defThn,\defS]}(-,m+1)\cong \St_W(\sigma). \qedhere
    \end{align*}
\end{proof}

We are now ready to prove \cref{basechangeSt} 
which we restate for the reader's convenience.

\begin{prop} \label{basechangeStappendix}
Let $f\colon W\to Z$ be a map in $\pcatThn$. Then the following square of left adjoint functors commutes.
\begin{tz}
        \node[](1) {$\sThnssetslice{W}$}; 
        \node[right of=1,xshift=2.8cm](2) {$[\Ch W^{\op},\Thnsset]$}; 
        \node[below of=1](3) {$\sThnssetslice{Z}$}; 
        \node[below of=2](4) {$[\Ch Z^{\op},\Thnsset]$}; 

        \draw[->] (1) to node[above,la]{$\St_W$} (2); 
        \draw[->] (1) to node[left,la]{$f_!$} (3); 
        \draw[->] (2) to node[right,la]{$(\Ch f)_!$} (4); 
        \draw[->] (3) to node[below,la]{$\St_Z$} (4);

         \cell[la,above][n][.5]{2}{3}{$\cong$};
    \end{tz}
\end{prop}

\begin{proof}
Since all functors involved are left adjoints and there is an equivalence of categories $\sThnssetslice{W}\simeq \set^{(\int_{\DThnS} W)^{\op}}$, it is enough to show that, for every representable object $\sigma\colon \repD[m,\defThn,\defS]\to W$ in $\sThnssetslice{W}$, we have an isomorphism in $[\Ch Z^{\op},\Thnsset]$
\[ (\Ch f)_! \St_W(\sigma)\cong \St_Z (f_!\sigma)=\St_Z(f\sigma). \]
By applying \cref{lem:Stvssigma} once to $\sigma$ and once to the composite $f\sigma$ and using the functoriality of left Kan extensions and of the functor $\Ch$, we get isomorphisms in $[\Ch Z^{\op},\Thnsset]$
\begin{align*}
    (\Ch f)_! \St_W(\sigma) &\cong (\Ch f)_! (\Ch \sigma)_! \St_{L\repD[m,\defThn,\defS]}(\id_{\repD[m,\defThn,\defS]}) \\
    &\cong (\Ch (f\sigma ))_!\St_{L\repD[m,\defThn,\defS]}(\id_{\repD[m,\defThn,\defS]}) \cong \St_Z(f\sigma). \qedhere
\end{align*}
\end{proof}

\begin{rmk} 
 A similar proof of this argument in the case of quasi-categories can be found in \cite[Proposition 2.12]{hhr2021straightening}. While our proof is motivated by theirs, we do carefully separate our argument in a formal part (deducing the desired Beck-Chevalley condition from an abstract cofinality assumption in \cref{prop:BeckChevalley2}) and a computational part (using specific properties of necklaces to deduce the desired cofinality in \cref{cofinality}). This separation paves the way for possible future generalizations via appropriate computations.
\end{rmk}

\subsection{Straightening of \texorpdfstring{$\repD[\ell,X]\to L\repD[m,Y]$}{F[l,X]->LF[m,Y]}, general case} \label{sec:Stofgeneral}

We first prove \cref{Stof[lX]} which gives a formula for the straightening of a map $[\mapDelta,f]\colon \repD[\ell,X]\to L \repD[m,Y]$ in $\sThnsset$ with $\mapDelta\colon [\ell]\to [m]$ an injective map in $\Delta$ and $f\colon X\to Y$ a map in $\Thnsset$ between connected $\Thn$-spaces.

Recall the definition of the object $\Cone$ in $\pcatThn$ from \cref{notationCone}. We first want to understand necklaces in $\Cone$. To do so, we first rewrite $\Cone$ as a certain pushout in $\sThnsset$. 

\begin{lemma} \label{ConePushout}
    Let $\mapDelta\colon [\ell]\to [m]$ be an injective map in $\Delta$, and $f\colon X\to Y$ be a map in $\Thnsset$. Then $\Cone$ is the following pushout in $\sThnsset$. 
\begin{tz}
\node[](1) {$(\coprod_{m+1} Y)\amalg X$};
\node[below of=1](2) {$(\coprod_{m+1} \pi_0 Y)\amalg \pi_0 X$}; 
\node[right of=1,xshift=3.6cm](3) {$\repD[m,Y]\amalg_{\repD[\ell,X]} \repD[\ell+1,X]$}; 
\node[below of=3](4) {$\Cone$}; 
\pushout{4};

\draw[->] (1) to (3);
\draw[->] (1) to (2);
\draw[->] (3) to (4);
\draw[->] (2) to (4);
\end{tz}
\end{lemma}

\begin{proof}
This is an instance of pushouts commuting with pushouts, using \cref{RmkPushout}.
\end{proof}

We now fix an injective map $\mapDelta\colon [\ell]\to [m]$ in $\Delta$ and a map $f\colon X\to Y$ in $\Thnsset$ between connected $\Thn$-spaces.

\begin{lemma} \label{Cone0}
     There is an isomorphism in $\Thnsset$
     \[ \Cone_0\cong \{0,1,\ldots,m+1\}. \]
\end{lemma}

\begin{proof} 
    By applying the colimit-preserving functor $(-)_0\colon \sThnsset\to \Thnsset$ to the push\-out $\repD[m,Y]\amalg_{\repD[\ell,X]}\repD[\ell+1,X]$, we get an isomorphism in $\Thnsset$
    \[ \textstyle(\repD[m,Y]\amalg_{\repD[\ell,X]}\repD[\ell+1,X])_0\cong(\coprod_{m+1} Y)\amalg X. \]
    Now, by applying $(-)_0\colon \sThnsset\to \Thnsset$ to the pushout from \cref{ConePushout} and using that $\pi_0 Y\cong \pi_0 X\cong \{*\}$, we get isomorphisms in $\Thnsset$
    \[ \Cone_0\cong \{0,1,\ldots,m\}\amalg \{\ell+2\}\cong \{0,1,\ldots,m+1\}. \qedhere \]
\end{proof}

\begin{rmk} \label{ConePushoutThetak}
    For $\defThn\in \Thn$ and $\defS\geq 0$, by applying the colimit-preserving evaluation functor $(-)_{-,\defThn,\defS}\colon \sThnsset\to \Dset$ to the pushout from \cref{ConePushout}, we obtain the following pushout in $\Dset$. 
    \begin{tz}
\node[](1) {$(\coprod_{m+1} \coprod_{Y_{\defThn,\defS}}\repD[0])\amalg (\coprod_{X_{\defThn,\defS}} \repD[0])$};
\node[below of=1](2) {$\coprod_{m+2} \repD[0]$}; 
\node[right of=1,xshift=6.3cm](3) {$(\coprod_{Y_{\defThn,\defS}} \repD[m])\amalg_{\coprod_{X_{\defThn,\defS}} \repD[\ell]} (\coprod_{X_{\defThn,\defS}} \repD[\ell+1])$}; 
\node[below of=3](4) {$\Cone_{-,\defThn,\defS}$}; 
\pushout{4};

\draw[->] (1) to (3);
\draw[->] (1) to (2);
\draw[->] (3) to (4);
\draw[->] (2) to (4);
\end{tz}
\end{rmk}

\begin{notation}
    We write $Q\colon \Cone\to \repD[m+1]$ for the unique map in $\pcatThn$ making the following diagram commute.
    \begin{tz}
\node[](1) {$L\repD[\ell,X]$}; 
\node[below of=1](2) {$L\repD[\ell+1,X]$}; 
\node[right of=1,xshift=1.8cm](3) {$L\repD[m,Y]$}; 
\node[below of=3](4) {$\Cone$}; 
\node[below right of=4,xshift=1cm](5) {$\repD[m+1]$}; 
\pushout{4};

\draw[->] (1) to node[above,la]{$L[\mapDelta,f]$} (3);
\draw[->] (1) to node[left,la]{$L[d^{\ell+1},X]$} (2);
\draw[->] (3) to node[right,la]{$\iota_f$} (4);
\draw[->] (2) to (4);
\draw[->,bend left] (3) to node[right,la]{$L[d^{m+1},!]$} (5);
\draw[->,bend right=15] (2) to node[below,la,yshift=-3pt]{$L[\mapDelta+1,!]$} (5);
\draw[->,dashed] (4) to node[above,la,xshift=2pt]{$Q$} (5);
\end{tz}
\end{notation}

For $\defThn\in \Thn$, $\defS\geq 0$, and $0\leq i\leq m$, the map in $\Dset$ 
\[ Q_{-,\defThn,\defS}\colon \Cone_{-,\defThn,\defS}\to \repD[m+1]_{-,\defThn,\defS}=\repD[m+1] \]
induces by post-composition a functor 
\[ (Q_{-,\defThn,\defS})_!\colon \catnec{\Cone_{-,\defThn,\defS}}{i}{m+1}\to \catnec{\repD[m+1]}{i}{m+1}. \]
We show that it is a discrete fibration as defined e.g.~in \cite[Definition 2.1.1]{LRfib}.

\begin{prop} \label{prop:Qgendiscfib}
    For $\defThn\in \Thn$, $\defS\geq 0$, and $0\leq i\leq m$, the functor 
    \[ (Q_{-,\defThn,\defS})_!\colon \catnec{\Cone_{-,\defThn,\defS}}{i}{m+1}\to \catnec{\repD[m+1]}{i}{m+1} \]
    is a discrete fibration.
\end{prop}

\begin{proof}
    Given a necklace $T\to (\Cone_{-,\defThn,\defS})_{i,m+1}$, consider its image $T\to \repD[m+1]_{i,m+1}$ under $(Q_{-,\defThn,\defS})_!$ and let $f\colon U\to T$ be a map in $\catnec{\repD[m+1]}{i}{m+1}$. Then the composite
\[ U\xrightarrow{f} T\to (\Cone_{-,\defThn,\defS})_{i,m+1},\]
is the unique lift of $f$ via $(Q_{-,\defThn,\defS})_!$. Hence $(Q_{-,\defThn,\defS})_!$ is a discrete fibration.
\end{proof}

We now construct a functor that takes a necklace to a totally non-degenerate one in order to get a generalization of the bead functor from \cref{beadfunctor} to all necklaces, and we study the compatibility of the discrete fibration $(Q_{-,\defThn,\defS})_!$ with this construction.

\begin{rmk} \label{rem:DSepi}
    Let $K$ be a simplicial set and $a,b\in K_0$. As explained in \cite[p.~18]{DuggerSpivakRigidification}, for every necklace $T\to K_{a,b}$, there is a unique epimorphism of simplicial sets $T\twoheadrightarrow \overline{T}$ over~$K_{a,b}$ to a totally non-degenerate necklace $\overline{T}\to K_{a,b}$. Moreover, a map $g\colon U\to T$ in $\catnec{K}{a}{b}$ induces a unique map $\overline{g}\colon \overline{U}\to \overline{T}$ between the induced totally non-degenerate necklaces over $K_{a,b}$ making the following square in $\catnec{K}{a}{b}$ commute. 
    \begin{tz}
        \node[](1) {$U$}; 
        \node[below of=1](2) {$\overline{U}$}; 
        \node[right of=1](3) {$T$}; 
        \node[below of=3](4) {$\overline{T}$}; 
        \draw[->] (1) to node[above,la]{$g$} (3);
        \draw[->>] (1) to (2); 
        \draw[->>] (3) to (4); 
        \draw[->] (2) to node[below,la]{$\overline{g}$} (4);
    \end{tz}
    This yields a functor $\overline{(-)}\colon \catnec{K}{a}{b}\to \tndnec{K}{a}{b}$.
\end{rmk}

\begin{lemma} \label{Qnondegsimplex}
For $\defThn\in \Thn$ and $\defS\geq 0$, a simplex of $\Cone_{-,\defThn,\defS}$ is non-degenerate if and only if its image under the map $Q_{-,\defThn,\defS}\colon \Cone_{-,\defThn,\defS}\to \repD[m+1]$ in $\Dset$ is a non-degenerate simplex of $\repD[m+1]$.
\end{lemma}

\begin{proof}
    Let $\sigma\colon \repD[m']\to \Cone_{-,\defThn,\defS}$ be an $m'$-simplex of $\Cone_{-,\defThn,\defS}$. By the description of $\Cone_{-,\defThn,\defS}$ given in \cref{ConePushoutThetak}, if $m+1$ is in the image of $\sigma$, such an $m'$-simplex comes from an $m'$-simplex
    \[  \textstyle \repD[m']\xrightarrow{\overline{\sigma}} \{x\}\times \repD[\ell+1]\hookrightarrow(\coprod_{Y_{\defThn,\defS}} \repD) \amalg_{\coprod_{X_{\defThn,\defS}} \repD[\ell]} (\coprod_{X_{\defThn,\defS}} \repD[\ell+1]) \]
    for some $x\in X_{\defThn,\defS}$, and if $m+1$ is not in the image of $\sigma$, such an $m'$-simplex comes from an $m'$-simplex
    \[  \textstyle  \repD[m']\xrightarrow{\overline{\sigma}} \{y\}\times \repD[m]\hookrightarrow(\coprod_{Y_{\defThn,\defS}} \repD) \amalg_{\coprod_{X_{\defThn,\defS}} \repD[\ell]} (\coprod_{X_{\defThn,\defS}} \repD[\ell+1]) \]
    for some $y\in Y_{\defThn,\defS}$. These are sent by $Q_{-,\defThn,\defS}$ to $m'$-simplices
\[\textstyle \repD[m']\xrightarrow{\overline{\sigma}} \{x\}\times \repD[\ell+1]\xhookrightarrow{\mapDelta+1} \repD[m+1]\quad\text{ or }\quad\repD[m']\xrightarrow{\overline{\sigma}} \{y\}\times \repD[m]\xhookrightarrow{d^{m+1}} \repD[m+1]. \]
    Hence an $m'$-simplex $\sigma$ of $\Cone_{-,\defThn,\defS}$ is non-degenerate if and only if the corresponding $m'$-simplex $\overline{\sigma}$ in $\{x\}\times F[\ell+1]$ or $\{y\}\times F[m]$ is non-degenerate if and only if the image of $\sigma$ under $Q_{-,\defThn,\defS}$ is a non-degenerate simplex of $\repD[m+1]$.
\end{proof}

\begin{lemma}
    For $\defThn\in \Thn$, $\defS\geq 0$, and $0\leq i\leq m$, the functor 
    \[ (Q_{-,\defThn,\defS})_!\colon \catnec{\Cone_{-,\defThn,\defS}}{i}{m+1}\to \catnec{\repD[m+1]}{i}{m+1} \]
    restricts to a functor 
    \[(Q_{-,\defThn,\defS})_!\colon \tndnec{\Cone_{-,\defThn,\defS}}{i}{m+1}\to \tndnec{\repD[m+1]}{i}{m+1}. \]
\end{lemma}

\begin{proof}
    It suffices to show that the functor $(Q_{-,\defThn,\defS})_!$ sends a totally non-degenerate necklace in $(\Cone_{-,\defThn,\defS})_{i,m+1}$ to a totally non-degenerate necklace in $\repD[m+1]_{i,m+1}$. Since $(Q_{-,\defThn,\defS})_!$ is given by post-composition with the map $Q_{-,\defThn,\defS}\colon \Cone_{-,\defThn,\defS}\to \repD[m+1]$, this follows directly from \cref{Qnondegsimplex} as totally non-degenerate necklaces are precisely those necklaces whose beads are sent to non-degenerate simplices.
\end{proof}

\begin{prop} \label{Qcommutewithoverline}
    For $\defThn\in \Thn$, $\defS\geq 0$, and $0\leq i\leq m$, the following diagram in $\cat$ commutes. 
    \begin{tz}
        \node[](1) {$\catnec{\Cone_{-,\defThn,\defS}}{i}{m+1}$}; 
        \node[right of=1,xshift=4.3cm](2) {$\tndnec{\Cone_{-,\defThn,\defS}}{i}{m+1}$}; 
        \node[below of=1](3) {$\catnec{\repD[m+1]}{i}{m+1}$}; 
        \node[below of=2](4) {$\tndnec{\repD[m+1]}{i}{m+1}$}; 

        \draw[->] (1) to node[above,la]{$\overline{(-)}$} (2); 
        \draw[->] (1) to node[left,la]{$(Q_{-,\defThn,\defS})_!$} (3); 
        \draw[->] (2) to node[right,la]{$(Q_{-,\defThn,\defS})_!$} (4); 
        \draw[->] (3) to node[below,la]{$\overline{(-)}$} (4);
    \end{tz}
\end{prop}

\begin{proof}
Let $T=\repD[m_1]\vee\ldots\vee\repD[m_t]\to (\Cone_{-,\defThn,\defS})_{i,m+1}$ be a necklace. Then the (unique) totally non-degenerate necklace $\overline{T}\to (\Cone_{-,\defThn,\defS})_{i,m+1}$ is obtained by replacing each bead $\sigma_i\colon \repD[m_i]\hookrightarrow T\to \Cone_{-,\defThn,\defS}$ by the corresponding non-degenerate $m'_i$-simplex $\sigma'_i$ obtained by factoring $\sigma_i$ \[ \repD[m_i]\twoheadrightarrow \repD[m'_i]\xrightarrow{\sigma'_i} \Cone_{-,\defThn,\defS} \]
as an epimorphism followed by a non-degenerate simplex, for all $1\leq i\leq t$. By the description of $\Cone_{-,\defThn,\defS}$ given in \cref{ConePushoutThetak}, if $m+1$ is in the image of $\sigma_i$ and so also in the image of $\sigma'_i$, then $\sigma_i$ and $\sigma'_i$ come from simplices
    \[  \textstyle \repD[m_i]\twoheadrightarrow\repD[m'_i]\xrightarrow{\overline{\sigma_i}'} \{x_i\}\times \repD[\ell+1]\hookrightarrow(\coprod_{Y_{\defThn,\defS}} \repD) \amalg_{\coprod_{X_{\defThn,\defS}} \repD[\ell]} (\coprod_{X_{\defThn,\defS}} \repD[\ell+1]) \]
    for some $x_i\in X_{\defThn,\defS}$, and if $m+1$ is not in the image of $\sigma_i$ and so also not in the image of $\sigma'_i$, then $\sigma_i$ and $\sigma'_i$ come from simplices
    \[  \textstyle  \repD[m_i]\twoheadrightarrow\repD[m'_i]\xrightarrow{\overline{\sigma_i}'} \{y_i\}\times \repD[m]\hookrightarrow(\coprod_{Y_{\defThn,\defS}} \repD) \amalg_{\coprod_{X_{\defThn,\defS}} \repD[\ell]} (\coprod_{X_{\defThn,\defS}} \repD[\ell+1]) \]
    for some $y_i\in Y_{\defThn,\defS}$. These are sent by $Q_{-,\defThn,\defS}$ to simplices
\[ \repD[m_i]\twoheadrightarrow\repD[m'_i]\xrightarrow{\overline{\sigma_i}'} \{x_i\}\times \repD[\ell+1]\xhookrightarrow{\mapDelta+1} \repD[m+1] \]
\[\repD[m_i]\twoheadrightarrow\repD[m'_i]\xrightarrow{\overline{\sigma_i}'} \{y_i\}\times \repD[m]\xhookrightarrow{d^{m+1}} \repD[m+1]. \]
In particular, the (unique) totally non-degenerate necklace $\overline{T}\to \repD[m+1]_{i,m+1}$ obtained from $T=\repD[m_1]\vee\ldots\vee\repD[m_t]\to (\Cone_{-,\defThn,\defS})_{i,m+1}\to \repD[m+1]_{i,m+1}$ is also given by the above factorization of its beads. Hence the diagram commutes.
\end{proof}

Recall from \cite[Theorem 2.1.2]{LRfib} that there is an equivalence between the categories of functors $(\catnec{\repD[m+1]}{i}{m+1})^{\op}\to \set$ and of discrete fibrations over $\catnec{\repD[m+1]}{i}{m+1}$. We now identify the set-valued functor corresponding to the discrete fibration $(Q_{-,\defThn,\defS})_!$ under this equivalence. For this, recall the notations from \cref{notn:+1,rmk:lastbead,rem:DSepi}.

\begin{notation} \label{notation:newGbar}
    For $0\leq i\leq m$, we define a functor 
    \[ \newGbar\colon (\catnec{\repD[m+1]}{i}{m+1})^{\op}\to \Thnsset. \]
    It sends an object $T\to \repD[m+1]_{i,m+1}$ in $\catnec{\repD[m+1]}{i}{m+1}$ to the $\Thn$-space
    \[ \begin{cases}
    (\prod_{B(\overline{T})\setminus \{B_\omega^{\overline{T}}\}} Y)\times X & \text{if } B_\omega^{\overline{T}}\subseteq \im(\mapDelta+1) \\
    \emptyset & \text{else .}
    \end{cases} \]
   It sends a map $g\colon U\to T$ in $\catnec{\repD[m+1]}{i}{m+1}$ to the map in $\Thnsset$
    \[  \begin{cases}
    (\prod_{B(\overline{T})\setminus \{B^{\overline{T}}_\omega\}} Y)\times X\xrightarrow{\hat{f}B(\overline{g})^*} (\prod_{B(\overline{U})\setminus \{B^{\overline{U}}_\omega\}} Y)\times X & \text{if } B^{\overline{U}}_\omega\subseteq B^{\overline{T}}_\omega\subseteq \im(\mapDelta+1) \\
    \emptyset\to (\prod_{B(\overline{U})\setminus \{B^{\overline{U}}_\omega\}} Y)\times X & \text{if } B^{\overline{U}}_\omega\subseteq \im(\mapDelta+1), \, B^{\overline{T}}_\omega \not\subseteq \im(\mapDelta+1) \\
    \emptyset \to \emptyset & \text{else,}
    \end{cases}\]
    where $\hat{f} B(\overline{g})^*$ is the composite in $\Thnsset$ from \cref{notn:Falphatnd} replacing $g\colon U\to T$ with $\overline{g}\colon \overline{U}\to\overline{T}$.

    For $\defThn\in \Thn$ and $\defS\geq 0$, we write $\newGbar_{\defThn,\defS}$ for the composite 
    \[ \newGbar_{\defThn,\defS}\colon (\catnec{\repD[m+1]}{i}{m+1})^{\op}\xrightarrow{\newGbar}\Thnsset\xrightarrow{(-)_{\defThn,\defS}} \set. \]
\end{notation}

\begin{prop} \label{QvsnewG}
    For $\defThn\in \Thn$, $\defS\geq 0$, and $0\leq i\leq m$, the discrete fibration 
    \[ (Q_{-,\defThn,\defS})_!\colon \catnec{\Cone_{-,\defThn,\defS}}{i}{m+1}\to \catnec{\repD[m+1]}{i}{m+1} \]
    corresponds to the functor 
    \[ \newGbar_{\defThn,\defS}\colon (\catnec{\repD[m+1]}{i}{m+1})^{\op}\to \set. \]
\end{prop}

\begin{proof}
    We first show that there is an isomorphism of sets between the fiber at a given necklace $T\to \repD[m+1]_{i,m+1}$ of the functor $(Q_{-,\defThn,\defS})_!\colon \catnec{\Cone_{-,\defThn,\defS}}{i}{m+1}\to \catnec{\repD[m+1]}{i}{m+1}$ and the value at this necklace of the functor $\newGbar_{\defThn,\defS}\colon (\catnec{\repD[m+1]}{i}{m+1})^{\op}\to \set$.
    We write $\mathrm{fib}_{T\to \repD[m+1]_{i,m+1}}((Q_{-,\defThn,\defS})_!)$ for the fiber of $(Q_{-,\defThn,\defS})_!$ at the necklace $T\to \repD[m+1]_{i,m+1}$. Let $\overline{T}=\repD[m_1]\vee\ldots\vee\repD[m_t]\hookrightarrow \repD[m+1]_{i,m+1}$ be the (unique) totally non-degenerate necklace  obtained from $T\to \repD[m+1]_{i,m+1}$ by \cref{rem:DSepi}. 

    Given a necklace $T\to (\Cone_{-,\defThn,\defS})_{i,m+1}$ which is sent by the functor $(Q_{-,\defThn,\defS})_!$ to the necklace $T\to \repD[m+1]_{i,m+1}$, it follows from \cref{Qcommutewithoverline} that the (unique) totally non-degenerate $\overline{T}\to (\Cone_{-,\defThn,\defS})_{i,m+1}$ obtained from $T\to (\Cone_{-,\defThn,\defS})_{i,m+1}$ by \cref{rem:DSepi} is sent by $(Q_{-,\defThn,\defS})_!$ to the totally non-degenerate necklace $\overline{T}\hookrightarrow \repD[m+1]_{i,m+1}$. Moreover, since $T\to (\Cone_{-,\defThn,\defS})_{i,m+1}$ is the composite 
    \[ T\twoheadrightarrow \overline{T} \to (\Cone_{-,\defThn,\defS})_{i,m+1},\]
    it is uniquely determined by $\overline{T} \to (\Cone_{-,\defThn,\defS})_{i,m+1}$.
    
    Now, if $B_\omega^{\overline{T}}=\repD[m_t]\subseteq \im(\mapDelta+1)$, we show that there is an isomorphism of sets \begin{equation} \label{iso1} \textstyle \mathrm{fib}_{T\to \repD[m+1]_{i,m+1}}((Q_{-,\defThn,\defS})_!)\cong \prod_{B(\overline{T})\setminus \{B_\omega^{\overline{T}}\}} Y_{\defThn,\defS}\times X_{\defThn,\defS}=\newGbar_{\defThn,\defS}(T). \end{equation} 
    For $1\leq i\leq t-1$, the $i$th bead $\sigma_i\colon \repD[m_i]\hookrightarrow\overline{T} \to (\Cone_{-,\defThn,\defS})_{i,m+1}$ of $\overline{T}$ corresponds, by the description of $\Cone_{-,\defThn,\defS}$ given in \cref{ConePushoutThetak} and the fact that $m+1$ is not in the image of $\sigma_i$, to a non-degenerate $m_i$-simplex
    \[  \textstyle \repD[m_i]\xrightarrow{\overline{\sigma_i}} \{y_i\}\times \repD[m]\hookrightarrow(\coprod_{Y_{\defThn,\defS}} \repD) \amalg_{\coprod_{X_{\defThn,\defS}} \repD[\ell]} (\coprod_{X_{\defThn,\defS}} \repD[\ell+1]), \]
    for some $y_i\in Y_{\defThn,\defS}$. Then, the last bead $\sigma_t\colon B_{\omega}^{\overline{T}}=\repD[m_t]\hookrightarrow\overline{T} \to (\Cone_{-,\defThn,\defS})_{i,m+1}$ corresponds, by the description of $\Cone_{-,\defThn,\defS}$ given in \cref{ConePushoutThetak} and the fact that $m+1$ is in the image of $\sigma_t$, to a non-degenerate $m_t$-simplex
    \[  \textstyle \repD[m_t]\xrightarrow{\overline{\sigma_t}} \{x\}\times \repD[\ell+1]\hookrightarrow(\coprod_{Y_{\defThn,\defS}} \repD) \amalg_{\coprod_{X_{\defThn,\defS}} \repD[\ell]} (\coprod_{X_{\defThn,\defS}} \repD[\ell+1]), \]
    for some $x\in X_{\defThn,\defS}$. Then the data $(T\to \repD[m+1]_{i,m+1},\{y_i\}_{1\leq i\leq t-1},x)$ uniquely determine the necklace $T\to (\Cone_{-,\defThn,\defS})_{i,m+1}$, hence giving the desired isomorphism. 

    If $B_\omega^{\overline{T}}=\repD[m_t]\not\subseteq \im(\mapDelta+1)$, then \begin{equation} \label{iso2} \textstyle \mathrm{fib}_{T\to \repD[m+1]_{i,m+1}}((Q_{-,\defThn,\defS})_!)=\emptyset=\newGbar_{\defThn,\defS}(T). \end{equation} 
    Indeed, then there are no $m_t$-simplices of $\Cone_{-,\defThn,\defS}$ that contains $m+1$ and can be mapped to $\repD[m_t]\subseteq \repD[m+1]$ by $Q_{-,\defThn,\defS}\colon \Cone_{-,\defThn,\defS}\to \repD[m+1]$, as $Q_{-,\defThn,\defS}$ acts as $\mapDelta+1$ on simplices containing $m+1$. So there are no lifts of $T\to \repD[m+1]_{i,m+1}$.

    It remains to show that the isomorphisms \eqref{iso1} and \eqref{iso2} assemble into a natural isomorphism. For this, note that if $g\colon U\to T$ is a map in $\catnec{\repD[m+1]}{i}{m+1}$, then by the proof of \cref{prop:Qgendiscfib}, the map $g$ acts on the fibers of $(Q_{-,\defThn,\defS})_!$ by pre-composition
    \[ g^*\colon \mathrm{fib}_{T\to\repD[m+1]_{i,m+1}}((Q_{-,\defThn,\defS})_!)\to \mathrm{fib}_{U\to\repD[m+1]_{i,m+1}}((Q_{-,\defThn,\defS})_!). \] 
   Using that the map $g\colon U\to T$ is completely determined by the induced map $\overline{g}\colon \overline{U}\hookrightarrow \overline{T}$ obtained by \cref{rem:DSepi}, a direct computation using the above description of morphisms on fibers of $(Q_{-,\defThn,\defS})_!$ and the definition of $\newGbar_{\defThn,\defS}$ on morphisms shows that the isomorphisms \eqref{iso1} and \eqref{iso2} are natural in $T\to \repD[m+1]_{i,m+1}$. 
\end{proof}

We now aim to describe the hom $\Thn$-spaces of the categorification $\Ch\Cone$.

\begin{notation}
    For $m\geq 0$ and $0\leq i\leq m$, we define a functor 
    \[ \Hbar\colon \catnec{\repD[m+1]}{i}{m+1}\to \sset. \]
    It sends an object $T\to \repD[m+1]_{i,m+1}$ in $\catnec{\repD[m+1]}{i}{m+1}$ to the space $\Hom_{\CL T}(\alpha,\omega)$ and a map $g\colon U\to T$ in $\catnec{\repD[m+1]}{i}{m+1}$ to the map in $\sset$
    \[ (\CL g)_{\alpha,\omega}\colon \Hom_{\CL T}(\alpha,\omega)\to \Hom_{\CL U}(\alpha,\omega). \]
\end{notation}

\begin{lemma} \label{rewritecolim}
    For $0\leq i\leq m$, there is a natural isomorphism in $\Thnssset$
    \[ \colim_{T\in \catnec{\Cone_{-,\star,\star}}{i}{m+1}} \Hom_{\CL T}(\alpha,\omega) \cong \colim^{\newGbar_{\star,\star}}_{\catnec{\repD[m+1]}{i}{m+1}} \Hbar, \]
    where $\colim^{\newGbar_{\star,\star}}_{\catnec{\repD[m+1]}{i}{m+1}} \Hbar$ is the object of $\Thnssset$ given at $\defThn\in \Thn$ and $\defS\geq 0$ by the colimit in $\sset$ of the functor $\Hbar$ weighted by $\newGbar_{\defThn,\defS}$.
\end{lemma}

\begin{proof}
    Let $\defThn\in \Thn$ and $\defS\geq 0$. By \cref{QvsnewG} the category of elements of the functor $\newGbar_{\defThn,\defS}\colon (\catnec{\repD[m+1]}{i}{m+1})^{\op}\to \set$ is given by the discrete fibration
\[ (Q_{-,\defThn,\defS})_!\colon \catnec{\Cone_{-,\defThn,\defS}}{i}{m+1}\to \catnec{\repD[m+1]}{i}{m+1}.\] 
So by \cite[(7.1.8)]{RiehlCHT} we have natural isomorphisms in $\sset$
\begin{align*} 
\colim_{T\in \catnec{\Cone_{-,\defThn,\defS}}{i}{m+1}} \Hom_{\CL T}(\alpha,\omega) &= \colim_{\catnec{\Cone_{-,\defThn,\defS}}{i}{m+1}} \Hbar (Q_{-,\defThn,\defS})_! \\
&\cong  \colim^{\newGbar_{\defThn,\defS}}_{\catnec{\repD[m+1]}{i}{m+1}} \Hbar. \qedhere
\end{align*}
\end{proof}

Recall the inclusion $\iota\colon \Thnsset\hookrightarrow \Thnssset$ from \cref{rmk:enrichment}.

\begin{lemma} \label{rewritecolim2}
    For $0\leq i\leq m$, there is a natural isomorphism in $\Thnssset$
    \[ \colim_{T\in \catnec{\Cone_{-,\star,\star}}{i}{m+1}} \Hom_{\CL T}(\alpha,\omega) \cong \colim^{\Hbar}_{(\catnec{\repD[m+1]}{i}{m+1})^{\op}} \iota\newGbar, \]
    where $\colim^{\Hbar}_{(\catnec{\repD[m+1]}{i}{m+1})^{\op}} \iota\newGbar$ is the $\sset$-enriched colimit of $\iota\newGbar$ weighted by~$\Hbar$.
\end{lemma}

\begin{proof}
    By \cref{rewritecolim} and \cite[Lemma 3.3.5]{MRR1}, we have isomorphisms in $\Thnssset$
    \begin{align*} 
\colim_{T\in \catnec{\Cone_{-,\star,\star}}{i}{m+1}} \Hom_{\CL T}(\alpha,\omega) &\cong \colim^{\newGbar_{\star,\star}}_{\catnec{\repD[m+1]}{i}{m+1}} \Hbar \\
&\cong \colim^{\Hbar}_{(\catnec{\repD[m+1]}{i}{m+1})^{\op}} \iota\newGbar. \qedhere
\end{align*}
\end{proof}

\begin{prop} \label{homconeasweigthed}
    For $0\leq i\leq m$, there is a natural isomorphism in $\Thnsset$
    \[ \Hom_{\Ch\Cone}(i,m+1)\cong \diag(\colim^{\Hbar}_{(\catnec{\repD[m+1]}{i}{m+1})^{\op}} \iota\newGbar), \]
    where $\colim^{\Hbar}_{(\catnec{\repD[m+1]}{i}{m+1})^{\op}} \iota\newGbar$ is the $\sset$-enriched colimit of $\iota\newGbar$ weighted by~$\Hbar$.
\end{prop}

\begin{proof}
    By \cref{cor:computationhomsC,rewritecolim2}, we have natural isomorphisms in $\Thnsset$
    \begin{align*} \Hom_{\Ch\Cone}(i,m+1) &\cong \diag(\colim_{T\in \catnec{\Cone_{-,\star,\star}}{i}{m+1}} \Hom_{\CL T}(\alpha,\omega)) \\
    &\cong \diag(\colim^{\Hbar}_{(\catnec{\repD[m+1]}{i}{m+1})^{\op}} \iota\newGbar). \qedhere
    \end{align*}
\end{proof}

With this computation, we are now ready to calculate the straightening of the desired map.

\begin{rmk}
    For $0\leq i\leq m$, the construction $\newGbar$ from \cref{notation:newGbar} defines a functor $\newGbar[-]\colon \Thnssetslice{Y}\to (\Thnsset)^{(\catnec{\repD[m+1]}{i}{m+1})^{\op}}$. Given a map $X'\xrightarrow{\sigma} X\xrightarrow{f} Y$ in $\Thnssetslice{Y}$, there is an induced natural transformation $\overline{\sigma}\colon \newGbar[f\sigma]\to \newGbar$ which acts as $\sigma$ on the copies of $X'$ and $X$ in each component of $\newGbar[f\sigma]$ and $\newGbar$.
\end{rmk}

Using this functoriality, we get the following. We write $X\cong \colim_{\repD[\defThn,\defS]\xrightarrow{\sigma} X} \repD[\defThn,\defS]$ as a colimit of representables in $\Thnsset$, where we recall that $X$ is the source of the given map $f\colon X\to Y$ in $\Thnsset$ between connected $\Thn$-spaces.  

\begin{lemma} \label{colimofnewG}
    For $0\leq i\leq m$, there is a natural isomorphism in $(\Thnsset)^{(\catnec{\repD[m+1]}{i}{m+1})^{\op}}$
    \[ \colim_{\repD[\defThn,\defS]\xrightarrow{\sigma}X} \newGbar[f\sigma]\cong \newGbar. \]
\end{lemma}

\begin{proof}
    The component at every object $T\to \repD[m+1]_{i,m+1}$ in $\catnec{\repD[m+1]}{i}{m+1}$ of the canonical map $\colim_{\repD[\defThn,\defS]\xrightarrow{\sigma}X} \newGbar[f\sigma]\to \newGbar$ in $(\Thnsset)^{(\catnec{\repD[m+1]}{i}{m+1})^{\op}}$ 
    is an isomorphism in $\Thnsset$. Indeed, this follows from the fact that, for every object $T\to \repD[m+1]_{i,m+1}$ in $\catnec{\repD[m+1]}{i}{m+1}$, we have isomorphisms in $\Thnsset$
    \begin{align*} \textstyle \colim_{\repD[\defThn,\defS]\xrightarrow{\sigma}X}((\prod_{B(\overline{T})\setminus \{B_\omega^{\overline{T}}\}} Y)\times \repD[\defThn,\defS]) & \textstyle\cong  (\prod_{B(\overline{T})\setminus \{B_\omega^{\overline{T}}\}} Y)\times (\colim_{\repD[\defThn,\defS]\xrightarrow{\sigma}X}\repD[\defThn,\defS]) \\ &\textstyle\cong (\prod_{B(\overline{T})\setminus \{B_\omega^{\overline{T}}\}} Y)\times X,
    \end{align*}
    where the first isomorphisms holds since products in $\Thnsset$ commute with colimits.
\end{proof}

\begin{lemma} \label{colimofweightedcolim}
    For $0\leq i\leq m$, there is a natural isomorphism in $\Thnssset$
    \[ \colim_{\repD[\defThn,\defS]\xrightarrow{\sigma} X} \colim^{\Hbar}_{(\catnec{\repD[m+1]}{i}{m+1})^{\op}} \iota\newGbar[f\sigma]\cong \colim^{\Hbar}_{(\catnec{\repD[m+1]}{i}{m+1})^{\op}} \iota\newGbar. \]
\end{lemma}

\begin{proof}
    This follows directly from the isomorphism from \cref{colimofnewG}, and the fact that the functor $\colim^{\Hbar}_{(\catnec{\repD[m+1]}{i}{m+1})^{\op}} \iota(-)\colon (\Thnsset)^{(\catnec{\repD[m+1]}{i}{m+1})^{\op}} \to \Thnssset$ preserves colimits.
\end{proof}

\begin{lemma} \label{colimofhomCone}
    There is a natural isomorphism in $[\Ch L\repD[m,Y]^{\op},\Thnsset]$
    \[ \colim_{\repD[\defThn,\defS]\xrightarrow{\sigma}X} \Hom_{\Ch\Cone[f\sigma]}(-,m+1)\circ \Ch(\iota_{f\sigma})\cong \Hom_{\Ch\Cone}(-,m+1)\circ \Ch(\iota_f). \]
\end{lemma}

\begin{proof}
    Let $0\leq i\leq m$. By applying the colimit-preserving functor $\diag\colon \Thnssset\to \Thnsset$ to the isomorphism from \cref{colimofweightedcolim}, we get natural isomorphisms in $\Thnsset$
    \begin{align*} \colim_{\repD[\defThn,\defS]\xrightarrow{\sigma} X} \diag(\colim&^{\Hbar}_{(\catnec{\repD[m+1]}{i}{m+1})^{\op}} \iota\newGbar[f\sigma]) \\
    &\cong \diag(\colim_{\repD[\defThn,\defS]\xrightarrow{\sigma} X} \colim^{\Hbar}_{(\catnec{\repD[m+1]}{i}{m+1})^{\op}} \iota\newGbar[f\sigma]) \\
    &\cong \diag(\colim^{\Hbar}_{(\catnec{\repD[m+1]}{i}{m+1})^{\op}} \iota\newGbar).
    \end{align*}
    Then, by \cref{homconeasweigthed}, this yields a natural isomorphism in $\Thnsset$
    \[ \colim_{\repD[\defThn,\defS]\xrightarrow{\sigma} X} \Hom_{\Ch\Cone[f\sigma]}(i,m+1)\cong \Hom_{\Ch\Cone}(i,m+1). \]
    Since the above isomorphisms are natural in $0\leq i\leq m$, they assemble into a natural isomorphism in $[\Ch L\repD[m,Y]^{\op},\Thnsset]$
    \[ \colim_{\repD[\defThn,\defS]\xrightarrow{\sigma} X} \Hom_{\Ch\Cone[f\sigma]}(-,m+1)\circ \Ch(\iota_{f\sigma})\cong \Hom_{\Ch\Cone}(-,m+1)\circ \Ch(\iota_{f}). \qedhere \]
\end{proof}

We finally prove \cref{Stof[lX]} which we restate for the reader's convenience..

\begin{prop} \label{Stof[lX]appendix}
    Let $\mapDelta\colon [\ell]\to [m]$ be an injective map in $\Delta$, and $f\colon X\to Y$ be map in $\Thnsset$ between connected $\Thn$-spaces. The straightening functor $\St_{L\repD[m,Y]}$ sends the object $[\mapDelta,f]\colon \repD[\ell,X]\to L\repD[m,Y]$ in $\sThnssetsliceshort{L\repD[m,Y]}$ to the $\Thnsset$-enriched functor 
    \[ \St_{L\repD[m,Y]} ([\mapDelta,f])\colon \Ch L\repD[m,Y]^{\op}\xrightarrow{\Ch(\iota_{f})} \Ch\Cone^{\op}\xrightarrow{\Hom_{\Ch \Cone}(-,m+1)} \Thnsset. \]
\end{prop}

\begin{proof}
    Write $X\cong \colim_{\repD[\defThn,\defS]\xrightarrow{\sigma} X} \repD[\defThn,\defS]$ as a colimit of representables in $\Thnsset$. Since $\St_{L\repD[m,Y]}\colon \sThnssetsliceshort{L\repD[m,Y]} \to [\Ch L\repD[m,Y]^{\op},\Thnsset]$ commutes with colimits and by its definition on representables, we have isomorphisms in $[\Ch L\repD[m,Y]^{\op},\Thnsset]$
    \begin{align*}
        \St_{L\repD[m,Y]}([\mapDelta,f]) &\cong \colim_{\repD[\defThn,\defS]\xrightarrow{\sigma} X} \St_{L\repD[m,Y]}([\mapDelta,f\sigma]) \\
        &\cong \colim_{\repD[\defThn,\defS]\xrightarrow{\sigma} X} \Hom_{\Ch(L\repD[m,Y]_{[\mapDelta,f\sigma]})}(-,\top)\circ \Ch(\iota_{f\sigma}).
    \end{align*}
    By \cref{pushoutlowersigma,notationCone}, we see that there is an isomorphism $L\repD[m,Y]_{[\mapDelta,f\sigma]}\cong \Cone[f\sigma]$ in $\pcatThn$ that identifies $\top$ with $m+1$, for every map $\sigma\colon \repD[\defThn,\defS]\to X$ in $\Thnsset$. Hence, combining with \cref{colimofhomCone}, we get isomorphisms in $[\Ch L\repD[m,Y]^{\op},\Thnsset]$
    \begin{align*}
    \colim_{\repD[\defThn,\defS]\xrightarrow{\sigma} X} \Hom_{\Ch(L\repD[m,Y]_{[\mapDelta,f\sigma]})}&(-,\top)\circ \Ch(\iota_{f\sigma}) \\ &\cong \colim_{\repD[\defThn,\defS]\xrightarrow{\sigma} X} \Hom_{\Ch\Cone[f\sigma]}(-,m+1)\circ \Ch(\iota_{f\sigma}) \\
    &\cong \Hom_{\Ch\Cone}(-,m+1)\circ \Ch(\iota_{f}).
    \end{align*}
    All together they assemble into the desired isomorphism in $[\Ch L\repD[m,Y]^{\op},\Thnsset]$.
\end{proof}

\subsection{Enrichment of straightening} \label{sec:enrichmentofSt}

We now prove \cref{lem:Stprestensors} which states that the straightening functor is enriched over $\Thnsset$.

Let us fix $m,\defS\geq 0$, $\defThn\in \Thn$, and an object $X\in \Thnsset$. We first compare the straightening of the canonical projection 
\[ [\id_{[m]},\pi]\colon \repD[m,\defThn,\defS]\times X\to \repD[m,\defThn,\defS] \]
with the tensor by $X$ of the straightening of the identity at $\repD[m,\defThn,\defS]$.

\begin{lemma} \label{lem:hompivsid}
    For $0\leq i\leq m$, there is an isomorphism in $\Thnsset$
    \[ \Hom_{\Cone[\pi][\id_{[m]}]}(i,m+1)\cong \Hom_{\Cone[\id_{\repD[\defThn,\defS]}][\id_{[m]}]}(i,m+1)\times X. \]
\end{lemma}

\begin{proof}
    For every object $T\to \repD[m+1]_{i,m+1}$ in $\catnec{\repD[m+1]}{i}{m+1}$, a direct computation shows that
    \[ \textstyle \newGbar[\pi][i][\id_{[m]}](T)=(\prod_{B(\overline{T})} \repD[\defThn,\defS])\times X=\newGbar[\id_{\repD[\defThn,\defS]}](T)\times X=(\newGbar[\id_{\repD[\defThn,\defS]}]\otimes X)(T)\]
    and so we have $\newGbar[\pi][i][\id_{[m]}]=\newGbar[\id_{\repD[\defThn,\defS]}]\otimes X$ in $(\Thnsset)^{(\catnec{\repD[m+1]}{i}{m+1})^{\op}}$. Since the inclusion functor $\iota\colon \Thnsset\to \Thnssset$ preserves products, we obtain an isomorphism in $(\Thnssset)^{(\catnec{\repD[m+1]}{i}{m+1})^{\op}}$
    \[ \iota\newGbar[\pi][i][\id_{[m]}]= \iota(\newGbar[\id_{\repD[\defThn,\defS]}]\otimes X)\cong (\iota\newGbar[\id_{\repD[\defThn,\defS]}])\otimes \iota X. \]
    Now, since weighted colimits commute with each other and tensors are an example of such by \cite[(3.44)]{Kelly}, we get isomorphisms in $\Thnssset$
    \begin{align*} \colim^{\Hbar}_{(\catnec{\repD[m+1]}{i}{m+1})^{\op}}\iota\newGbar[\pi][i][\id_{[m]}] &\cong \colim^{\Hbar}_{(\catnec{\repD[m+1]}{i}{m+1})^{\op}}(\iota\newGbar[\id_{\repD[\defThn,\defS]}]\otimes \iota X) \\
    &\cong (\colim^{\Hbar}_{(\catnec{\repD[m+1]}{i}{m+1})^{\op}}\iota\newGbar[\id_{\repD[\defThn,\defS]}])\times \iota X. 
    \end{align*}
    Finally, as $\diag\colon \Thnssset\to \Thnsset$ commutes with products and $\diag(\iota X)=X$, we get isomorphisms in $\Thnsset$
    \begin{align*} \diag(\colim^{\Hbar}_{(\catnec{\repD[m+1]}{i}{m+1})^{\op}} \iota\newGbar[\pi][i][\id_{[m]}]) &\cong \diag ((\colim^{\Hbar}_{(\catnec{\repD[m+1]}{i}{m+1})^{\op}}\iota\newGbar[\id_{\repD[\defThn,\defS]}])\times \iota X) \\
    &\cong \diag (\colim^{\Hbar}_{(\catnec{\repD[m+1]}{i}{m+1})^{\op}}\iota\newGbar[\id_{\repD[\defThn,\defS]}])\times X. 
    \end{align*}
    By \cref{homconeasweigthed}, this gives an isomorphism in $\Thnsset$
    \[ \Hom_{\Cone[\pi][\id_{[m]}]}(i,m+1)\cong \Hom_{\Cone[\id_{\repD[\defThn,\defS]}][\id_{[m]}]}(i,m+1)\times X. \qedhere \]
\end{proof}

\begin{lemma} \label{lem:streppivsid}
    There is an isomorphism in $[\Ch L\repD[m,\defThn,\defS]^{\op},\Thnsset]$
    \[ \St_{L\repD[m,\defThn,\defS]}([\id_{[m]},\pi])\cong \St_{L\repD[m,\defThn,\defS]}(\id_{\repD[m,\defThn,\defS]})\otimes X. \]
\end{lemma}

\begin{proof}
    Let $0\leq i\leq m$. By \cref{Stof[lX]appendix,lem:hompivsid}, we have isomorphisms in $\Thnsset$
    \begin{align*}
        \St_{L\repD[m,\defThn,\defS]}([\id_{[m]},\pi])(i)
     &\cong \Hom_{\Cone[\pi][\id_{[m]}]}(i,m+1) \cong \Hom_{\Cone[\id_{\repD[\defThn,\defS]}][\id_{[m]}]}(i,m+1)\times X \\
        &\cong \St_{L\repD[m,\defThn,\defS]}(\id_{\repD[m,\defThn,\defS]})(i)\times X=(\St_{L\repD[m,\defThn,\defS]}(\id_{\repD[m,\defThn,\defS]})\otimes X)(i).
    \end{align*}
    Since the above isomorphisms are natural in $0\leq i\leq m$, they assemble into a natural isomorphism in $[\Ch L\repD[m,\defThn,\defS]^{\op},\Thnsset]$
    \[  \St_{L\repD[m,\defThn,\defS]}([\id_{[m]},\pi])\cong \St_{L\repD[m,\defThn,\defS]}(\id_{\repD[m,\defThn,\defS]})\otimes X. \qedhere \]
\end{proof}

From this, we can deduce \cref{lem:Stprestensors} which we restate for the reader's convenience.

\begin{prop} \label{Stpreservetensors}
Let $W$ be an object in $\pcatThn$ and $X\in\Thnsset$. Then the following square of left adjoint functors commutes up to a natural isomorphism.
\begin{tz}
        \node[](1) {$\sThnssetslice{W}$}; 
        \node[right of=1,xshift=2.8cm](2) {$[\Ch W^{\op},\Thnsset]$}; 
        \node[below of=1](3) {$\sThnssetslice{W}$}; 
        \node[below of=2](4) {$[\Ch W^{\op},\Thnsset]$}; 

        \draw[->] (1) to node[above,la]{$\St_W$} (2); 
        \draw[->] (1) to node[left,la]{$(-)\otimes X$} (3); 
        \draw[->] (2) to node[right,la]{$(-)\otimes X$} (4); 
        \draw[->] (3) to node[below,la]{$\St_W$} (4);
         \cell[la,above][n][.5]{2}{3}{$\cong$};
    \end{tz}
\end{prop}

\begin{proof}
Since all functors involved are left adjoints and there is an equivalence of categories $\sThnssetslice{W}\simeq \set^{(\int_{\DThnS} W)^{\op}}$, it is enough to show that, for every representable object $\sigma\colon \repD[m,\defThn,\defS]\to W$ in $\sThnssetslice{W}$, we have an isomorphism in $[\Ch W^{\op},\Thnsset]$
    \[ \St_W(\sigma\otimes X)\cong \St_W(\sigma)\otimes X. \]
    By \cref{lem:streppivsid}, we have an isomorphism in $[\Ch L\repD[m,\defThn,\defS]^{\op},\Thnsset]$
    \[  \St_{L\repD[m,\defThn,\defS]}(\id_{\repD[m,\defThn,\defS]}\otimes X)=\St_{L\repD[m,\defThn,\defS]}([\id_{[m]},\pi])\cong \St_{L\repD[m,\defThn,\defS]}(\id_{\repD[m,\defThn,\defS]})\otimes X. \]
    As $(\Ch\sigma)_!$ preserves tensors by \cref{subsec:MSproj}, we get isomorphisms in $[\Ch W^{\op},\Thnsset]$
    \begin{align*} (\Ch\sigma)_!\St_{L\repD[m,\defThn,\defS]}(\id_{\repD[m,\defThn,\defS]}\otimes X) &\cong (\Ch\sigma)_!(\St_{L\repD[m,\defThn,\defS]}(\id_{\repD[m,\defThn,\defS]})\otimes X) \\
    &\cong ((\Ch\sigma)_!\St_{L\repD[m,\defThn,\defS]}(\id_{\repD[m,\defThn,\defS]}))\otimes X. 
    \end{align*}
    Finally, by \cref{basechangeStappendix}, this gives an isomorphism in $[\Ch W^{\op},\Thnsset]$
    \[ \St_W(\sigma\otimes X)\cong \St_W(\sigma)\otimes X. \qedhere \]
\end{proof}

\subsection{Straightening of \texorpdfstring{$\repD[\ell,X]\to L\repD[m,Y]$}{F[l,X]->LF[m,Y]}, monomorphism case} \label{sec:Stofmono}

We now prove \cref{Stofalphafmono} computing the straightening of a map $[\mapDelta,f]\colon \repD[\ell,X]\to L \repD[m,Y]$ in $\sThnsset$ in the special case where $f\colon X\to Y$ is a monomorphism in $\Thnsset$.

Let us fix an injective map $\mapDelta\colon [\ell]\to [m]$ in $\Delta$ and a monomorphism $f\colon X\hookrightarrow Y$ in $\Thnsset$ between connected spaces. Under these assumptions, we first show that the simplicial set $\Cone_{-,\defThn,\defS}$ is $1$-ordered as defined in \cite[Definition 2.1.3]{MRR1}.

\begin{prop} \label{prop:Cone1ordered}
    For $\defThn\in \Thn$ and $\defS\geq 0$, the simplicial set $\Cone_{-,\defThn,\defS}$ is $1$-ordered.
\end{prop}

\begin{proof}
    By \cref{Cone0}, we have that $\Cone_{0,\defThn,\defS} \cong \{0,1,\ldots,m+1\}$ and by construction every $1$-simplex goes from $i$ to $j$ where $i \leq j$. Hence the relation $\preceq_{\Cone_{-,\defThn,\defS}}$ is anti-symmetric.

    For $m'\geq 1$, we first show that the Segal map 
    \[ \Cone_{m',\defThn,\defS}\to \Cone_{1,\defThn,\defS}\times_{\Cone_{0,\defThn,\defS}} \ldots \times_{\Cone_{0,\defThn,\defS}}  \Cone_{1,\defThn,\defS}\]
    restricts to a map 
    \begin{equation} \label{eq1}
    \Cone_{m',\defThn,\defS}^{\mathrm{nd}}\to \{ g\colon \Sp[m']\to \Cone_{-,\defThn,\defS} \mid g \text{ mono} \},
    \end{equation} 
    where $\Sp[m']$ denotes the spine of $\repD[m']$. Let $\sigma\colon \repD[m']\to \Cone_{-,\defThn,\defS}$ be a non-degenerate $m'$-simplex of $\Cone_{-,\defThn,\defS}$. By the description of $\Cone_{-,\defThn,\defS}$ given in \cref{ConePushoutThetak}, if $m+1$ is in the image of $\sigma$, such an $m'$-simplex comes from a non-degenerate $m'$-simplex
    \[  \textstyle \repD[m']\xrightarrow{\overline{\sigma}} \{x\}\times \repD[\ell+1]\hookrightarrow(\coprod_{Y_{\defThn,\defS}} \repD) \amalg_{\coprod_{X_{\defThn,\defS}} \repD[\ell]} (\coprod_{X_{\defThn,\defS}} \repD[\ell+1]), \]
    for some $x\in X_{\defThn,\defS}$, and if $m+1$ is not in the image of $\sigma$, such an $m'$-simplex comes from a non-degenerate $m'$-simplex
    \[  \textstyle \repD[m']\xrightarrow{\overline{\sigma}} \{y\}\times \repD[m]\hookrightarrow(\coprod_{Y_{\defThn,\defS}} \repD) \amalg_{\coprod_{X_{\defThn,\defS}} \repD[\ell]} (\coprod_{X_{\defThn,\defS}} \repD[\ell+1]), \]
    for some $y\in Y_{\defThn,\defS}$. Since the simplicial sets $\repD[\ell+1]$ and $\repD$ are $1$-ordered by \cref{examplesof1ordered}, it follows that the induced map \[ \Sp[m']\hookrightarrow \repD[m']\xrightarrow{\overline{\sigma}} \{x\}\times\repD[\ell+1] \quad \text{or} \quad \Sp[m']\hookrightarrow\repD[m']\xrightarrow{\overline{\sigma}} \{y\}\times\repD \]
    is a monomorphism. Now, as $X_{\defThn,\defS}\subseteq Y_{\defThn,\defS}$, the composites
    \[ \textstyle\{x\}\times \repD[\ell+1]\hookrightarrow (\coprod_{Y_{\defThn,\defS}} \repD) \amalg_{\coprod_{X_{\defThn,\defS}} \repD[\ell]} (\coprod_{X_{\defThn,\defS}} \repD[\ell+1])\to \Cone_{-,\defThn,\defS} \]
    \[ \textstyle\{y\}\times \repD[m]\hookrightarrow (\coprod_{Y_{\defThn,\defS}} \repD) \amalg_{\coprod_{X_{\defThn,\defS}} \repD[\ell]} (\coprod_{X_{\defThn,\defS}} \repD[\ell+1])\to \Cone_{-,\defThn,\defS} \]
    are also monomorphisms and so the induced map $\Sp[m']\hookrightarrow\repD[m']\xrightarrow{\sigma} \Cone_{-,\defThn,\defS}$ is a composite of monomorphisms, hence a monomorphism. 

    Now we show that the map \eqref{eq1} is injective. Let $\sigma,\tau\colon \repD[m']\to \Cone_{-,\defThn,\defS}$ be non-degenerate $m'$-simplices of $\Cone_{-,\defThn,\defS}$ such that their restrictions along $\Sp[m']\hookrightarrow \repD[m']$ coincide. First note that $m+1$ is in the image of $\sigma$ if and only if it is in the image of~$\tau$. Indeed, if $m+1$ is in the image of $\sigma$, then it is in the image of its restriction along $\Sp[m']\hookrightarrow \repD[m']$, and so it must be in the image of $\tau$; and conversely. 
    
    In the case where $m+1$ is in the image of $\sigma$ and $\tau$, as before, the $m'$-simplices $\sigma$ and $\tau$ come from non-degenerate $m'$-simplices 
    \[ \overline{\sigma}\colon \repD[m']\to \{x\}\times \repD[\ell+1] \quad \text{and} \quad  \overline{\tau}\colon \repD[m']\to \{x'\}\times \repD[\ell+1], \]
    for some $x,x'\in X_{\defThn,\defS}$. As the restrictions of $\sigma$ and $\tau$ along $\Sp[m']\hookrightarrow \repD[m']$ coincide, and the map 
    \[ \textstyle (\coprod_{Y_{\defThn,\defS}} \repD) \amalg_{\coprod_{X_{\defThn,\defS}} \repD[\ell]} (\coprod_{X_{\defThn,\defS}} \repD[\ell+1])\to \Cone_{-,\defThn,\defS}\]
    is injective on $1$-simplices, it follows that $x=x'$. So $\overline{\sigma}$ and $\overline{\tau}$ are two non-degenerate $m'$-simplices of $\{x\}\times \repD[\ell+1]$ whose restrictions along $\Sp[m']\hookrightarrow \repD[m']$ coincide. Hence, as $\repD[\ell+1]$ is $1$-ordered by \cref{examplesof1ordered}, it follows that $\overline{\sigma}=\overline{\tau}$ and so $\sigma=\tau$. 

    The case where $m+1$ is not in the image of $\sigma$ and $\tau$ proceeds similarly, replacing $\{x\}\times \repD[\ell+1]$ and $\{x'\}\times \repD[\ell+1]$ by $\{y\}\times \repD$ and $\{y'\}\times \repD$ for some $y,y'\in Y_{\defThn,\defS}$.
\end{proof}

Since $\Cone$ is $1$-ordered, in order to compute the hom $\Thn$-spaces of the categorification $\Ch\Cone$, we need to study totally non-degenerate necklaces in $\Cone$. For this, we first show that the restriction of the discrete fibration \[ (Q_{-,\defThn,\defS})_!\colon \catnec{\Cone_{-,\defThn,\defS}}{i}{m+1}\to \catnec{\repD[m+1]}{i}{m+1} \] from \cref{prop:Qgendiscfib} to subcategories of totally non-degenerate necklaces is in fact its pullback.  

\begin{prop} \label{prop:pullbackQ}
For $\defThn\in \Thn$, $\defS\geq 0$, and $0\leq i\leq m$, we have the following pullback square in $\cat$.
\begin{tz}
        \node[](1) {$\tndnec{\Cone_{-,\defThn,\defS}}{i}{m+1}$}; 
        \node[right of=1,xshift=4.3cm](2) {$\catnec{\Cone_{-,\defThn,\defS}}{i}{m+1}$}; 
        \node[below of=1](3) {$\tndnec{\repD[m+1]}{i}{m+1}$}; 
        \node[below of=2](4) {$\catnec{\repD[m+1]}{i}{m+1}$}; 

        \draw[right hook->] (1) to (2); 
        \draw[->] (1) to node[left,la]{$(Q_{-,\defThn,\defS})_!$} (3); 
        \draw[->] (2) to node[right,la]{$(Q_{-,\defThn,\defS})_!$} (4); 
        \draw[right hook->] (3) to (4);
        \pullback{1};
    \end{tz}
 \end{prop}

\begin{proof}
    To show that the above square is a pullback, it is enough to show that a necklace in $(\Cone_{-,\defThn,\defS})_{i,m+1}$ is totally non-degenerate if and only if its image under the functor
    \[ (Q_{-,\defThn,\defS})_!\colon \catnec{\Cone_{-,\defThn,\defS}}{i}{m+1}\to \catnec{\repD[m+1]}{i}{m+1} \] 
    is a totally non-degenerate necklace in $\repD[m+1]_{i,m+1}$. However, this follows directly from \cref{Qnondegsimplex} as totally non-degenerate necklaces are precisely those necklaces whose beads are sent to non-degenerate simplices. 
\end{proof}

Since discrete fibrations are closed under pullbacks, as a consequence of \cref{prop:Qgendiscfib,prop:pullbackQ}, we obtain the following. 

\begin{prop} \label{prop:Qmonodiscfib}
    For $\defThn\in \Thn$, $\defS\geq 0$, and $0\leq i\leq m$, the functor 
    \[ (Q_{-,\defThn,\defS})_!\colon \tndnec{\Cone_{-,\defThn,\defS}}{i}{m+1}\to \tndnec{\repD[m+1]}{i}{m+1} \]
    is a discrete fibration.
\end{prop}

We now identify the functor corresponding to this discrete fibration. For this, recall the functor $\newG\colon (\tndnec{\repD[m+1]}{i}{m+1})^{\op}\to \Thnsset$ from \cref{notn:Falphatnd}.

\begin{prop} \label{QmonovsnewG}
    For $\defThn\in\Thn$, $\defS\geq 0$, and $0\leq i\leq m$, the discrete fibration 
    \[ (Q_{-,\defThn,\defS})_!\colon \tndnec{\Cone_{-,\defThn,\defS}}{i}{m+1}\to \tndnec{\repD[m+1]}{i}{m+1} \]
    corresponds to the functor 
    \[ \newG_{\defThn,\defS}\colon (\tndnec{\repD[m+1]}{i}{m+1})^{\op}\to \set. \]
\end{prop}

\begin{proof}
    Recall that the equivalence between discrete fibrations and set-valued functors from \cite[Theorem 2.1.2]{LRfib} is natural with respect to functors $F\colon \cC\to \cD$ which act by sending a discrete fibration over $\cD$ to its pullback along $F$ and a functor $\cD\to \set$ to its pre-composition with $F$. Hence, by \cref{QvsnewG,prop:pullbackQ} the functor $(\tndnec{\repD[m+1]}{i}{m+1})^{\op}\to \set$ corresponding to the discrete fibration $(Q_{-,\defThn,\defS})_!\colon \tndnec{\Cone_{-,\defThn,\defS}}{i}{m+1}\to \tndnec{\repD[m+1]}{i}{m+1}$ is obtained by restricting the functor $\newGbar\colon (\catnec{\repD[m+1]}{i}{m+1})^{\op}\to \Thnsset$ from \cref{notation:newGbar} along the inclusion $\tndnec{\repD[m+1]}{i}{m+1}\hookrightarrow \catnec{\repD[m+1]}{i}{m+1}$. But this is precisely the functor $\newG_{\defThn,\defS}\colon (\tndnec{\repD[m+1]}{i}{m+1})^{\op}\to \set$.
\end{proof}

Recall the functor $H^i_{m+1}\colon \tndnec{\repD[m+1]}{i}{m+1}\to \sset$ from \cref{notationH}.

\begin{lemma} \label{rewritecolimmono}
    For $0\leq i\leq m$, there is a natural isomorphism in $\Thnssset$
    \[ \colim_{T\in \tndnec{\Cone_{-,\star,\star}}{i}{m+1}} \Hom_{\CL T}(\alpha,\omega) \cong \colim^{\newG_{\star,\star}}_{\tndnec{\repD[m+1]}{i}{m+1}} H^i_{m+1}, \]
    where $\colim^{\newG_{\star,\star}}_{\tndnec{\repD[m+1]}{i}{m+1}} H^i_{m+1}$ is the object of $\Thnssset$ given at $\defThn\in \Thn$ and $\defS\geq 0$ by the colimit in $\sset$ of the functor $H^i_{m+1}$ weighted by $\newG_{\defThn,\defS}$.
\end{lemma}

\begin{proof}
    The proof works as in \cref{rewritecolim}, using \cref{QmonovsnewG}.
\end{proof}

\begin{lemma} \label{rewritecolim2mono}
    For $0\leq i\leq m$, there is a natural isomorphism in $\Thnsset$
    \[ \colim_{T\in \tndnec{\Cone_{-,\star,\star}}{i}{m+1}} \Hom_{\CL T}(\alpha,\omega) \cong \colim^{H^i_{m+1}}_{(\tndnec{\repD[m+1]}{i}{m+1})^{\op}} \iota\newG, \]
    where $\colim^{H^i_{m+1}}_{(\tndnec{\repD[m+1]}{i}{m+1})^{\op}} \iota\newG$ is the $\sset$-enriched colimit of $\iota\newG$ weighted by $H^i_{m+1}$.
\end{lemma}

\begin{proof}
    The proof works as in \cref{rewritecolim2}, using \cref{rewritecolimmono}.
\end{proof}

Finally, we can compute the hom $\Thn$-spaces of the categorification $\Ch\Cone$, and hence the straightening of the desired map. 

\begin{prop} \label{homconeasweigthedmono}
    For $0\leq i\leq m$, there is a natural isomorphism in $\Thnsset$
    \[ \Hom_{\Ch\Cone}(i,m+1)\cong \diag(\colim^{H^i_{m+1}}_{(\tndnec{\repD[m+1]}{i}{m+1})^{\op}} \iota\newG), \]
    where $\colim^{H^i_{m+1}}_{(\tndnec{\repD[m+1]}{i}{m+1})^{\op}} \iota\newG$ is the $\sset$-enriched colimit of $\iota\newG$ weighted by $H^i_{m+1}$.
\end{prop}

\begin{proof}
    Recall from \cref{prop:Cone1ordered} that, for every $\defThn\in \Thn$ and $\defS\geq 0$, the simplicial set $\Cone_{-,\defThn,\defS}$ is $1$-ordered. Hence, by \cref{cor:computationshomC1ordered,rewritecolim2mono}, we have natural isomorphisms in $\Thnsset$
    \begin{align*} \Hom_{\Ch\Cone}(i,m+1) &\cong \diag(\colim_{T\in \tndnec{\Cone_{-,\star,\star}}{i}{m+1}} \Hom_{\CL T}(\alpha,\omega)) \\
    &\cong \diag(\colim^{H^i_{m+1}}_{(\tndnec{\repD[m+1]}{i}{m+1})^{\op}} \iota\newG). \qedhere
    \end{align*}
\end{proof}

We can now prove \cref{Stofalphafmono} which we restate for the reader's convenience. 

\begin{prop} 
    Let $\mapDelta\colon [\ell]\to [m]$ be an injective map in $\Delta$, and $f\colon X\hookrightarrow Y$ be a monomorphism in $\Thnsset$ between connected $\Thn$-spaces. For $0\leq i\leq m$, there is a natural isomorphism in $\Thnsset$
    \[ \St_{L\repD[m,Y]}([\mapDelta,f])(i)\cong \diag(\colim^{H^i_{m+1}}_{(\tndnec{\repD[m+1]}{i}{m+1})^{\op}} \iota\newG). \]
\end{prop}

\begin{proof}
    By \cref{Stof[lX]appendix,homconeasweigthedmono}, we have natural isomorphisms in $\Thnsset$
    \[
        \St_{L\repD[m,Y]}([\mapDelta,f])(i) \cong \Hom_{\Ch\Cone}(i,m+1) \cong \diag(\colim^{H^i_{m+1}}_{(\tndnec{\repD[m+1]}{i}{m+1})^{\op}} \iota\newG). \qedhere \]
\end{proof}

\section{Construction of the comparison map}

The goal of this is to prove \cref{natphic}. In \cref{sec:B1}, we first construct a $\Thnsset$-enriched functor $\Pi_{m,Y}\colon \Ch L\repD[m+1,Y]\to \Ch L\repD[m,Y]\amalg_{[0]} \Sigma Y$. Then, building on this, in \cref{sec:comparisonmap}, we construct, for every object $c\in \cC$, the desired $\Thnsset$-enriched natural transformation in $[\cC^{\op},\Thnsset]$ 
\[ \textstyle (\varepsilon_\cC)_!\St_{\Nh\cC} \int_\cC^\Nh\Hom_\cC(-,c)\to \Hom_\cC(-,c). \]
We then verify that it is a section of the $\Thnsset$-enriched natural transformation
\[ \textstyle\Hom_\cC(-,c)\cong (\varepsilon_\cC)_!\St_{\Nh\cC}(c)\xrightarrow{(\varepsilon_\cC)_!\St_{\Nh\cC}(\id_c)} (\varepsilon_\cC)_!\St_{\Nh\cC} \int_\cC^\Nh\Hom_\cC(-,c), \]
and that it is natural in $c\in \cC$, hence proving \cref{natphic}.

\subsection{The projection \texorpdfstring{$\Ch L\repD[m+1,Y]\to \Ch L\repD[m,Y]\amalg_{[0]}\Sigma Y$}{CL[m+1,Y]-->CL[m,X]USigmaY}} \label{sec:B1}

Let us fix $m\geq 0$ and a connected $\Thn$-space $Y$. In this section, we build a $\Thnsset$-enriched functor
\[ \Pi_{m,Y}\colon \Ch L\repD[m+1,Y]\to \Ch L\repD[m,Y]\amalg_{[0]} \Sigma Y.  \]

First, we let the $\Thnsset$-enriched functor $\Pi_{m,Y}$ be the identity on objects, and it remains to describe its action on hom $\Thn$-spaces. For this, we first study the totally non-degenerate necklaces in $L(\repD[m,Y]\amalg_{F[0,Y]}\repD[1,Y])_{-,\defThn,\defS}$. We get the following description.

\begin{lemma} \label{pullbackwedge}
For $\defThn\in \Thn$, $\defS\geq 0$, and $0\leq i<j\leq m+1$, we have the following pullback square in $\cat$.
\begin{tz}
        \node[](1) {$\tndnec{L(\repD[m,Y]\amalg_{F[0,Y]}\repD[1,Y])_{-,\defThn,\defS}}{i}{j}$}; 
        \node[right of=1,xshift=5.6cm](2) {$\tndnec{L\repD[m+1,Y]_{-,\defThn,\defS}}{i}{j}$}; 
        \node[below of=1](3) {$\tndnec{\repD[m]\vee\repD[1]}{i}{j}$}; 
        \node[below of=2](4) {$\tndnec{\repD[m+1]}{i}{j}$}; 

        \draw[right hook->] (1) to (2); 
        \draw[->] (1) to (3); 
        \draw[->] (2) to node[right,la]{$(Q_{-,\defThn,\defS})_!$} (4); 
        \draw[right hook->] (3) to (4);
        \pullback{1};
    \end{tz}
 \end{lemma}

 As a consequence, we get that the hom $\Thn$-spaces of the $\Thnsset$-enriched categories coincide when $0\leq i<j\leq m$.

 \begin{lemma} \label{necofijlowerthanm}
     For $\defThn\in \Thn$, $\defS\geq 0$, and $0\leq i<j\leq m$, there is an isomorphism of categories
     \[ \tndnec{L(\repD[m,Y]\amalg_{F[0,Y]}\repD[1,Y])_{-,\defThn,\defS}}{i}{j}\cong \tndnec{L\repD[m+1,Y]_{-,\defThn,\defS}}{i}{j}. \]
 \end{lemma}

 \begin{proof}
     When $0\leq i<j\leq m$, there is an isomorphism of categories 
     \[ \tndnec{\repD[m]\vee\repD[1]}{i}{j}\cong \tndnec{\repD[m+1]}{i}{j}. \]
     Hence, the desired result follows from \cref{pullbackwedge}. 
 \end{proof}

 \begin{prop} \label{isoonhomsijlowerthanm}
     For $0\leq i<j\leq m$, there is an isomorphism in $\Thnsset$
     \[ (\Pi_{m,Y})_{i,j}\colon \Hom_{\Ch L\repD[m+1,Y]}(i,j)\cong \Hom_{\Ch L\repD[m,Y]\amalg_{[0]} \Sigma Y}(i,j). \]
 \end{prop}

 \begin{proof}
By \cref{cor:computationshomC1ordered}, we have the following isomorphisms in $\Thnsset$
\begin{tz}
    \node[](1) {$\Hom_{\Ch L \repD[m+1,Y]}(i,j)\cong \diag(\colim_{T\in\tndnec{L\repD[m+1,Y]_{-,\star,\star}}{i}{j}}\Hom_{\CL T}(\alpha,\omega))$}; 
    \node[below of=1,xshift=.55cm](2) {$\Hom_{\Ch L\repD[m,Y]\amalg_{[0]}\Sigma Y}(i,j)\cong\diag(\colim_{T\in\tndnec{L(\repD[m,Y]\amalg_{F[0,Y]} \repD[1,Y])_{-,\star,\star}}{i}{j}}\Hom_{\CL T}(\alpha,\omega))$};
    \punctuation{2}{.};
    
    \draw[->] ($(1.south)-(4.5cm,0)$) to node[left,la]{$(\Pi_{m,Y})_{i,j}$} ($(2.north)-(5.05cm,0)$);
    \draw[->] ($(1.south)+(2.1cm,0)$) to node[right,la]{} ($(2.north)+(1.55cm,0)$);
    \end{tz}
    By \cref{necofijlowerthanm}, there is a canonical isomorphism in $\Thnsset$ between the right-hand terms and hence we get an isomorphism between the left-hand terms as desired.
 \end{proof}

 It remains to construct the action on the hom $\Thn$-spaces whose target is $m+1$. For this, we first provide a more combinatorial description of the category $\tndnec{\repD[m+1]}{i}{m+1}$ and its subcategory $\tndnec{\repD[m]\vee\repD[1]}{i}{m+1}$.

 \begin{defn}
     For $0\leq i\leq m$, we define the category $\cP air_{i,m+1}$ to be the poset such that
     \begin{itemize}[leftmargin=.6cm]
         \item its objects are pairs $(J,V)$ of subsets $\{i,m+1\}\subseteq J\subseteq V\subseteq \{i,\ldots,m+1\}$,
         \item there is a morphism $(J,V)\to (J',V')$ if and only if $V\subseteq V'$ and $J'\subseteq J$.
     \end{itemize}

     We write $\cP air_{i,m+1}^m$ for the full subcategory of $\cP air_{i,m+1}$ consisting of those pairs $(J,V)$ such that $m\in J\subseteq V$. 
 \end{defn}

\begin{prop} \label{isowithpairs}
For $0\leq i\leq m$, there are isomorphisms of categories
\[\tndnec{\repD[m+1]}{i}{m+1} \cong \cP air_{i,m+1} \quad \text{and}\quad \tndnec{\repD[m]\vee \repD[1]}{i}{m+1} \cong \cP air^m_{i,m+1}.\]
\end{prop}

 \begin{proof}
The left-hand isomorphism is induced by the functor $\tndnec{\repD[m+1]}{i}{m+1}\to \cP air_{i,m+1}$ which sends a totally non-degenerate necklace $T\hookrightarrow \repD[m+1]_{i,m+1}$ to the pair $(J_T,V_T)$ of subsets of $\repD[m+1]_0$ consisting of the joints and vertices of $T$, respectively. It is straightforward to check that this defines an isomorphism of categories and that it restricts to an isomorphism as depicted above right.
\end{proof}

We then obtain the following description of the hom $\Thn$-spaces from \cite[Corollary 3.10]{DuggerSpivakRigidification}.

 \begin{prop}
     For $0\leq i\leq m$ and $T\hookrightarrow \repD[m+1]_{i,m+1}$ a totally non-degenerate necklace, there is an isomorphism in $\sset$
     \[ \textstyle\Hom_{\CL T}(\alpha,\omega)\cong \prod_{V_T\setminus J_T}\repS[1]. \]
 \end{prop}
 
\begin{notation}
    For $0\leq i\leq m$, we define a functor 
    \[ (-)^{+m}\colon \cP air_{i,m+1}\to \cP air^m_{i,m+1}, \]
    which sends a pair $(J,V)$ to the pair $(J\cup \{m\},V\cup\{m\})$. This is well-defined.
\end{notation}

\begin{rmk}
    For $0\leq i\leq m$, the above functor induces under the isomorphisms from \cref{isowithpairs} a functor
    \[ (-)^{+m}\colon \tndnec{\repD[m+1]}{i}{m+1}\to \tndnec{\repD[m]\vee \repD[1]}{i}{m+1}. \]
    This functor adds the vertex $m$ to the set of joints and vertices of a totally non-degenerate necklace $T\hookrightarrow\repD[m+1]_{i,m+1}$.
\end{rmk}

\begin{rmk} \label{composetoid}
    Note that the composite of functors 
    \[ \cP air^m_{i,m+1}\hookrightarrow \cP air_{i,m+1}\xrightarrow{(-)^{+m}} \cP air^m_{i,m+1} \]
    is the identity, and hence the composite of functors
    \[ \tndnec{\repD[m]\vee \repD[1]}{i}{m+1}\hookrightarrow\tndnec{\repD[m+1]}{i}{m+1}\xrightarrow{(-)^{+m}}\tndnec{\repD[m]\vee \repD[1]}{i}{m+1}\]
    is also the identity.
\end{rmk}

\begin{rmk}
    As the below outer square of functors commute by \cref{composetoid}, we get a 
    functor 
    \[ (-)^{+m}\colon \tndnec{L\repD[m+1,Y]_{-,\defThn,\defS}}{i}{m+1}\to \tndnec{L(\repD[m,Y]\amalg_{\repD[0,Y]}\repD[1,Y])_{-,\defThn,\defS}}{i}{m+1}\] 
    given by the universal property of the pullback from \cref{pullbackwedge}. 
    \begin{tz}
    \node[](1') {$\tndnec{L\repD[m+1,Y]_{-,\defThn,\defS}}{i}{m+1}$};  
        \node[below of=1,yshift=-.5cm](3') {$\tndnec{\repD[m+1]}{i}{m+1}$}; 
        \node[below of=1',xshift=4.4cm,yshift=.5cm](1) {$\tndnec{L(\repD[m,Y]\amalg_{F[0,Y]}\repD[1,Y])_{-,\defThn,\defS}}{i}{m+1}$}; 
        \node[right of=1,xshift=5cm](2) {$\tndnec{L\repD[m+1,Y]_{-,\defThn,\defS}}{i}{m+1}$}; 
        \node[below of=1,yshift=-.5cm](3) {$\tndnec{\repD[m]\vee\repD[1]}{i}{m+1}$}; 
        \node[below of=2,yshift=-.5cm](4) {$\tndnec{\repD[m+1]}{i}{m+1}$}; 

        \draw[right hook->] (1) to (2); 
        \draw[->] (1) to node[left,la]{$(Q_{-,\defThn,\defS})_!$} (3);
        \draw[->] (1') to node[left,la]{$(Q_{-,\defThn,\defS})_!$} (3');
        \draw[->,dashed] (1') to node[above,la,xshift=25pt,yshift=-3pt]{$(-)^{+m}$} (1);
        \draw[->] (3') to node[below,la,xshift=-7pt]{$(-)^{+m}$} (3);
        \draw[->] (2) to node[right,la]{$(Q_{-,\defThn,\defS})_!$} (4); 
        \draw[right hook->] (3) to (4);
        \draw[->,bend left=10] (1') to node[above,la]{$\id$} (2);
        \pullback{1};
    \end{tz}
\end{rmk}

\begin{constr} \label{constr:varphimX}
    For $0\leq i\leq m$ and $T\hookrightarrow \repD[m+1]_{i,m+1}$ a totally non-degenerate necklace, the inclusion $V_T\cup \{m\}\setminus J_T\cup \{m\}\subseteq V_T\setminus J_T$ induces by pre-composition a map in $\sset$
    \begin{tz}
    \node[](1) {$\Hom_{\CL T}(\alpha,\omega)\cong \prod_{V_T\setminus J_T} \repS[1]$}; 
    \node[below of=1,xshift=.43cm](2) {$\Hom_{\CL (T^{+m})}(\alpha,\omega)\cong \prod_{V_T\cup \{m\}\setminus J_T\cup \{m\}} \repS[1]$};
    \punctuation{2}{.};
    
    \draw[->] ($(1.south)-(1.3cm,0)$) to node[left,la]{$\varphi_T$} ($(2.north)-(1.73cm,0)$);
    \draw[->] ($(1.south)+(1.4cm,0)$) to node[right,la]{} ($(2.north)+(.97cm,0)$);
    \end{tz}
    It is straightforward to check that this construction is natural in $T\in\tndnec{\repD[m+1]}{i}{m+1}$.

    For $\defThn\in \Thn$ and $\defS\geq 0$, we get an induced map between colimits in $\sset$ making the following diagrams commute
    \begin{tz}
        \node[](1) {$\Hom_{\CL T}(\alpha,\omega)$}; 
        \node[right of=1,xshift=6.2cm](2) {$\colim_{T\in\tndnec{L\repD[m+1,Y]_{-,\defThn,\defS}}{i}{m+1}}\Hom_{\CL T}(\alpha,\omega)$}; 
        \node[below of=1](3) {$\Hom_{\CL (T^{+m})}(\alpha,\omega)$}; 
        \node[below of=2](4) {$\colim_{T\in\tndnec{L(\repD[m,Y]\amalg_{F[0,Y]} \repD[1,Y])_{-,\defThn,\defS}}{i}{m+1}}\Hom_{\CL T}(\alpha,\omega)$};

        \draw[->](1) to node[left,la]{$\varphi_T$} (3); 
        \draw[->](1) to node[above,la]{$\gamma_T$} (2); 
        \draw[->](2) to node[right,la]{$(\varphi^i_{m,Y})_{\defThn,\defS}$} (4); 
        \draw[->] (3) to node[below,la]{$\gamma_{T^{+m}}$} (4);
    \end{tz}
    for all $T\in \tndnec{L\repD[m+1,Y]_{-,\defThn,\defS}}{i}{m+1}$, where $\gamma$ denote the colimit cones. Since the above defined map is natural in $\defThn\in \Thn$ and $\defS\geq 0$, this induces a map in $\Thnssset$
    \begin{tz} 
        \node[](2) {$\colim_{T\in\tndnec{L\repD[m+1,Y]_{-,\star,\star}}{i}{m+1}}\Hom_{\CL T}(\alpha,\omega)$};  
        \node[below of=2](4) {$\colim_{T\in\tndnec{L(\repD[m,Y]\amalg_{F[0,Y]} \repD[1,Y])_{-,\star,\star}}{i}{m+1}}\Hom_{\CL T}(\alpha,\omega)$};

        \draw[->](2) to node[right,la]{$\varphi_{m,Y}^i$} (4); 
    \end{tz}
\end{constr}

We are now ready to build the desired $\Thnsset$-enriched functor. 

\begin{constr}
    We construct a $\Thnsset$-enriched functor 
    \[ \Pi_{m,Y}\colon \Ch L\repD[m+1,Y]\to \Ch L\repD[m,Y]\amalg_{[0]} \Sigma Y \]
    between directed $\Thnsset$-enriched categories such that 
    \begin{itemize}[leftmargin=0.6cm]
        \item it is the identity on objects,
        \item for $0\leq i<j\leq m$, the map on hom $\Thn$-spaces is given by the isomorphism in $\Thnsset$ from \cref{isoonhomsijlowerthanm}  
        \[ (\Pi_{m,Y})_{i,j}\colon \Hom_{\Ch L\repD[m+1,Y]}(i,j)\cong \Hom_{\Ch L\repD[m,Y]\amalg_{[0]} \Sigma Y}(i,j), \]
        \item for $0\leq i\leq m$, the map on hom $\Thn$-spaces is given by the map in $\Thnsset$ obtained by taking the diagonal of the map in $\Thnssset$ from \cref{constr:varphimX}
    \end{itemize}
        \begin{tz}
    \node[](1) {$\Hom_{\Ch L\repD[m+1,Y]}(i,m+1)\cong \diag(\colim_{T\in\tndnec{L\repD[m+1,Y]_{-,\star,\star}}{i}{m+1}}\Hom_{\CL T}(\alpha,\omega))$}; 
    \node[below of=1,xshift=.55cm](2) {$\Hom_{\Ch L\repD[m,Y]\amalg_{[0]}\Sigma Y}(i,m+1)\cong\diag(\colim_{T\in\tndnec{L(\repD[m,Y]\amalg_{F[0,Y]} \repD[1,Y])_{-,\star,\star}}{i}{m+1}}\Hom_{\CL T}(\alpha,\omega))$};
    
    \draw[->] ($(1.south)-(4.5cm,0)$) to node[left,la]{$(\Pi_{m,Y})_{i,m+1}$} ($(2.north)-(5.05cm,0)$);
    \draw[->] ($(1.south)+(2.1cm,0)$) to node[right,la]{$\diag(\varphi_{m,Y}^i)$} ($(2.north)+(1.55cm,0)$);
    \end{tz}
    It is straightforward to check that this construction is compatible with composition.
\end{constr}

Moreover, we observe that, by construction, we have the following. 

\begin{lemma} \label{compisinclusion}
    The composite of $\Thnsset$-enriched functor 
    \[ \Ch L\repD[m,Y]\xrightarrow{\Ch L[d^{m+1},Y]} \Ch L\repD[m+1,Y]\xrightarrow{\Pi_{m,Y}} \Ch L\repD[m,Y]\amalg_{[0]} \Sigma Y \]
    coincides with the canonical inclusion of the coproduct. 
\end{lemma}

\subsection{The comparison map} \label{sec:comparisonmap}

Let us fix a $\Thnsset$-enriched category $\cC$ and an object $c\in \cC$. We want to construct a $\Thnsset$-enriched natural transformation in $[\cC^{\op},\Thnsset]$ 
\begin{equation}\label{nattransf} \textstyle (\varepsilon_\cC)_!\St_{\Nh\cC} \int_\cC^\Nh\Hom_\cC(-,c)\to \Hom_\cC(-,c). 
\end{equation}
Since the composite of left adjoints $(\varepsilon_\cC)_!\St_{\Nh\cC}$ preserves colimits, we have the following isomorphisms in $[\cC^{\op},\Thnsset]$
\begin{align*} \textstyle(\varepsilon_\cC)_!\St_{\Nh\cC} \int_\cC^\Nh \Hom_\cC(-,c)&\cong (\varepsilon_\cC)_!\St_{\Nh\cC}(\colim_{\repD[m,\defThn,\defS]\xrightarrow{\sigma} \int_\cC^\Nh \Hom_\cC(-,c)} \repD[m,\defThn,\defS]) \\
&\cong \colim_{\repD[m,\defThn,\defS]\xrightarrow{\sigma} \int_\cC^\Nh \Hom_\cC(-,c)}(\varepsilon_\cC)_!\St_{\Nh\cC}(\pi_{\Hom_\cC(-,c)}\sigma).
\end{align*}
Hence to construct \eqref{nattransf}, we need to define, for every map $\sigma\colon \repD[m,\defThn,\defS]\to \int_\cC^\Nh \Hom_\cC(-,c)$ in $\sThnsset$, a $\Thnsset$-enriched natural transformation in $[\cC^{\op},\Thnsset]$
\[ (\varepsilon_\cC)_!\St_{\Nh\cC}(\pi_{\Hom_\cC(-,c)}\sigma)\to \Hom_\cC(-,c)\]
natural in $\sigma\in \int_{\DThnS}(\int_\cC^\Nh\Hom_\cC(-,c))$.

\begin{prop} \label{boat}
    For every $m\geq 0$ and every $Y\in \Thnsset$, there is a one-to-one correspondence between maps in $\sThnsset$
    \[ \textstyle \sigma\colon \repD[m,Y]\to \int_\cC^\Nh \Hom_\cC(-,c)\]
    and $\Thnsset$-enriched functors 
    \[ F_\sigma\colon \Ch L\repD[m,Y]\amalg_{[0]}\Sigma Y\to \cC \]
    such that $F_\sigma(m+1)=c$. 
    
    Moreover, for every morphism $\mapDelta\colon [\ell]\to [m]$ in $\Delta$, every map $f\colon X\to Y$ in $\Thnsset$, and every map $\sigma\colon \repD[m,Y]\to \int_\cC^\Nh \Hom_\cC(-,c)$ in $\sThnsset$, the following boat commutes, 
    \begin{tz} 
        \node[](1') {$\Ch L\repD[\ell,X]$};
        \node[below of=1'](2') {$\Ch L\repD[m,Y]$};
        \node[right of=1',xshift=2.6cm](1'') {$\Ch L\repD[\ell,X]\amalg_{[0]} \Sigma X$};
        \node[below of=1''](2'') {$\Ch L\repD[m,Y]\amalg_{[0]} \Sigma Y$};
        \node[right of=1'',yshift=-.75cm,xshift=2.5cm](3) {$\cC$};

        \draw[->] ($(1''.east)$) to node[above,la]{$F_{\sigma[\mapDelta,f]}$} (3);
        \draw[->] ($(2''.east)$) to node[below,la]{$F_{\sigma}$} (3);
        \draw[->] (1') to (1'');
        \draw[->] (1') to node[left,la]{$\Ch L[\mapDelta,f]$} (2');
        \draw[->] (2') to (2'');
    \end{tz}
    where the horizontal maps are the canonical inclusions of the coproducts.
\end{prop}

\begin{proof}
    Composites in $\sThnsset$
    \[  \textstyle\repD[m,Y]\xrightarrow{\sigma} \int_\cC\Hom_\cC(-,c) \to \Nh\cC \]
    correspond by Yoneda to composites in $\Thnsset$
    \[ \textstyle Y\xrightarrow{\sigma} (\int_\cC\Hom_\cC(-,c))_m\to (\Nh\cC)_m, \]
    which by definition of $(\int_\cC\Hom_\cC(-,c))_m$ correspond to composites in $\sThnsset$ 
    \[ Y\xrightarrow{\sigma}(\Nh\cC)_m\times_{(\Nh\cC)_0}\times (\Nh\cC)_1\times_{(\Nh\cC)_0} \{c\}\to (\Nh\cC)_m. \]
    By the universal property of pullback and Yoneda, above composites in $\Thnsset$ correspond to composites in $\sThnsset$
    \[ F[m,Y]\to F[m,Y]\amalg_{[0]} F[1,Y]\xrightarrow{\sigma} \Nh\cC\]
    with $\sigma(m+1)=c$, and so by the adjunction $\Ch L\dashv I\Nh$ and \cref{Sh1issigma} to composites of $\Thnsset$-enriched functors 
    \[ \Ch L F[m,Y]\to \Ch L F[m,Y]\amalg_{[0]} \Sigma Y\xrightarrow{F_\sigma} \cC\]
    such that $F_\sigma(m+1)=c$. In particular, we can deduce from this the desired bijection between maps $\sigma$ and $\Thnsset$-enriched functors $F_\sigma$. Finally, the commutativity of the boat follows from the naturality of the adjunction $\Ch L\dashv I\Nh$.
\end{proof}

Let us now fix $m\geq 0$, a connected $\Thn$-space $Y$, and a map $\sigma\colon \repD[m,Y]\to \int_\cC^\Nh \Hom_\cC(-,c)$ in $\Thnsset$. In particular, one can take $Y=\repD[\defThn,\defS]$. 

\begin{notation} \label{not:tauf}
    We write $\tau$ for the composite of maps in $\sThnsset$ 
    \[ \textstyle\tau\colon \repD[m,Y]\xrightarrow{\sigma} \int_\cC^\Nh\Hom_\cC(-,c)\to \Nh\cC. \]
\end{notation}

We aim to construct a $\Thnsset$-enriched natural transformation in $[\cC^{\op}, \Thnsset]$
    \[ \varphi_\sigma\colon (\varepsilon_\cC)_!\St_{\Nh\cC}(\tau)\to \Hom_\cC(-,c). \]

    \begin{rmk} \label{stofid}
        By \cref{Stof[lX]appendix} in the case where $\mapDelta=\id_{[m]}$ and $f=\id_Y$, there is an isomorphism in $[\Ch L\repD[m,Y]^{\op},\Thnsset]$
        \[ \St_{L\repD[m,Y]}(\id_{\repD[m,Y]})\cong \Ch L[d^{m+1},Y]^* \Hom_{\Ch L\repD[m+1,Y]}(-,m+1). \]
    \end{rmk}
    
\begin{constr} \label{constr:varphisigma}
    Recall from \cref{boat} that the map $\sigma\colon \repD[m,Y]\to \int_\cC^\Nh \Hom_\cC(-,c)$ corresponds to a $\Thnsset$-enriched functor $F_\sigma\colon \Ch L\repD[m,Y]\amalg_{[0]} \Sigma Y\to \cC$ such that $F_\sigma(m+1)=c$. We further have a commutative diagram in $\Thncat$
    \begin{tz}
        \node[](1) {$\Ch L\repD[m,Y]$}; 
        \node[below of=1](2) {$\Ch L\repD[m+1,Y]$};
        \node[right of=2,xshift=3.1cm](2') {$\Ch L\repD[m,Y]\amalg_{[0]} \Sigma Y$}; 
        \node[right of=2',xshift=1.95cm](2'') {$\cC$}; 
        \node[above of=2''](1'') {$\Ch\Nh\cC$};
        \draw[->] (1) to node[left,la]{$\Ch L[d^{m+1},Y]$} (2); 
        \draw[->] (1) to node[above,la]{$\Ch\tau$} (1'');
        \draw[->] (1'') to node[right,la]{$\varepsilon_\cC$} (2'');
        \draw[->] (2) to node[below,la]{$\Pi_{m,Y}$} (2');
        \draw[->] (2') to node[below,la]{$F_\sigma$} (2'');
    \end{tz}
    Then the $\Thnsset$-enriched functor $F_\sigma\Pi_{m,Y}$ induces a $\Thnsset$-enriched natural transformation in $[\Ch L\repD[m+1,Y]^{\op}, \Thnsset]$
    \[ (F_\sigma\Pi_{m,Y})_{-,m+1}\colon \Hom_{\Ch L\repD[m+1,Y]}(-,m+1)\to \Pi_{m,Y}^*F_\sigma^*\Hom_\cC(-,c). \]
    By restricting along $\Ch L[d^{m+1},Y]$, we obtain a $\Thnsset$-enriched natural transformation in $[\Ch L\repD[m,Y]^{\op}, \Thnsset]$
    \[ \gamma_\sigma\coloneqq \Ch L[d^{m+1},Y]^*\Hom_{\Ch L\repD[m+1,Y]}(-,m+1)\to \Ch L[d^{m+1},Y]^*\Pi_{m,Y}^*F_\sigma^*\Hom_\cC(-,c), \]
    which by \cref{stofid} and the fact that the above diagram commutes amounts to a $\Thnsset$-enriched natural transformation in $[\Ch L\repD[m,Y]^{\op}, \Thnsset]$
    \[ \gamma_\sigma\colon \St_{L\repD[m,Y]}(\id_{\repD[m,Y]})\to (\Ch\tau)^*(\varepsilon_\cC)^*\Hom_\cC(-,c). \]
    By transposing along the adjunctions $(\Ch\tau)_!\dashv (\Ch\tau)^*$ and $(\varepsilon_\cC)_!\dashv (\varepsilon_\cC)^*$, we get a $\Thnsset$-enriched natural transformation in $[\cC^{\op}, \Thnsset]$
    \[ \varphi_\sigma\coloneqq (\varepsilon_\cC)_!\St_{\Nh\cC}(\tau)\cong (\varepsilon_\cC)_!(\Ch\tau)_!\St_{L\repD[m,Y]}(\id_{\repD[m,Y]})\to \Hom_\cC(-,c), \]
    where the isomorphism holds by \cref{basechangeStappendix} in the case where $f=\tau$. 
\end{constr}

To show that this construction is natural, we will make use of the following formal results. 

\begin{lemma} \label{unitisFonhoms}
    Let $F\colon \cC\to \cD$ be a $\Thnsset$-enriched functor and $c\in \cC$ be an object. The unit of the adjunction $F_!\dashv F^*$ evaluated at $\Hom_\cC(-,c)$ is the $\Thnsset$-enriched natural transformation in $[\cC^{\op},\Thnsset]$
    \[ F_{-,c}\colon \Hom_\cC(-,c)\to \Hom_\cD(F(-),Fc)=F^*\Hom_\cD(-,Fc)\cong F^*F_!\Hom_\cC(-,c). \]
\end{lemma}

\begin{proof}
    By the enriched Yoneda lemma, the $\Thnsset$-enriched natural transformation $F_{-,c}$ is determined by the fact that it sends $\id_c\in \Hom_\cC(c,c)$ to $\id_{Fc}\in \Hom_\cC(Fc,Fc)$. It then follows that the pair $(\Hom_\cD(-,Fc),F_{-,c})$ is initial in the category of pairs $(G,\eta)$ of a $\Thnsset$-enriched functor $G\colon \cC\to \cD$ and a $\Thnsset$-enriched natural transformation $\eta\colon \Hom_\cC(-,c)\Rightarrow F^* G$. To see this, note that the latter is again determined by the image of $\id_c\in \Hom_\cC(c,c)$ under $\eta$, namely $\eta(\id_c)\in G(Fc)$. This shows that $F_{-,c}$ is the unit of the adjunction $F_!\dashv F^*$.
\end{proof}

\begin{rmk} \label{rem:mates}
    Consider a commutative square in $\Thncat$ as below left.
    \begin{tz}
        \node[](1) {$\cA$}; 
        \node[below of=1](2) {$\cC$}; 
        \node[right of=1,xshift=.5cm](3) {$\cB$}; 
        \node[below of=3](4) {$\cD$}; 

        \draw[->] (1) to node[left,la]{$I$} (2); 
        \draw[->] (1) to node[above,la]{$F$} (3); 
        \draw[->] (2) to node[below,la]{$G$} (4); 
        \draw[->] (3) to node[right,la]{$J$} (4);

        \node[right of=3,xshift=2cm](1) {$[\cA^{\op},\Thnsset]$}; 
        \node[below of=1](2) {$[\cC^{\op},\Thnsset]$}; 
        \node[right of=1,xshift=2.8cm](3) {$[\cB^{\op},\Thnsset]$}; 
        \node[below of=3](4) {$[\cD^{\op},\Thnsset]$}; 

        \cell[][n][.5]{1}{4}{};
        \draw[->] (2) to node[left,la]{$I^*$} (1); 
        \draw[->] (1) to node[above,la]{$F_!$} (3); 
        \draw[->] (2) to node[below,la]{$G_!$} (4); 
        \draw[->] (4) to node[right,la]{$J^*$} (3);
    \end{tz} 
    Then the component at a $\Thnsset$-enriched functor $H\colon \cC^{\op}\to \Thnsset$ of the natural transformation $F_!I^*\to J^*G_!$ as in the above right square in $\cat$ is given by the $\Thnsset$-enriched natural transformation $F_!I^*(H)\to J^*G_!(H)$ in $[\cB^{\op},\Thnsset]$ corresponding under the adjunction $F_!\dashv F^*$ to the $\Thnsset$-enriched natural transformation in $[\cA^{\op},\Thnsset]$
    \[ I^*(H)\xrightarrow{I^*\eta_H} I^*G^*G_!(H)=F^*J^*G_!(H), \]
    where $\eta$ denotes the unit of the adjunction $G_!\dashv G^*$.
\end{rmk}

We now show that the construction $\varphi_\sigma$ from \cref{constr:varphisigma} is natural in $\sigma$. Let us fix a morphism $\mapDelta\colon [\ell]\to [m]$ in $\Delta$, a map $f\colon X\to Y$ in $\Thnsset$ between connected $\Thn$-spaces, and a map $\sigma\colon \repD[m,Y]\to \int_\cC^\Nh \Hom_\cC(-,c)$ in $\sThnsset$.

\begin{lemma} \label{firstcommute}
    The following diagram in $[\Ch L\repD[\ell,X]^{\op},\Thnsset]$ commutes. 
    \begin{tz}
        \node[](1) {$\Ch L[d^{\ell+1},X]^*\Hom_{\Ch L\repD[\ell+1,X]}(-,\ell+1)$}; 
        \node[below of=1](1') {$\Ch L[d^{\ell+1},X]^*\Ch L[\mapDelta+1,f]^*\Hom_{\Ch L\repD[m+1,Y]}(-,m+1)$};
        \node[below of=1',yshift=.5cm](2) {$\Ch L[\mapDelta,f]^*\Ch L[d^{m+1},Y]^*\Hom_{\Ch L\repD[m+1,Y]}(-,m+1)$}; 
        \node[right of=1,xshift=7cm](3) {$\Ch L[\mapDelta,f]^*(\Ch \tau)^*(\varepsilon_\cC)^*\Hom_\cC(-,c)$}; 
        \node[below of=3,yshift=-1cm](4) {$\Ch L[\mapDelta,f]^*(\Ch\tau)^*(\varepsilon_\cC)^*\Hom_\cC(-,c)$}; 
        \draw[->] (1) to node[above,la]{$\gamma_{\sigma [\mapDelta,f]}$} (3);
        \draw[->] (1) to node[left,la]{$\Ch L[d^{\ell+1},Y]^* \Ch L[\mapDelta+1,f]_{-,\ell+1}$} (1');
        \draw[d] (3) to (4);
        \draw[->] (2) to node[below,la]{$\Ch L[\mapDelta,f]^*\gamma_\sigma$} (4);
        \node at ($(1')!0.5!(2)$) {\rotatebox{270}{$\cong$}};
    \end{tz}
\end{lemma}

\begin{proof}
    Combining \cref{compisinclusion,boat}, the following diagram of $\Thnsset$-enriched functors commutes.
    \begin{tz}
        \node[](1) {$\Ch L\repD[\ell,X]$}; 
        \node[right of=1,xshift=2.5cm](1') {$\Ch L\repD[\ell+1,X]$};
        \node[below of=1](2) {$\Ch L\repD[m,Y]$};
        \node[below of=1'](2') {$\Ch L\repD[m+1,Y]$};
        \node[right of=1',xshift=3cm](1'') {$\Ch L\repD[\ell,X]\amalg_{[0]} \Sigma X$};
        \node[below of=1''](2'') {$\Ch L\repD[m,Y]\amalg_{[0]} \Sigma Y$};
        \node[right of=1'',yshift=-.75cm,xshift=2.5cm](3) {$\cC$};

        \draw[->] ($(1''.east)$) to node[above,la]{$F_{\sigma[\mapDelta,f]}$} (3);
        \draw[->] ($(2''.east)$) to node[below,la]{$F_{\sigma}$} (3);
        \draw[->] (1) to node[left,la]{$\Ch L[\mapDelta,f]$} (2); 
        \draw[->] (1) to node[above,la]{$\Ch L[d^{\ell+1},X]$} (1');
        \draw[->] (1') to node[above,la]{$\Pi_{\ell,X}$} (1'');
        \draw[->] (2) to node[below,la]{$\Ch L[d^{m+1},Y]$} (2');
        \draw[->] (2') to node[below,la]{$\Pi_{m,Y}$} (2'');
    \end{tz}
    By taking the induced $\Thnsset$-enriched natural transformations between hom $\Thn$-spaces with target $\ell+1$ and restricting along $\Ch L[d^{\ell+1},X]$, we get the desired result. 
\end{proof}

\begin{lemma} \label{Stadjunct}
    The $\Thnsset$-enriched natural transformation in $[\Ch L\repD[m,Y]^{\op},\Thnsset]$ 
    \[ \St_{L\repD[m,Y]}([\mapDelta,f])\colon \Ch L[\mapDelta,f]_!\St_{L \repD[\ell,X]}(\id_{\repD[\ell,Y]})\cong \St_{L \repD[m,Y]}([\mapDelta,f])\to \St_{L\repD[m,Y]}(\id_{\repD[m,Y]}) \]
    corresponds under the adjunction $\Ch L[\mapDelta,f]_!\dashv \Ch L[\mapDelta,f]^*$ to the $\Thnsset$-enriched natural transformation in $[\Ch L\repD[\ell,X]^{\op},\Thnsset]$
    \[ \Ch L[d^{\ell+1},X]^* \Ch L[\mapDelta+1,f]_{-,\ell+1}\colon \St_{LF[\ell,X]}(\id_{F[\ell,X]})\to \Ch L[\mapDelta,f]^*\St_{LF[m,Y]}(\id_{F[m,Y]})\]
\end{lemma}

\begin{proof}
Consider the following diagram in $\Thncat$. 
\begin{tz}
\node[](1) {$\Ch L\repD[\ell,X]$}; 
\node[below of=1](2) {$\Ch L\repD[\ell+1,X]$}; 
\node[right of=1,xshift=2cm](3) {$\Ch L\repD[m,Y]$}; 
\node[below of=3](4) {$\Ch \Cone$}; 
\node[below right of=4,xshift=1.3cm](5) {$\Ch L\repD[m+1,Y]$}; 
\pushout{4};

\draw[->] (1) to node[above,la]{$\Ch L[\mapDelta,f]$} (3);
\draw[->] (1) to node[left,la]{$\Ch L[d^{\ell+1},X]$} (2);
\draw[->] (3) to node[right,la]{$\Ch \iota_f$} (4);
\draw[->] (2) to node[below,la]{$J$} (4);
\draw[->,bend left] (3) to node[right,la]{$\Ch L[d^{m+1},Y]$} (5);
\draw[->,bend right=15] (2) to node[below,la,yshift=-3pt]{$\Ch L[\mapDelta+1,f]$} (5);
\draw[->,dashed] (4) to node[above,la,xshift=2pt]{$G$} (5);
\end{tz}
Then we can deduce from \cref{Stof[lX]appendix} that the $\Thnsset$-enriched natural transformation in $[\Ch L\repD[m,Y]^{\op},\Thnsset]$ 
    \[ \St_{L\repD[m,Y]}([\mapDelta,f])\colon  \St_{L \repD[m,Y]}([\mapDelta,f])\to \St_{L\repD[m,Y]}(\id_{\repD[m,Y]}) \]
    is given by the $\Thnsset$-enriched natural transformation 
    \[ (\Ch \iota_f)^*G_{-,m+1}\colon (\Ch \iota_f)^*\Hom_{\Cone}(-,m+1)\to \Ch L[d^{m+1},Y]^*\Hom_{\Ch L\repD[m+1,Y]}(-,m+1). \]
    Consider the following diagram of natural transformations. 
    \begin{tz}
        \node[](1) {$[\Ch L\repD[\ell,X]^{\op},\Thnsset]$}; 
        \node[below of=1](2) {$[\Ch L\repD[\ell+1,X]^{\op},\Thnsset]$}; 
        \node[right of=1,xshift=5cm](3) {$[\Ch L\repD[m,Y]^{\op},\Thnsset]$}; 
        \node[below of=3](4) {$[\Ch \Cone^{\op},\Thnsset]$}; 
        \node[below of=2](5) {$[\Ch L\repD[\ell+1,X]^{\op},\Thnsset]$}; 
        \node[below of=4](6) {$[\Ch L\repD[m+1,Y]^{\op},\Thnsset]$};

         \cell[la,above][n][.5]{1}{4}{};

          \cell[la,above][n][.5]{2}{6}{};

        \draw[->] (2) to node[left,la]{$\Ch L\repD[d^{\ell+1},X]^*$} (1); 
        \draw[d] (2) to (5);
        \draw[->] (1) to node[above,la]{$\Ch L[\mapDelta,f]_!$} (3);
        \draw[->] (5) to node[below,la]{$\Ch L[\mapDelta+1,f]_!$} (6);
        \draw[->] (2) to node[below,la]{$J_!$} (4); 
        \draw[->] (4) to node[right,la]{$(\Ch \iota_f)^*$} (3);
        \draw[->] (6) to node[right,la]{$G^*$} (4);
        \draw[->,bend right=50] ($(6.north)+(2cm,0)$) to node[right,la]{$\Ch L[d^{m+1},Y]^*$} ($(3.south)+(2cm,0)$);
    \end{tz}
    By \cref{lem:Stvssigma}, the component of the upper natural transformation at the representable object $\Hom_{\Ch L\repD[\ell+1,X]}(-,\ell+1)\in [\Ch L\repD[\ell+1,X]^{\op},\Thnsset]$ is given by the isomorphism in $[\Ch L\repD[m,Y]^{\op},\Thnsset]$
    \[  \Ch L[\mapDelta,f]_! \St_{L\repD[\ell,X]} (\id_{\repD[\ell,X]}) \cong \St_{L\repD[m,Y]}([\mapDelta,f]). \]
    Then, combining \cref{rem:mates,unitisFonhoms}, the component of the lower natural transformation at $\Hom_{\Ch L\repD[\ell+1,X]}(-,\ell+1)$ is given by the $\Thnsset$-enriched natural transformation in $[\Ch \Cone ^{\op},\Thnsset]$ 
    \[ G_{-,m+1}\colon \Hom_{\Cone}(-,m+1)\to G^*\Hom_{\Ch L\repD[m+1,Y]}(-,m+1). \]
    Hence, by the above, their composite is given by the $\Thnsset$-enriched natural transformation in $[\Ch L\repD[m,Y]^{\op},\Thnsset]$
    \[ \Ch L[\mapDelta,f]_!\St_{L \repD[\ell,X]}(\id_{\repD[\ell,Y]})\cong \St_{L \repD[m,Y]}([\mapDelta,f])\xrightarrow{\St_{L\repD[m,Y]}([\mapDelta,f])} \St_{L\repD[m,Y]}(\id_{\repD[m,Y]}). \]
    On the other hand, combining \cref{rem:mates,unitisFonhoms}, the component of the composite of the two natural transformations at $\Hom_{\Ch L\repD[\ell+1,X]}(-,\ell+1)$ is given by the $\Thnsset$-enriched natural transformation in $[\Ch L\repD[m,Y]^{\op},\Thnsset]$ corresponding under the adjunction $\Ch L[\mapDelta,f]_!\dashv \Ch L[\mapDelta,f]^*$ to the restriction along $\Ch L[d^{\ell+1},X]$ of the $\Thnsset$-enriched natural transformation in $[\Ch L\repD[\ell+1,X]^{\op},\Thnsset]$
    \[ \Ch L[\mapDelta+1,f]_{-,\ell+1}\colon \Hom_{\Ch L\repD[\ell+1,X]}(-,\ell+1)\to \Ch L[\mapDelta+1,f]^*\Hom_{\Ch L\repD[m+1,Y]}(-,m+1). \]
    Using that $\Ch L[d^{\ell+1},X]^*\Ch L[\mapDelta+1,f]^*=\Ch L[\mapDelta,f]^*\Ch L[d^{m+1},Y]^*$ and \cref{Stof[lX]appendix}, it is isomorphic to the $\Thnsset$-enriched natural transformation in $[\Ch L\repD[\ell,X]^{\op},\Thnsset]$ 
    \[ \Ch L[d^{\ell+1},X]^* \Ch L[\mapDelta+1,f]_{-,\ell+1}\colon \St_{LF[\ell,X]}(\id_{F[\ell,X]})\to \Ch L[\mapDelta,f]^*\St_{LF[m,Y]}(\id_{F[m,Y]}). \]
    This gives the desired result.
\end{proof}

\begin{lemma} \label{secondcommute}
    The following diagram in $[\Ch L\repD[m,Y]^{\op},\Thnsset]$ commutes. 
    \begin{tz}
        \node[](1) {$\Ch L[\mapDelta,f]_!\St_{LF[\ell,X]}(\id_{F[\ell,X]})$}; 
        \node[below of=1,yshift=.5cm](1') {$\St_{LF[m,Y]}([\mapDelta,f])$};
        \node[below of=1'](2) {$\St_{LF[m,Y]}(\id_{F[m,Y]})$}; 
        \node[right of=1,xshift=6.5cm](3) {$\Ch L[\mapDelta,f]_!\Ch L[\mapDelta,f]^*(\Ch \tau)^*(\varepsilon_\cC)^*\Hom_\cC(-,c)$}; 
        \node[below of=3,yshift=-1cm](4) {$(\Ch\tau)^*(\varepsilon_\cC)^*\Hom_\cC(-,c)$}; 
        \draw[->] (1) to node[above,la]{$\Ch L[\mapDelta,f]_! \gamma_{\sigma [\mapDelta,f]}$} (3);
        \draw[->] (1') to node[left,la]{$\St_{L\repD[m,Y]}([\mapDelta,f])$} (2);
        \draw[->] (3) to node[right,la]{$\epsilon$} (4);
        \draw[->] (2) to node[below,la]{$\gamma_\sigma$} (4);
        \node at ($(1)!0.5!(1')$) {\rotatebox{270}{$\cong$}};
    \end{tz}
\end{lemma}

\begin{proof}
    This is obtained by transposing the commutative diagram from \cref{firstcommute} along the adjunction $\Ch L[\mapDelta,f]_!\dashv \Ch L[\mapDelta,f]^*$ using \cref{stofid,Stadjunct}.
\end{proof}

\begin{prop} \label{thirdcommute}
    The following diagram in $[\cC^{\op},\Thnsset]$ commutes. 
    \begin{tz}
        \node[](1) {$(\varepsilon_\cC)_!\St_{\Nh\cC}(\tau [\mapDelta,f])$}; 
        \node[below of=1](2) {$(\varepsilon_\cC)_!\St_{\Nh\cC}(\tau)$}; 
        \node[right of=1,xshift=2.7cm](3) {$\Hom_\cC(-,c)$};
        \node[below of=3](4) {$\Hom_\cC(-,c)$};
        \draw[->](1) to node[left,la]{$(\varepsilon_\cC)_!\St_{\Nh\cC}([\mapDelta,g])$} (2);
        \draw[->](1) to node[above,la]{$\varphi_{\sigma[\mapDelta,f]}$} (3);
        \draw[->](2) to node[below,la]{$\varphi_{\sigma}$} (4);
        \draw[d] (3) to (4);
    \end{tz}
\end{prop}

\begin{proof}
    This is obtained by transposing the diagram from \cref{secondcommute} along the adjunctions $(\Ch\tau)_!\dashv (\Ch\tau)^*$ using \cref{basechangeStappendix} and the definitions of $\varphi_\sigma$ and $\varphi_{\sigma[\mapDelta,f]}$.
\end{proof}

\begin{constr}
    By \cref{thirdcommute}, the $\Thnsset$-enriched natural transformations in $[\cC^{\op},\Thnsset]$ 
    \[ \varphi_\sigma\colon (\varepsilon_\cC)_!\St_{\Nh\cC}(\tau)\to \Hom_\cC(-,c), \]
    with $\sigma\colon \repD[m,\defThn,\defS]\to \int_\cC^\Nh\Hom_\cC(-,c)$ in $\sThnsset$, form a natural cone over $\Hom_\cC(-,c)$. Hence we get an induced $\Thnsset$-enriched natural transformation in $[\cC^{\op},\Thnsset]$ 
    \[ \textstyle\varphi_c\colon (\varepsilon_\cC)_!\St_{\Nh\cC}\int_\cC^\Nh\Hom_\cC(-,c)\cong \colim_{\repD[m,\defThn,\defS]\xrightarrow{\sigma} \int_\cC^\Nh\Hom_\cC(-,c)} (\varepsilon_\cC)_!\St_{\Nh\cC}(\tau)\to \Hom_\cC(-,c). \]
\end{constr}

We first show that $\varphi_c$ provides a retract of the map $(\varepsilon_\cC)_! \St_{\Nh\cC}(\id_c)$.

\begin{prop}
    The following composite in $[\cC^{\op},\Thnsset]$ is the identity. 
    \[ \textstyle\Hom_\cC(-,c)\cong (\varepsilon_\cC)_! \St_{\Nh\cC}(c)\xrightarrow{(\varepsilon_\cC)_! \St_{\Nh\cC}(\id_c)} (\varepsilon_\cC)_! \St_{\Nh\cC}\int_\cC\Hom_\cC(-,c)\xrightarrow{\varphi_c} \Hom_\cC(-,c)\]
\end{prop}

\begin{proof}
    First note that the above composite is simply given by the component of the colimit at $\id_c\colon \repD[0]\to \int_\cC \Hom_\cC(-,c)$ of $\varphi_c$, i.e., it is the $\Thnsset$-enriched natural transformation in $[\cC^{\op},\Thnsset]$ 
    \[ \varphi_{\id_c}\colon (\varepsilon_\cC)_! \St_{\Nh\cC}(c)\cong (\varepsilon_\cC)_! c^* \St_{F[0]}(\id_{F[0]})\to \Hom_\cC(-,c) \]
    from \cref{constr:varphisigma} taking $\sigma=\id_c\colon \repD[0]\to \int_\cC \Hom_\cC(-,c)$ and $\tau=c\colon \repD[0]\to \Nh\cC$. By noticing that $F_{\id_c}\colon [1]\to \cC$ is the $\Thnsset$-enriched that picks the identity morphism $\id_c$ in~$\cC$, we see that, by construction, the $\Thnsset$-enriched natural transformation $\varphi_{\id_c}$ corresponds under the adjunctions $c_!\dashv c^*$ and $(\varepsilon_\cC)_!\dashv (\varepsilon_\cC)^*$ to the map in $[\Ch\repD[0],\Thnsset]\cong \Thnsset$ 
    \[ \gamma_{\id_c}\colon \St_{\repD[0]}(\id_{\repD[0]})\cong \repS[0]\to c^*(\varepsilon_\cC)^* \Hom_\cC(-,c)\cong \Hom_\cC(c,c), \]
    which picks the identity $\id_c$. Using the Yoneda Lemma and \cref{unitisFonhoms} in the case where $F=F_{\id_c}$, we get that the component at $c$ of the $\Thnsset$-enriched natural transformation
    \[ \varphi_{\id_c}\colon \Hom_\cC(-,c)\cong (\varepsilon_\cC)_! \St_{\Nh\cC}(c)\to \Hom_\cC(-,c) \]
    takes $\id_c$ to $\id_c$ and so must be the identity. 
\end{proof}

We now show that the construction $\varphi_c$ is natural in $c\in \cC$. Let us fix a morphism $g\colon c\to d$ in the underlying category of the $\Thnsset$-enriched category $\cC$. Let us further fix $m\geq 0$, a connected $\Thn$-space $Y\in \Thnsset$, and a map $\sigma\colon \repD[m,Y]\to \int_\cC^\Nh\Hom_\cC(-,c)$ in $\sThnsset$.

\begin{prop} \label{firstnatc}
    The following diagram in $[\Ch L\repD[m,Y]^{\op},\Thnsset]$ commutes. 
    \begin{tz}
        \node[](1) {$\St_{L\repD[m,Y]}(\id_{\repD[m,Y]})$}; 
        \node[right of=1,xshift=3.8cm](2) {$(\varepsilon_\cC)^*(\Ch\tau)^*\Hom_\cC(-,c)$}; 
        \node[below of=2](3) {$(\varepsilon_\cC)^*(\Ch\tau)^*\Hom_\cC(-,d)$};

        \draw[->] (1) to node[above,la]{$\gamma_\sigma$} (2);
        \draw[->] (1) to node[below,la]{$\gamma_{\int_\cC^\Nh g_*\sigma}$} (3);
        \draw[->] (2) to node[right,la]{$(\varepsilon_\cC)^* (\Ch\tau)^* g_*$} (3);
    \end{tz}
\end{prop}

\begin{proof}
By \cref{stofid} and the definitions of $\gamma_\sigma$ and $\gamma_{\int_\cC^\Nh g_*\sigma}$, this amounts to show that the following diagram in $[\Ch L\repD[m,Y]^{\op},\Thnsset]$ commutes. 
    \begin{tz}
        \node[](1) {$\Ch L[d^{m+1},Y]^*\Hom_{\Ch L\repD[m+1,Y]}(-,m+1)$}; 
        \node[right of=1,xshift=7.7cm](2) {$\Ch L[d^{m+1},Y]^*\Pi_{m,Y}^* F_\sigma^* \Hom_\cC(-,c)$}; 
        \node[below of=2](3) {$\Ch L[d^{m+1},Y]^*\Pi_{m,Y}^* F_{\int_\cC^\Nh g_*\sigma}^* \Hom_\cC(-,d)$};

        \draw[->] (1) to node[above,la,yshift=5pt]{$\Ch L[d^{m+1},Y]^*(\Pi_{m,Y}F_\sigma)_{-,m+1}$} (2);
        \draw[->] (1) to node[below,la,xshift=-40pt]{$\Ch L[d^{m+1},Y]^*(\Pi_{m,Y}F_{\int_\cC^\Nh g_*\sigma})_{-,m+1}$} (3);
        \draw[->] (2) to node[right,la]{$\Ch L[d^{m+1},Y]^*\Pi_{m,Y}^* F_\sigma^* g_*$} (3);
    \end{tz}
As $\Ch L[d^{m+1},Y]\Pi_{m,Y} F_\sigma=(\Ch\tau) \varepsilon_\cC=\Ch L[d^{m+1},Y]\Pi_{m,Y} F_{\int_\cC^\Nh g_*\sigma}$, to check the commutativity of the above diagram, it is enough to verify that it commutes at the object $m\in \Ch L\repD[m,Y]$. For this, note that the map in $\sThnsset$
\[ \textstyle Y\cong \repD[0,Y]\xrightarrow{[\langle m\rangle,Y]} \repD[m,Y] \xrightarrow{\sigma}\int_\cC^\Nh\Hom_\cC(-,c)\xrightarrow{\int_\cC^\Nh g_*} \int_\cC^\Nh\Hom_\cC(-,d) \]
is given by Yoneda and the connectedness of $Y$ by a map in $\Thnsset$
\[ Y\xrightarrow{f}\Hom_\cC(e,c)\xrightarrow{g_*} \Hom_\cC(e,d), \]
where $f$ is the map corresponding to $Y\cong \repD[0,Y]\xrightarrow{[\langle m\rangle,Y]} \repD[m,Y] \xrightarrow{\sigma}\int_\cC^\Nh\Hom_\cC(-,c)$. Hence the commutativity of the above diagram follows from the commutativity of the following diagram.
\begin{center} \hfill
\begin{tikzpicture}[scale=1]
        \node[](1) {$\Hom_{\Ch L\repD[m+1,Y]}(m,m+1)=Y$}; 
        \node[right of=1,xshift=3.7cm](2) {$\Hom_\cC(e,c)$}; 
        \node[below of=2](3) {$\Hom_\cC(e,d)$};

        \draw[->] (1) to node[above,la]{$f$} (2);
        \draw[->] (1) to node[below,la,xshift=-5pt]{$g_*f$} (3);
        \draw[->] (2) to node[right,la]{$g_*$} (3);
    \end{tikzpicture} \qedhere
    \end{center}
\end{proof}

\begin{prop} \label{secondnatc}
    The following diagram in $[\cC^{\op},\Thnsset]$ commutes. 
    \begin{tz}
        \node[](1) {$(\varepsilon_\cC)_!\St_{\Nh\cC}(\tau)$}; 
        \node[right of=1,xshift=2.3cm](2) {$\Hom_\cC(-,c)$}; 
        \node[below of=2](3) {$\Hom_\cC(-,d)$};

        \draw[->] (1) to node[above,la]{$\varphi_\sigma$} (2);
        \draw[->] (1) to node[below,la]{$\varphi_{\int_\cC^\Nh g_*\sigma}$} (3);
        \draw[->] (2) to node[right,la]{$g_*$} (3);
    \end{tz}
\end{prop}

\begin{proof}
    This is obtained by transposing the diagram from \cref{firstnatc} along the adjunctions $(\Ch\tau)_!\dashv (\Ch\tau)^*$ and $(\varepsilon_\cC)_!\dashv(\varepsilon_\cC)^*$ using \cref{basechangeStappendix} and the definitions of $\varphi_\sigma$ and $\varphi_{\int_\cC^\Nh g_*\sigma}$.
\end{proof}

\begin{prop}
    The $\Thnsset$-enriched natural transformations in $[\cC^{\op},\Thnsset]$ 
    \[ \textstyle\varphi_c\colon (\varepsilon_\cC)_!\St_{\Nh\cC}\int_\cC^\Nh\Hom_\cC(-,c)\to \Hom_\cC(-,c) \] 
    are natural in objects $c\in \cC$. 
\end{prop}

\begin{proof}
    We need to show that, given a morphism $g\colon c\to d$ in $\cC$, then the following diagram commutes.
    \begin{tz}
        \node[](1) {$(\varepsilon_\cC)_!\St_{\Nh\cC}\int_\cC^\Nh\Hom_\cC(-,c)$}; 
        \node[below of=1](2) {$(\varepsilon_\cC)_!\St_{\Nh\cC}\int_\cC^\Nh\Hom_\cC(-,d)$}; 
        \node[right of=1,xshift=3.3cm](3) {$\Hom_\cC(-,c)$};
        \node[below of=3](4) {$\Hom_\cC(-,d)$};
        \draw[->](1) to node[left,la]{$(\varepsilon_\cC)_!\St_{\Nh\cC}\int_\cC^\Nh g_*$} (2);
        \draw[->](1) to node[above,la]{$\varphi_c$} (3);
        \draw[->](2) to node[below,la]{$\varphi_d$} (4);
        \draw[->] (3) to node[right,la]{$g_*$} (4);
    \end{tz}
    As $(\varepsilon_\cC)_!\St_{\Nh\cC}\int_\cC^\Nh\Hom_\cC(-,c)\cong \colim_{\repD[m,\defThn,\defS]\xrightarrow{\sigma} \int_\cC^\Nh\Hom_\cC(-,c)} (\varepsilon_\cC)_!\St_{\Nh\cC}(\tau)$, by the universal property of colimits, it is enough to check that this diagram commute at a specific component $\sigma\colon \repD[m,\defThn,\defS]\to \int^\Nh_\cC\Hom_\cC(-,c)$ of the colimit. However, this is precisely \cref{secondnatc}.
\end{proof}

\bibliographystyle{amsalpha}
\bibliography{refstraightening}

\providecommand{\bysame}{\leavevmode\hbox to3em{\hrulefill}\thinspace}
\providecommand{\MR}{\relax\ifhmode\unskip\space\fi MR }
\providecommand{\MRhref}[2]{%
  \href{http://www.ams.org/mathscinet-getitem?mr=#1}{#2}
}
\providecommand{\href}[2]{#2}
\begin{thebibliography}{AGS22b}

\bibitem[AGS22a]{abellangarciastern20222cartfibi}
Fernando Abell\'{a}n~Garc\'{\i}a and Walker~H. Stern, \emph{2-{C}artesian
  fibrations {I}: {A} model for {$\infty $}-bicategories fibred in {$\infty
  $}-bicategories}, Appl. Categ. Structures \textbf{30} (2022), no.~6,
  1341--1392. \MR{4519636}

\bibitem[AGS22b]{abellangarciastern20222cartfibii}
\bysame, \emph{$2$-cartesian fibrations {II}: A {G}rothendieck construction for
  $\infty$-bicategories},
  \href{https://arxiv.org/abs/2201.09589}{arXiv:2201.09589}, 2022.

\bibitem[BdB18]{debrito2018leftfibration}
Pedro Boavida~de Brito, \emph{Segal objects and the {G}rothendieck
  construction}, An alpine bouquet of algebraic topology, Contemp. Math., vol.
  708, Amer. Math. Soc., Providence, RI, 2018, pp.~19--44. \MR{3807750}

\bibitem[Ber07]{BergerIterated}
Clemens Berger, \emph{Iterated wreath product of the simplex category and
  iterated loop spaces}, Adv. Math. \textbf{213} (2007), no.~1, 230--270.
  \MR{2331244}

\bibitem[BR13]{br1}
Julia~E. Bergner and Charles Rezk, \emph{Comparison of models for
  {$(\infty,n)$}-categories, {I}}, Geom. Topol. \textbf{17} (2013), no.~4,
  2163--2202. \MR{3109865}

\bibitem[BR20]{br2}
\bysame, \emph{Comparison of models for {$(\infty, n)$}-categories, {II}}, J.
  Topol. \textbf{13} (2020), no.~4, 1554--1581. \MR{4186138}

\bibitem[CP86]{CordierPorterHomotopyCoherent}
Jean-Marc Cordier and Timothy Porter, \emph{Vogt's theorem on categories of
  homotopy coherent diagrams}, Math. Proc. Cambridge Philos. Soc. \textbf{100}
  (1986), no.~1, 65--90. \MR{838654}

\bibitem[DS11]{DuggerSpivakRigidification}
Daniel Dugger and David~I. Spivak, \emph{Rigidification of quasi-categories},
  Algebr. Geom. Topol. \textbf{11} (2011), no.~1, 225--261. \MR{2764042}

\bibitem[Dug01]{Dugger}
Daniel Dugger, \emph{Combinatorial model categories have presentations}, Adv.
  Math. \textbf{164} (2001), no.~1, 177--201. \MR{1870516}

\bibitem[Dug08]{duggerhocolimnotes}
\bysame, \emph{A primer on homotopy colimits}, preprint available at
  \url{https://pages.uoregon.edu/ddugger/hocolim.pdf}, 2008.

\bibitem[GH15]{gepnerhaugseng2015enriched}
David Gepner and Rune Haugseng, \emph{Enriched {$\infty$}-categories via
  non-symmetric {$\infty$}-operads}, Adv. Math. \textbf{279} (2015), 575--716.
  \MR{3345192}

\bibitem[GHL20]{gagnaharpazlanari2020inftytwolimits}
Andrea Gagna, Yonatan Harpaz, and Edoardo Lanari, \emph{Fibrations and lax
  limits of $(\infty,2)$-categories},
  \href{https://arxiv.org/abs/2012.04537}{arXiv:2012.04537}, 2020.

\bibitem[GHL22]{gagnaharpazlanari2022equivalence}
\bysame, \emph{On the equivalence of all models for {$(\infty,2)$}-categories},
  J. Lond. Math. Soc. (2) \textbf{106} (2022), no.~3, 1920--1982. \MR{4498545}

\bibitem[Hau15]{haugseng2015rectenrichedinftycat}
Rune Haugseng, \emph{Rectification of enriched {$\infty$}-categories}, Algebr.
  Geom. Topol. \textbf{15} (2015), no.~4, 1931--1982. \MR{3402334}

\bibitem[HHR21]{hhr2021straightening}
Fabian Hebestreit, Gijs Heuts, and Jaco Ruit, \emph{A short proof of the
  straightening theorem},
  \href{https://arxiv.org/abs/2111.00069}{arXiv:2111.00069}, 2021.

\bibitem[Hir03]{Hirschhorn}
Philip~S. Hirschhorn, \emph{Model categories and their localizations},
  Mathematical Surveys and Monographs, vol.~99, American Mathematical Society,
  Providence, RI, 2003. \MR{1944041}

\bibitem[Hir15]{HirschhornOvercategories}
\bysame, \emph{{Overcategories and undercategories of model categories}},
  \href{https://arxiv.org/abs/1507.01624}{arXiv:1507.01624}, 2015.

\bibitem[HM15]{heutsmoerdijk2015leftfibrationi}
Gijs Heuts and Ieke Moerdijk, \emph{Left fibrations and homotopy colimits},
  Math. Z. \textbf{279} (2015), no.~3-4, 723--744. \MR{3318247}

\bibitem[HM16]{heutsmoerdijk2016leftfibrationii}
\bysame, \emph{Left fibrations and homotopy colimits ii},
  \href{https://arxiv.org/abs/1602.01274v1}{arXiv:1602.01274v1}, 2016.

\bibitem[Joy97]{JoyalDisks}
Andr{\'e} Joyal, \emph{{Disks, Duality and $\Theta$-categories}}, preprint
  available at \url{https://ncatlab.org/nlab/files/JoyalThetaCategories.pdf},
  1997.

\bibitem[Kel05]{Kelly}
G.~M. Kelly, \emph{Basic concepts of enriched category theory}, Repr. Theory
  Appl. Categ. (2005), no.~10, vi+137, Reprint of the 1982 original [Cambridge
  Univ. Press, Cambridge; MR0651714]. \MR{2177301}

\bibitem[LR20]{LRfib}
Fosco Loregian and Emily Riehl, \emph{Categorical notions of fibration}, Expo.
  Math. \textbf{38} (2020), no.~4, 496--514. \MR{4177953}

\bibitem[Lur09a]{htt}
Jacob Lurie, \emph{Higher topos theory}, Annals of Mathematics Studies, vol.
  170, Princeton University Press, Princeton, NJ, 2009. \MR{2522659}

\bibitem[Lur09b]{lurieGoodwillie}
\bysame, \emph{\textnormal{$(\infty, 2)$-categories and the Goodwillie Calculus
  I}}, \href{https://arxiv.org/abs/0905.0462}{arXiv:0905.0462}, 2009.

\bibitem[ML98]{MacLane}
Saunders Mac~Lane, \emph{Categories for the working mathematician}, second ed.,
  Graduate Texts in Mathematics, vol.~5, Springer-Verlag, New York, 1998.
  \MR{1712872 (2001j:18001)}

\bibitem[MOR22]{MOR}
Lyne Moser, Viktoriya Ozornova, and Martina Rovelli, \emph{Model independence
  of $(\infty,2)$-categorical nerves},
  \href{https://arxiv.org/abs/2206.00660}{arXiv:2206.00660}, 2022.

\bibitem[Mos19]{Moserinj}
Lyne Moser, \emph{Injective and projective model structures on enriched diagram
  categories}, Homology Homotopy Appl. \textbf{21} (2019), no.~2, 279--300.
  \MR{3923784}

\bibitem[MRR22]{MRR1}
Lyne Moser, Nima Rasekh, and Martina Rovelli, \emph{A homotopy coherent nerve
  for $(\infty,n)$-categories},
  \href{https://arxiv.org/abs/2208.02745}{arXiv:2208.02745}, 2022.

\bibitem[Nui21]{nuiten20222straightening}
Joost Nuiten, \emph{On straightening for segal spaces},
  \href{https://arxiv.org/abs/2108.11431}{arXiv:2108.11431}, 2021.

\bibitem[Ras21]{RasekhD}
Nima Rasekh, \emph{Yoneda lemma for $\mathcal{D}$-simplicial spaces},
  \href{https://arxiv.org/abs/2108.06168}{arXiv:2108.06168}, 2021.

\bibitem[Ras23]{rasekh2023left}
\bysame, \emph{Yoneda lemma for simplicial spaces}, Appl. Categ. Structures
  \textbf{31} (2023), no.~27.

\bibitem[Rez10]{rezkTheta}
Charles Rezk, \emph{A {C}artesian presentation of weak {$n$}-categories}, Geom.
  Topol. \textbf{14} (2010), no.~1, 521--571. \MR{2578310}

\bibitem[Rie14]{RiehlCHT}
Emily Riehl, \emph{Categorical homotopy theory}, New Mathematical Monographs,
  vol.~24, Cambridge University Press, Cambridge, 2014. \MR{3221774}

\bibitem[RV20]{RiehlVerityNCoh}
Emily Riehl and Dominic Verity, \emph{Recognizing quasi-categorical limits and
  colimits in homotopy coherent nerves}, Appl. Categ. Structures \textbf{28}
  (2020), no.~4, 669--716. \MR{4114996}

\bibitem[RV22]{riehlverity2022elements}
\bysame, \emph{Elements of {$\infty$}-category theory}, Cambridge Studies in
  Advanced Mathematics, vol. 194, Cambridge University Press, Cambridge, 2022.
  \MR{4354541}

\end{thebibliography}

\end{document}